\documentclass[aos]{imsart}

\RequirePackage{graphicx}
\RequirePackage{pdfpages}
\RequirePackage{xcolor}		
\RequirePackage{array}
\RequirePackage{booktabs}		
\RequirePackage{color}
\RequirePackage[numbers,sort&compress]{natbib}
\RequirePackage[colorlinks,citecolor=blue,urlcolor=blue]{hyperref}
\RequirePackage{graphicx}

\usepackage{float}
\usepackage{comment}

\RequirePackage{hyperref}

\RequirePackage[OT1]{fontenc}

\RequirePackage[numbers]{natbib}

\RequirePackage{amsthm,amsmath,amssymb,amsfonts,dsfont}
\RequirePackage{bm,bbm,mathtools}
\RequirePackage{rotating}
\RequirePackage{lscape}
\RequirePackage{here}

\startlocaldefs
\numberwithin{equation}{section}
\theoremstyle{plain}
\newtheorem{Thrm}{Theorem}[section]
\newtheorem{Assu}{Assumption}
\newtheorem{DefThrm}[Thrm]{Definition and Theorem}
\newtheorem{Lem}[Thrm]{Lemma}
\newtheorem{Cor}[Thrm]{Corollary}
\newtheorem{Prop}[Thrm]{Proposition}
\newtheorem{Ex}{Example}

\newtheorem{Rem}{Remark}
\newcommand{\hide}[1]{} 

\newcommand{\N}{\mathbb{N}}
\newcommand{\Z}{\mathbb{Z}}
\newcommand{\R}{\mathbb{R}}
\newcommand{\C}{\mathbb{C}}
\newcommand{\FF}{\mathbb{F}}

\newcommand{\E}{\mathbb{E}}
\renewcommand{\P}{\mathbb{P}}

\newcommand{\A}{\mathcal{A}}
\newcommand{\B}{\mathcal{B}}
\newcommand{\F}{\mathcal{F}}
\newcommand{\G}{\mathcal{G}}
\newcommand{\PP}{\mathcal{P}}
\newcommand{\LL}{\mathcal{L}}
\newcommand{\NN}{\mathcal{N}}
\newcommand{\HH}{\mathcal{H}}
\newcommand{\KK}{\mathcal{K}}
\newcommand{\WW}{\mathcal{W}}
\newcommand{\ZZ}{\mathcal{Z}}
\newcommand{\CN}{\mathcal{CN}}

\newcommand{\cum}{\mathrm{cum}}
\newcommand{\Cov}{\mathrm{Cov}}

\newcommand{\Span}{\mathrm{span}}
\newcommand{\Id}{\mathrm{Id}}

\newcommand{\trace}{\mathrm{trace}}
\newcommand{\tsr}{\otimes}

\renewcommand{\phi}{\varphi}
\renewcommand{\F}{\mathcal F}
\renewcommand{\epsilon}{\varepsilon}

\newcommand{\<}{\langle}
\renewcommand{\>}{\rangle}

\newcommand{\wt}{\widetilde}
\newcommand{\wh}{\widehat}
\newcommand{\ol}{\overline}

\newcommand{\1}{\mathds{1}}
\renewcommand{\epsilon}{\varepsilon}

\endlocaldefs

\allowdisplaybreaks

\begin{document}


\begin{frontmatter}
\title{Frequency Domain Bootstrap for Functional Time Series}
%
\runtitle{Frequency Domain Bootstrap}
\begin{aug}

\author[A]{\fnms{Daniel}~\snm{Rademacher} 
\ead[label=e1]{d.rademacher@tugraz.at}}
\author[B]{\fnms{Jens-Peter}~\snm{Kreiss}\ead[label=e2]{j.kreiss@tu-bs.de}\orcid{0000-0001-6446-4558}}
\author[C]{\fnms{Efstathios}~\snm{Paparoditis}\ead[label=e3]{s.paparoditis@academyofcyprus.cy}\orcid{0000-0003-1958-781X}}
\address[A]{Institute for Statistics, TU Graz
\printead[presep={ ,\ }]{e1}}

\address[B]{Department of Mathematics, TU Braunschweig
\printead[presep={,\ }]{e2}}

\address[C]{Cyprus Academy of Sciences, Letters,  and Arts
\printead[presep={,\ }]{e3}}

%
%
%
%
%

\runauthor{Rademacher, Krei{\ss}, Paparoditis }
\end{aug}

\begin{abstract}
A frequency domain bootstrap procedure for   functional time series     is proposed  and   applied to  the class of  spectral mean operators. The procedure   works by  first generating independent 
pseudo periodogram 
operators  across the positive Fourier frequencies  using an estimator of the spectral density operator involved.  Functional   replicates of  the  
spectral mean operators of interest are then generated.
Through    an additive, projection-based decomposition of the  bootstrapped  spectral mean operator,  its  leading   $m$-dimensional  part  is   properly  complemented  to also capture  the  relevant  fourth order  characteristics  of the process.  The complementation is achieved  by means of  a resampling procedure based on convolved periodogram operators of subsamples.    The resulting  bootstrap  spectral mean operator  consistently estimates the entire second order  as well as the $m$-dimensional  fourth order  structure  of the distribution of  spectral mean operators.  By  allowing  
for the decomposition parameter 
$m$  to  increase  to infinity  as the sample size increases to infinity,
 consistency  in  estimating
   the entire  fourth order structure   of the process   also is  achieved.
The asymptotic theory   developed  investigates properties of the procedure  for fixed and for increasing $m$ and establishes  validity  of the frequency domain bootstrap proposal   under  rather weak  conditions on the  underlying functional process class.
\end{abstract}

\begin{keyword}[class=MSC]
\kwdgroup[type=primary]{\kwd{62F40}
\kwd{62M15}}
\kwdgroup[type=secondary]{\kwd{62M10}}
\end{keyword}

\begin{keyword} 
 \kwd{Spectral Density Operator}  
 \kwd{Spectral Mean Operator} 
 \kwd{Periodogram Operator}   
 \kwd{Frequency Domain Bootstrap} 
 \kwd{Convolved Periodogram Operators}   
\end{keyword}

%
\end{frontmatter}

\section{Introduction}

Functional time series  analysis has attracted  considerable  interest  during the last decades.  This branch of statistics  concerns  the development of statistical tools for 
analyzing  observations which can be considered as   realizations of  temporal dependent 
random elements   taking   values in   general,  separable Hilbert spaces. Several real-life data sets can be handled  using methods  developed 
in this area of mathematical  statistics,   like for instance,  atmospheric, environmental, finance, economic and  biological data.

Although a variety of asymptotic results for many  inference problems  has been successfully derived, the infinite dimensionality  of the random elements involved   
makes  the  implementation of asymptotically justifiable   distributional 
results  difficult in practice.  To address  these  difficulties,   different 
bootstrap or resampling methods  for functional time series have been  considered. 
This concerns  either the development of new bootstrap procedures  or the adaption  to the functional setting of methods proposed  for the multivariate  time series context.
Analogously  to the    finite dimensional case, the bootstrap can be implemented in the  time  or in the frequency domain. Concerning time domain methods, in \cite{Politis1994}  a  stationary bootstrap procedure  for the mean of a functional time series has been  proposed.   A moving block and a tapered block bootstrap procedure   to the $K$-sample mean problem  has been applied in \cite{PilavakisEtAl2019},  while  \cite{PilavakisEtAl2020} used  the moving block bootstrap to  the problem of testing  equality of autocovariance operators of two functional time series. 
In \cite{Shapirov2016} the authors addressed a  change point problem   for functional time series  using the  nonoverlapping block bootstrap,  the latter method has been also   considered  in \cite{Dehling2015}. Properties  of a residual-based  bootstrap for first order functional autoregressive processes have been investigated in \cite{ZhuPolitis2017} and  
\cite{FrankeNuar2016},  while \cite{Paparoditis2018} introduced  a functional sieve bootstrap procedure.

 In contrast to   the above time domain methods, frequency domain bootstrap procedures  for functional time series  have attracted less attention.  Among  the few exceptions,    is  the paper    
 \cite{LeuchtEtAl2023}, where  a functional version of  the multiplicative  periodogram bootstrap (see 
 \cite{HurvichZeger1987}, \cite{FrankeHardle1992} and   \cite{DahlhausJanas1996}), has been  applied in the context of testing the  equality of  spectral density  operators between  two functional time series.   However, and as in the finite dimensional case,   the  functional version of the  multiplicative periodogram bootstrap has a  rather  limited range of validity. In  fact this procedure successfully estimates the distribution of  a  statistic, provided this distribution only depends  on  second order characteristics of the underlying  functional  process.  Note that even in the finite dimensional case of  multivariate processes,  
  the class of  statistics for which this holds  true 
  is much more narrower compared to the corresponding class of statistics  for  univariate time series. Essentially, nonparametric estimators, like  kernel estimators of the 
  spectral density operator  and statistics based  thereof   are the main class of inference problems  to which   the functional version of the multiplicative  bootstrap  can successfully  be applied.  
 
The aforementioned    limitations are 
  due to the fact that the functional  multiplicative periodogram bootstrap  generates pseudo periodogram operators,  which are independent across  Fourier frequencies.  This independence makes the procedure unable to imitate the weak dependence (covariance structure) of periodogram operators across frequencies, where the latter   is affected by    the fourth order characteristics of the underlying process. 
   This  drawback in capturing the weak dependence of the periodogram operators manifests itself in a     failure of this bootstrap procedure  for    important classes of statistics. 
 For instance, the class of so-called  spectral mean operators,  which refers to  a  variety  of statistics that can be expressed as integrated periodogram operators,  is an important example of  statistics  the distribution of which  depends,  not only on  second,  but also on   fourth order characteristics of 
 the underlying process; see 
\cite{RademacherEtAl2024}.
Note that in the finite dimensional (univariate and multivariate) case,   frequency domain procedures have been proposed which   are able to also imitate the relevant  fourth order 
characteristics of the underlying process and therefore,  to extend   the validity of  frequency domain bootstrap methods   to larger  classes  of  statistics; see \cite{MeyerEtAl2020}, \cite{MeyerPaparoditis2023} and \cite{YuEtAl2023}. 
However,  and due to  the infinite dimensionality of the  underlying space,  the  bootstrap methods proposed in this context can not  directly be transferred to  the   functional time series set up. 

To elaborate,    in addition to the  different and much more involved  technical challenges that have to be addressed  in the functional setting, the aforementioned, finite dimensional    bootstrap procedures,  
essentially work by applying a proper variance  correction to the distribution  obtained from the multiplicative periodogram bootstrap  part.  This variance (matrix)  correction aims to   capture the terms  which   are due to fourth order properties of the underlying process.
 Central  in this procedure  is a  standardization of the distribution of the  bootstrap  statistic obtained in the multiplicative bootstrap step, prior  
 to the application of  the  variance correction via a rescaling of the standardized bootstrap statistic. This   standardization implicitly requires the existence of a bounded,  inverse spectral density  (matrix).
However,  due to the trace class property of the  considered spectral density operators in the  functional  setting,  these operators do not  possess   in general,   a bounded inverse operator.   Therefore, this approach is not feasible in the functional set up.
One way out will be  to    work  from the beginning  with a finite dimensional  projection, respectively, approximation of the operators involved.
Although this  will  allow for    the existence of    a bounded inverse operator  that  enables   a proper standardization, 
it   essentially  reduces   the  functional inference problem     to a  finite dimensional  multivariate one, thereby neglecting   the inherent infinite dimensional  structure of the  random elements  considered.

In  this paper a  frequency domain bootstrap procedure  for functional time series   is proposed  which avoids a finite dimensional reduction of the inference problem at hand and achieves consistency in estimating the quantities of interest under rather  weak conditions on the  underlying 
 processes class.  Instead of  projecting the statistic of interest   and working on a finite dimensional subspace,  our procedure retains      infinite dimensionality  by applying  a   projection-based  decomposition  of the statistic   considered
  in order  to consistently estimate  the covariance,  the relation operator and   the distribution of interest.  To elaborate,  
our starting point is the generation of  independent,  random pseudo elements of finite Fourier transforms which  are obtained 
by using the limiting complex Gaussian distribution of their sample counterparts. The unknown spectral density operator involved  is replaced by a   uniformly consistent estimator with respect to the trace norm.  
The generated   pseudo Fourier transforms lead to pseudo replicates of the periodogram operator and consequently  to  bootstrapped  
spectral mean operators, which  successfully  imitate the entire second order dependence structure of the underlying process. 

To complement the bootstrap spectral mean  operators for the missing fourth order  terms, the aforementioned  decomposition of the bootstrapped spectral mean operator is used. 
In particular, and
using the subspace spanned by the first $m$ orthonormal eigenfunctions 
corresponding to the $m$ largest eigenvalues of the lag-zero autocovariance operator,
we  decompose
the  bootstrapped spectral mean operator  in two additive  parts. The first  and leading  part  is obtained  by   projecting  the bootstrapped spectral mean operator  on the aforementioned subspace  while  the second  part  consists of   the  remaining random elements.
A second, complementary resampling procedure based on convolved periodogram operators of subsamples is then introduced, the aim of  which is to enrich   the leading  part  of 
this   decomposition for the  missing fourth order terms.
Specifically  and, 
based on  an analogue   decomposition  of the pseudo spectral mean operators obtained via the   convolved subsampling procedure, 
a    consistent estimation  of   the  missing  fourth order terms  of the covariance and  relation operator of the leading part  is achieved. 
Adding to  the  corrected main part of the bootstrapped  spectral mean operator  the  bootstrapped  second part,  which properly    imitates   the  corresponding  remaining   random elements,  leads  finally 
 to a    fully functional, frequency domain 
 bootstrap proposal which is used to  estimate  the distribution of spectral mean operators. 
To summarize, the described    bootstrap proposal has the following desirable   features:
\begin{enumerate}
\item[(i)]  As the original statistic, the generated bootstrap version of the spectral mean statistic  obtained in the final step of the procedure  is fully functional, that is,  a finite dimensional reduction  of the  inference problem considered is avoided.
\item[(ii)]  The projection-based decomposition 
of the spectral mean operator obtained from  the first part of the procedure,  is solely made in order  to correct the leading  $m$-dimensional part of the decomposed operator    for the missing fourth order terms. 
\item[(iii)]   Independently on  the choice of the dimension $m$ of this  decomposition,   the procedure  successfully mimics the entire, infinite dimensional  second order structure of the functional process that  affects the distribution of spectral mean operators. 
 \item[(iv)]  For any  $m$, the procedure also  correctly imitates the  fourth order terms  affecting  the distribution of spectral mean operators  which are  related  
 to   the leading, $m$-dimensional part of the decomposed
  spectral mean operator. 
 \item[(v)]  The approximation  errors  made  for any fixed $m$
  are   quantified. In particular,  it is shown how  these errors are affected by    
     the  closeness of the $m$-rank approximation  of   the   fourth order 
 spectral density operators  to its infinite dimensional population counterpart. 
 \item[(vi)]  By allowing   for the  dimension   $m$  to increase   to infinity as the  sample size $n$  increases to infinity,  
 the  approximation errors  vanish and  the procedure   successfully captures    also the  entire,  infinite dimensional  fourth order structure of the underlying functional process which 
 affects the distribution of spectral mean operators. 
\end{enumerate}

The   paper is organized as follows:  Section 2  describes the set up considered, introduces the notation used and summarizes results for the limiting behavior of spectral mean operators which are important  for  our subsequent  investigations. Section 3 is devoted to the frequency domain bootstrap procedure proposed. The  different components  of the procedure are described  in Section 3.1 while  the resulting  bootstrap algorithm is precisely stated in Section 3.2. Section 4  deals with   theoretical  investigations   under  both scenarios: the case of a  fixed $m$ and  the case where $m$   
increases  to infinity as the sample size $n$ increases to infinity.   It is shown  that  under rather  mild  conditions on the underlying process,  the procedure  consistently  estimates  the covariance and  the relation operator as well as the entire distribution of the class of spectral mean  operators.  Section 5 discusses  the issue  of how to choose the  bootstrap  tuning parameters   in practice 
   and  presents some numerical results. Technical proofs as well as  some  auxiliary lemmas  are deferred  to the Supplementary File.


\section{Statistical Inference for Functional Spectral Mean Operators} \label{sec2}
Let  $ \HH=\HH_{\mathbb R} + {\rm i} \HH_{\mathbb R}$ denote a  complexified separable  Hilbert space  and consider  a  stationary stochastic process ${\mathcal X}:= \{X_t,t\in\Z\}$, where  for each $t\in\Z$, the random elements $X_t$ take values in  the  real part  $\HH_{\mathbb R}$ of $\HH$ and satisfy $\E\|X_t\|^p < \infty$ for some integer $p\geq 4$. Here, $ \|\cdot\|$ denotes the norm induced by the inner product $ \langle \cdot, \cdot \rangle$  of $ {\mathcal H}$. The Lebesgue-Bochner space of all such random elements is denoted  by $L^p(\HH)$.   Since $\E\|X_t\| <\infty$, the expectation $ \E X_t =:\mu \in {\mathcal H}$ is well-defined (as a Bochner integral) and characterized
by the property $\<\E X_t,f\>=\E \<X_t,f\>$ for all $f\in \HH$. Moreover, $\E\|X\|^2 <\infty$ implies the existence of $ \Cov(X_t,X_s):= \E[(X_t-\mu)\otimes (X_s-\mu)]$, the cross-covariance operator 
of $ X_t$ and $ X_s$, where  the tensor operator $\otimes:\HH \times \HH\to N(\HH)$ is defined via $ f\otimes g(u):=\langle u,g\rangle f$ for $f, g, u\in{\mathcal H}$. ${\mathcal L}(\HH) $ denotes the space of bounded linear operators,  $ N({\mathcal H})$ 
the space of nuclear (or trace class) operators and $ HS({\mathcal H})$ the space  of Hilbert-Schmidt operators  from $\HH$ to $\HH$.  We denote by $ \|\cdot\|_{\mathcal L}$, $ \|\cdot \|_{HS}$ and $ \|\cdot \|_N$ the operator, the Hilbert-Schmidt and the nuclear norm, respectively,  avoiding in some  cases to explicitly  mention the particular spaces to which these norms  refer to.  For an operator $L$,   $ L^\ast$  denotes    the adjoint operator, that is the unique operator satisfying $\langle L(x),y\rangle =\langle x, L^\ast(y)\rangle$,   for all $x,y\in\HH$;  $\overline{L}$ the   conjugate operator defined by $ \overline{L}(x):= \overline{L(\overline{x})}$ for all $x\in\HH$  and $ L^\top$  the transposed operator $ L^\top := \big(\overline{L}\big)^\ast$; see \cite{RademacherEtAl2024} for  more details.  Furthermore,  if $L$ is a compact operator,  $ \sigma(L)$  stays  for its  discrete spectrum, while $ {\rm Re}(L)=(L+L^\ast)/2$ and ${\rm Im}(L)=(L-L^\ast)/(2i)$ denote the  real and imaginary part  of $L$, respectively.

We   introduce the  following   $2\times 2$ operator block  matrices,    which  are  useful in the sequel.  Let  $ C$, $ R \in HS(\HH)$  and denote by $ {\mathcal M}_{2\times 2}(HS(\HH))$ the space of all $2\times 2$ block matrices the entries of which are elements of $HS(\HH)$. More specifically,   consider the mapping 
$ \Sigma: HS(\HH)\oplus HS(\HH) \rightarrow {\mathcal M}_{2\times2}(HS(\HH))$ defined   by
\begin{equation} \label{eq.Sigma-Matrix}
 \Sigma(C,R) := \frac{1}{2}\left(\begin{array}{cc} {\rm Re}\{ C + R\}  & {\rm Im}\{-C+ R \} \\
	{\rm Im}\{C+R \} & {\rm Re}\{ C- R\} \end{array}\right).
\end{equation} 
We write $\Sigma_{i,j}$ for the $(i,j)$-th entry of $\Sigma$.  Note that  $ {\mathcal M}_{2\times2}(HS(\HH))$ generated in this way is a vector space over the field $\R$,  and   equipped with the  inner product $ \langle \Sigma^{(1)}, \Sigma^{(2)}\rangle _{HS}= \sum_{i,j=1}^2 \langle \Sigma^{(1)}_{i,j}, \Sigma^{(2)}_{i,j}\rangle_{HS}$,  forms  a Hilbert space.
Direct   calculations show  that  $ \|\Sigma(C,R)\|^2_{HS} = 2 \|C\|_{HS}^2 + 2\|R\|^2_{HS}$,  which implies that    for sequences $ (C_n,n\in\N)$ and $ (R_n,n\in\N)$ in $ HS(\HH)$,  
 $ \|\Sigma(C_n,R_n)-\Sigma(C,R)\|_{HS} \rightarrow 0$ if and only if $ \|C_n-C\|_{HS}\rightarrow 0$ and $ \|R_n-R\|_{HS} \rightarrow 0$, as $n\rightarrow \infty$. Furthermore,  if $C$ is the  covariance operator,  $ C=\E(Y\otimes Y)$  and $ R$ the relation operator,  $R=\E(Y\otimes \overline{Y})$, of a centered  random element $Y \in   \HH$,  then $ \Sigma(C,R)$ is  nonnegative definite since it is the covariance operator of the two dimensional random  element  $ ({\rm Re}(Y), {\rm Im}(Y))^\top$.

Recall that by stationarity $ \Cov(X_t,X_s)$ only depends on the difference $t-s$ and we, therefore, use the short hand notation $ \Gamma_h:= \Cov(X_h,X_0)$, $h\in \Z$,
for the autocovariance operators of ${\mathcal X}$.
Accordingly, for $ h_1,h_2,h_3\in \Z$, let $ \Gamma_{h_1,h_2,h_3} := \cum (X_{h_1}, X_{h_2}, X_{h_3},X_0) \in N(HS(\HH))$ denote the fourth order autocumulant operator of $\mathcal{X}$, where
\begin{align}\label{eq.ForthOrderCum}
	\begin{split}
		&\cum (X_{h_1}, X_{h_2}, X_{h_3},X_0)\\
		&:=\E[(X_{h_1}\otimes X_{h_2})\otimes(X_{h_3}\otimes X_0)]   
		- \E[X_{h_1}\otimes X_{h_2}]\otimes \E[X_{h_3}\otimes X_0]\\
		&-\E[X_{h_1}\otimes X_{h_3}]\otimes_{op} \E[X_{h_2}\otimes X_0]
		- \E[X_{h_1}\otimes X_{0}]\otimes^\top_{op} \E[X_{h_2}\otimes X_{h_3}].
	\end{split}
\end{align}
In the above expression 
and  for linear operators $L_j :  {\mathcal H}\rightarrow {\mathcal H}$,  $j=1,2,3$, the following definitions  of the operator tensor product and of the transposed operator  tensor product are used:     $ L_1\otimes_{op}L_2(L_3) :=L_1L_3 L_2^\ast $ and $ L_1\otimes^\top_{op}L_2(L_3) := L_1L_3^\top L_2^\top$.  We refer to the Supplement File of  \cite{RademacherEtAl2024} for properties of these operator tensor products. 

Assuming that   $\sum_{h\in\Z}\|\Gamma_h\|_N$ is finite,  ${\mathcal X}$ possesses a bounded and continuous (with respect to $\lambda$) spectral density operator ${\mathcal F}_{\cdot}\colon [\pi,\pi] \to N(\HH)$ defined by
\begin{equation} \label{eq.SpecOper}
	{\mathcal F}_\lambda := \frac{1}{2\pi} \sum_{h\in\Z}\Gamma_h e^{-i\lambda h};
\end{equation}
see \cite{PanaretosTavakoli2013} and \cite{CerovekiHoermann2015}.
Furthermore, for the fourth order cumulant operators we assume 
$$\sum_{h_1,h_2,h_3\in\Z} \| \Gamma_{h_1,h_2,h_3}\|_N < \infty. $$
This implies the  existence of a continuous  and bounded 
fourth order spectral density operator ${\mathcal F}_{\cdot,\cdot,\cdot}\colon [\pi,\pi]^3 \to N(HS)$ given by
\begin{equation} \label{eq.FourtOper}
	{\mathcal F}_{\lambda_1,\lambda_2,\lambda_3}:=\frac{1}{(2\pi)^3}\sum_{h_1,h_2,h_3\in \Z} \Gamma_{h_1,h_2,h_3}e^{-i(\lambda_1h_1+\lambda_2h_2 +\lambda_3h_3)},
\end{equation}
where  $\lambda_1,\lambda_2,\lambda_3 \in [-\pi,\pi]$; see  \cite{RademacherEtAl2024}.

Our interest  focuses on estimators of spectral mean operators, that is, on characteristics  of the process ${\mathcal X}$, which can be expressed as Bochner integrals of the form 
\begin{equation} \label{eq.SpecMeans}
	M(W,\F) := \int_{-\pi}^{\pi} W(\lambda) \F_\lambda d\lambda \in N(\HH),
\end{equation}
for some (sufficiently regular) integrable function $ W:[-\pi,\pi] \rightarrow \C$.  The following are some examples.\\

\noindent {\bf Example 1:}  \ For $ W(\lambda) = {\bf 1}_{(0, \lambda_0]}(\lambda) $, $\lambda_0\in(0,\pi]$,
\[    M({\bf 1}_{(0, \lambda_0]}, \F)=\int_0^{\lambda_0} \F_\lambda d\lambda,\]
is the spectral distribution operator   of ${\mathcal X}$.\\

\noindent {\bf Example 2:}  \ For $ W(\lambda)=e^{-i\lambda h}$, $h \in \Z$,  
\[ M(e^{-i\lambda h} ,\F)=  \int_{-\pi}^{\pi} e^{-i h \lambda} \F_{\lambda} d\lambda = \Gamma_h,\]
is the lag $h$ autocovariance operator of  ${\mathcal X}$.\\

\noindent {\bf Example 3:} \ For $ W(\lambda)=b^{-1} K((\lambda-\lambda_0)/b)$, $\lambda_0\in(0,\pi)$,  where $ K$ is a probability density on $\R$ and $  b>0$ a (fixed) bandwidth, 
\[ M( b^{-1} K((\cdot-\lambda_0)/b),\F)=  \frac{1}{b}\int_{-\pi}^{\pi} K((\lambda-\lambda_0)/b)\F_{\lambda} d\lambda, \]
is a smoothed version of the spectral density operator  $ \F_\lambda$ around frequency $\lambda=\lambda_0$.\\

Estimators of spectral means $ M(W,\F)$ can be obtained by replacing the unknown spectral density operator $\F_\lambda$ by the periodogram operator
\begin{equation} \label{eq.PerOper}
	P_{n,\lambda} := J_{n}(\lambda)\otimes J_n(\lambda), \ \ \lambda \in[-\pi,\pi],
\end{equation}
where $J_{n}(\lambda):=(2\pi n)^{-1/2}\sum_{t=1}^n X_t e^{-i  t \lambda}$ denotes the discrete Fourier transform of  
$X_1, X_2, $ $ \ldots, X_n$.    
An estimator of $ M(W,\F)$ is then given by $ M(W,P_{n})$, i.e.,  by replacing $ \F_\lambda$ in (\ref{eq.SpecMeans}) by $P_{n,\lambda}$. Since the periodogram operator is commonly calculated at the Fourier frequencies  $\lambda_{j,n}=2\pi j/n$ with 
$ j\in  {\mathcal F}(n)=\{ j\in\Z \,  |\, j=-\lfloor (n-1)/2\rfloor, \ldots, \lfloor n/2 \rfloor \}$, an easier to compute, discretized version of $ M(W,P_{n})$ is given by 
\begin{equation} \label{eq.Est-SpecMean}
	M_n(W,P_n) :=  \frac{2\pi}{n}\sum_{j \in {\mathcal G}(n)} W(\lambda_{j,n}) P_{n,\lambda_{j,n}}.
\end{equation}
Here $ {\mathcal G}(n) :=  \{j \in \Z\, | \, 1\leq |j| \leq N \}$, where  $ N:=\lfloor n/2 \rfloor $. 
Properties of the periodogram operator $ P_{n,\lambda}$  as well as of the estimator $ M_n(W,P_{n})$ have been investigated in \cite{RademacherEtAl2024}.
In particular, it has been shown in the afore cited paper, that under mild assumptions on the underlying functional process class, the following weak convergence result holds true 
in $HS(\HH)$  as $n\rightarrow \infty$: 
\begin{equation} \label{eq.SpecMeansNormal}
	L_n:=\sqrt{n}\big( M_n(W,P_n)-M_n(W,\F)\big) \stackrel{\mathcal D}{\longrightarrow}  {\mathcal Z}\sim {\mathcal C}{\mathcal N}(0, \Psi ,  \Upsilon),
\end{equation}
where $ M_n(W,\F):=2\pi n^{-1/2}\sum_{j \in {\mathcal G}(n)} W(\lambda_{j,n}) \F_{\lambda_{j,n}}$ and the covariance operator 
$\Psi ={\rm Cov}({\mathcal Z}):=\E({\mathcal Z}\otimes {\mathcal Z})$ as well as the relation operator $ \Upsilon={\rm Rel}({\mathcal Z}):=\E({\mathcal Z}\otimes \overline{{\mathcal Z}})$,  of the limiting Gaussian Hilbert-Schmidt operator $ {\mathcal Z}$, consist of  two terms, i.e.,
$\Psi=\Psi_1+\Psi_2$, respectively, 
$ \Upsilon= \Upsilon_1+\Upsilon_2$ .  These terms are given by  
\begin{align*}
	\Psi_1
	&= 2\pi \Big\{ \int_{\pi}^\pi W(\lambda)\overline{W}(\lambda) \F_\lambda \otimes_{op} \F_\lambda d\lambda +  \int_{\pi}^\pi W(\lambda)\overline{W}(-\lambda) \F_\lambda \otimes_{op}^\top \F_{-\lambda} d\lambda\Big\},\\
	\Psi_2 
	&= 2\pi \int_{\pi}^\pi \int_{-\pi}^\pi W(\lambda_1)\overline{W}(\lambda_2) \F_{\lambda_1,-\lambda_1,-\lambda_2}  d\lambda_1 d\lambda_2,
\end{align*}
and
\begin{align*}
	\Upsilon_1 
	&= 2\pi \Big\{ \int_{\pi}^\pi W(\lambda)W(-\lambda) \F_\lambda \otimes_{op} \F_\lambda d\lambda +  \int_{\pi}^\pi W(\lambda)W(\lambda) \F_\lambda \otimes_{op}^\top 
	\F_{-\lambda} d\lambda\Big\},\\ 
	\Upsilon_2 
	&=  2\pi \int_{\pi}^\pi \int_{-\pi}^\pi W(\lambda_1)W(-\lambda_2) \F_{\lambda_1,-\lambda_1,-\lambda_2}  d\lambda_1 d\lambda_2,
\end{align*}
where all integrals involved are  Bochner integrals .
Observe that $\Psi_1$ and $ \Upsilon_1$ only incorporate the second order dynamics of $\mathcal{X}$ in form of 
$\F_\lambda$, while $\Psi_2$ and $\Upsilon_2$ exclusively take into account the fourth order dynamics in form of $ \F_{\lambda_1,\lambda_2,\lambda_3}$. 
Note  that $ {\rm Cov}({\rm Re}({\mathcal Z}))= {\rm Re}\{\Psi + \Upsilon\}/2$, $ {\rm Cov}({\rm Re}({\mathcal Z}),{\rm Im}({\mathcal Z}))= {\rm Im}\{-\Psi + \Upsilon\}/2$, $ {\rm Cov}({\rm Im}({\mathcal Z}),{\rm Re}({\mathcal Z}))= {\rm Im}\{\Psi + \Upsilon\}/2$ and  $ {\rm Cov}({\rm Im}({\mathcal Z}))= {\rm Re}\{\Psi - \Upsilon\}/2$. By the continuity of the mapping $ L_n  \mapsto ({\rm Re}(L_n), {\rm Im}(L_n))^\top$ and the continuous mapping theorem,  
 (\ref{eq.SpecMeansNormal})  it then follows  that  in $ HS(\HH)\oplus HS(\HH)$,  
\begin{equation}  \label{eq.SpecMeansNormal-Alt}
	\big({\rm Re}(L_n),  \ {\rm Im}(L_n)\big)^\top  \stackrel{\mathcal D}{\longrightarrow}  {\mathcal N}\big(0, 
	\Sigma(\Psi,\Upsilon) \big).
\end{equation}
as $ n \rightarrow \infty$. Since the converse weak convergence   also is  true,  it follows that   (\ref{eq.SpecMeansNormal}) and (\ref{eq.SpecMeansNormal-Alt}) are equivalent.

An inspection of  the calculations leading to  the covariance  and relation operators  $\Psi$ and $ \Upsilon$, respectively, reveals that the terms $\Psi_1$ 
and $ \Upsilon_1$ are due to the covariance of the periodogram operator $P_{n,\lambda_{j,n}}$, while the terms 
$ \Psi_2$ and $ \Upsilon_2$ are due to  the cross-covariance of  the same  operator  at different Fourier frequencies $ | \lambda_{j,n}|\neq |\lambda_{k,n}|$.  
Although the latter  cross-covariance vanishes at the rate ${\mathcal O}_P(n^{-1})$, see 
\cite{RademacherEtAl2024}, Theorem 4.4,  summing up over $2N$ frequencies, see 
 (\ref{eq.Est-SpecMean}),
leads to a non vanishing  contribution  of these  cross-covariances to the covariance and  to the relation operator of $ {\mathcal Z}$ as  expressed by the terms $\Psi_2$ and $\Upsilon_2$.   

The dependence of the asymptotic law (\ref{eq.SpecMeansNormal}) on the unknown operators $ \F_\lambda$ and $ \F_{\lambda_1,\lambda_2,\lambda_3}$ poses a major difficulty 
for  the applicability of the  derived   convergence results. This  fact motivates the investigation for  alternatives, e.g., bootstrap methods, to estimate the Hilbert-Schmidt operators $\Psi $ and 
$\Upsilon$ as well as the  entire distribution  of  the random element $ L_n=\sqrt{n}\big( M_n(W,P_n)-M_n(W,\F)\big)$ in $HS(H)$.

 
\section{The Functional Frequency Domain Bootstrap} 

\subsection{Preliminaries} \label{sec.Boot-idea}
We assume without  loss of generality, that the normalized eigenvectors $\{ v_r; r\geq 1\}$ of the covariance operator $\Gamma_0={\rm Cov}(X_t)$ form an orthonormal basis
(ONB) for the whole space $\HH$, that is, 
the range of $\Gamma_0$ is  equal to $\HH$. The associated eigenvalues $ \sigma(\Gamma_0 )=\{\sigma_r: r\geq  1\}$ are assumed to be distinct and arranged
in a descending order, that is, $\sigma_1 > \sigma_2> \ldots \,$. For $m \in \N$  let $\PP_m$ denote the orthogonal 
projection operator onto $\HH_{m} := \mathrm{span}\{ v_1,\ldots,v_m \}$, that is, $ \PP_m=\sum_{r=1}^m v_r\otimes v_r$  and let $\PP_m^{\perp} = \Id - \PP_m$  be the orthogonal projection onto 
$\HH_m^{\perp} :=  \overline{\mathrm{span}}\{ v_r; r> m \}$.
It is well-known, see \cite{HsingEubank20152017},  that expanding $X_t$ with respect to this eigenbasis, 
\begin{equation} \label{eq.K-L}
	X_t =  \PP_m(X_t) +  \PP_m^ \perp(X_t) 
	=: X_{t,m} + U_{t,m},
\end{equation}
is optimal in the sense that $X_{t,m}$ minimizes the mean square error  $\E\|X_t-X_{t,m}\|^2$, over all $m$-dimensional approximations.  Moreover, the entire stochastic information of $X_t$ is captured by the so-called scores $\xi_{r,t} := \langle X_t, v_r \rangle$, $r=1,2, \ldots ,$  and we have the following well-known fundamental relations
\begin{equation*}
	\sum_r \E|\xi_{r,t}|^2 = \sum_r \sigma_r = \E\|X_t\|^2 = \|\Gamma_0\|_N.
\end{equation*}    

Let   ${\mathcal X}_m:= \{X_{t,m};t\in \Z\}$  and ${\mathcal U}_m:= \{U_{t,m};t\in\Z\} $. Observe that 
using $\PP_m$, we can decompose  
the spectral density operator as
\begin{align}\label{eq.SpectralDensityDecomp}
	\F_{\lambda}
	&= (\PP_m + \PP_m^{\perp})\,\F_{\lambda}\,(\PP_m + \PP_m^{\perp})\nonumber\\
		& =:  \F_{\lambda}^{({\mathcal X}_m)}  +  G^{(m)}_{\lambda},
\end{align}
where $   \F_{\lambda}^{({\mathcal X}_m)}  = \PP_m \F_\lambda \PP_m$ and 
$G^{(m)}_{\lambda}:=  \F_{\lambda}^{({\mathcal X}_m,{\mathcal U}_m)} + \F_{\lambda}^{({\mathcal U}_m,{\mathcal X}_m)} 
	+ \F_{\lambda}^{({\mathcal U}_m)}$ with the notation 
$  \F_{\lambda}^{({\mathcal X}_m,{\mathcal U}_m)}:=\PP_m\F_\lambda \PP^\perp_m$, $  \F_{\lambda}^{({\mathcal U}_m,{\mathcal X}_m)} := \PP_m^\perp\F_\lambda \PP_m$ and 
	$ \F_{\lambda}^{({\mathcal U}_m)} :=\PP_m^\perp \F_\lambda \PP_m^\perp$.
A decomposition similar to \eqref{eq.SpectralDensityDecomp} can also be obtained for the  periodogram operator 
$P_{n,\lambda} $. That is,  
\begin{align} \label{eq.PerDec}
	P_{n,\lambda} 
	& =: P_{n,\lambda}^{(\mathcal{X}_m)}  + R^{(m)}_{\lambda},
\end{align}
with $P_{n,\lambda}^{(\mathcal{X}_m)}   := \PP_m\,P_{n,\lambda}\,\PP_m$,
 $ R^{(m)}_{\lambda}:=P_{n,\lambda}^{(\mathcal{X}_m,\mathcal{U}_m)} + P_{n,\lambda}^{(\mathcal{U}_m,\mathcal{X}_m)}
	+ P_{n,\lambda}^{(\mathcal{U}_m)} $, using  the definitions 
	$ P_{n,\lambda}^{(\mathcal{X}_m,\mathcal{U}_m)}:= \PP_m\,P_{n,\lambda}\,\PP^\perp_m$, 
	$ P_{n,\lambda}^{(\mathcal{U}_m,\mathcal{X}_m)} :=  \PP_m^\perp\,P_{n,\lambda}\,\PP_m$ and 
	$  P_{n,\lambda}^{(\mathcal{U}_m)} := \PP_m^{\perp}P_{n,\lambda}\,\PP_m^{\perp}$.
	To simplify notation we also write $ \F_\lambda^{(m)}$ and $ P_{n,\lambda}^{(m)}$ for $  \F_{\lambda}^{({\mathcal X}_m)}   $ and $ P_{n,\lambda}^{(\mathcal{X}_m)}   $, respectively. 

Substituting \eqref{eq.SpectralDensityDecomp} and \eqref{eq.PerDec} into \eqref{eq.SpecMeansNormal}  leads  to the decomposition 
\begin{align}\label{eq.SP-dec}
	L_n & := \sqrt{n}( M_n(W,P_n)-M_n(W,\F))  \nonumber \\
	& =\frac{2\pi}{\sqrt{n}} \sum_{j \in {\mathcal G}(n)} W(\lambda_{j,n}) \big(P_{n,\lambda_{j,n}}^{(\mathcal{X}_m)} -  \F_{\lambda_{j,n}}^{({\mathcal X}_m)} \big) \nonumber\\
	& \ \ \ \ + \frac{2\pi}{\sqrt{n}} \sum_{j \in {\mathcal G}(n)} W(\lambda_{j,n}) \big(R^{(m)}_{\lambda_{j,n}} - G^{(m)}_{\lambda_{j,n}} \big)\nonumber\\
	& =:  T_{n,m} + Q_{n,m},
\end{align}
with an obvious notation for $ T_{n,m} $ and $ Q_{n,m} $. Observe that $T_{n,m} =\PP_m L_n\PP_m $ while  $  Q_{n,m}=L_n-\PP_m L_n\PP_m $.\\

\subsection{The core building blocks}
\label{sec.3.2}
The functional frequency domain bootstrap  we propose, aims to generate replicates of $L_n$ which properly imitate  all features of the underlying functional process that affect the distribution of the random element $L_n$. 
In the following we elaborate on  the  major parts of the procedure.

In  the   first part,   pseudo replicates of  $L_n$  are generated using the  basic property  that 
for any $\lambda \in [-\pi,\pi]$, 
\begin{align}\label{eq.WeakConvergenceDFT}
	J_n(\lambda) 
	 \quad\xrightarrow{\mathcal{D}} \quad \mathcal{Y}(\lambda) \sim \mathcal{C}\mathcal{N}(0, \F_{\lambda}),
\end{align}
as $n\to \infty$, where  ${\mathcal C}{\mathcal N}(0, \F_{\lambda})$  denotes a circular-symmetric,  complex Gaussian element  with covariance operator $ {\mathcal F}_\lambda$. Furthermore,  for different frequencies $ \lambda_1,\lambda_2 \in [0,\pi]$ with  $\lambda_1\not=\lambda_2$,  the weak limits $\mathcal{Y}(\lambda_1)$ and $\mathcal{Y}(\lambda_2)$ are independent;
see \cite{CerovekiHoermann2015}. In particular,  and using these 
properties,  
pseudo replicates $ J^\star_n(\lambda_{j,n})$  of $J_n(\lambda_{j,n})$  at the  positive Fourier frequencies $2\pi j/n$, for $ j=1,2, \ldots, N$, are first generated.
This  enables the generation of  pseudo  replicates  $ P^\star_{\lambda_{j,n}}$ of the periodogram operators at the set $ {\mathcal G}(n)$ of frequencies and 
consequently to  a bootstrapped version $L_n^\star$ of $ L_n$.  Define   the  estimated projection operator,  
\begin{align}\label{eq.EmpiricalProjectionOperator}
	\widehat{\PP}_m = \sum_{r=1}^m \,\widehat{v}_r\otimes \widehat{v}_r,
\end{align}
where   $\widehat{v}_r $  are  (empirical) orthonormalized eigenfunctions of $\widehat{\Gamma}_0 := n^{-1} \sum_{t=1}^n X_t \otimes X_t$, identifiable up to a sign and  which correspond to the $m$ largest eigenvalues of $ \widehat{\Gamma}_0$. Analogue to (\ref{eq.SP-dec}), we then decompose  the bootstrap random operator $L_n^\star$ as
\begin{equation} \label{LnBoot-dec}
 L_n^\star = T_{n,m}^\star +Q_{n,m}^\star.
\end{equation}
Here $ T_{n,m}^\star=\widehat{\PP}_m L_n^\star \widehat{\PP}_m $ and $ Q^\star_{n,m} = L_n^\star - \widehat{\PP}_m L_n^\star \widehat{\PP}_m $.
Observe  
that  although the $J_n^\star(\lambda_{j,n})$ are  generated
independently across  the positive Fourier frequencies,
 the components 
$J_n^{(\mathcal{X}_m)^\star}(\lambda_{j,n})=\widehat{\PP}_m(J^\star_n(\lambda_{j,n})) $ and $J_n^{(\mathcal{U}_m)^\star}(\lambda_{j,n}) =\widehat{\PP}^\perp_m(J^\star_n(\lambda_{j,n})) $
still retain  dependencies and we have, 
\begin{align}
	\Cov(J_n^{(\mathcal{X}_m)^\star}(\lambda_{j,n})) 
	&= \widehat{\PP}_m\, \widehat{\F}_{\lambda_{j,n}}\, \widehat{\PP}_m,\label{eq.CovarianceBootDTF_Xm}\\
	\Cov(J_n^{(\mathcal{U}_m)^\star}(\lambda_{j,n})) 
	&= \widehat{\PP}_m^\perp\, \widehat{\F}_{\lambda_{j,n}}\, \widehat{\PP}_m^\perp,\label{eq.CovarianceBootDTF_Um}\\
	\Cov(J_n^{(\mathcal{X}_m)^\star}(\lambda_{j,n}), J_n^{(\mathcal{U}_m)^\star}(\lambda_{j,n})) 
	&= \widehat{\PP}_m\, \widehat{\F}_{\lambda_{j,n}}\, \widehat{\PP}_m^\perp. \label{eq.CovarianceBootDTF_XmDTF_Um}
\end{align}

Now,    if $\widehat{\F}_{\lambda} $ is a (uniformly) consistent estimator for $\F_{\lambda}$ in a way to be specified later on,  then  it  holds true  that 
the bootstrap replicate $L^\star_n$ generated in the way described above,  imitates asymptotically correct the   Gaussian  random element  
$ {\mathcal C}{\mathcal N}(0,\Psi_1, \Upsilon_1)$; see  Lemma~\ref{le.App-CLT-1} of Section~\ref{sec.boocons}.
Observe  that   because of the independence of the  generated  periodogram operators $ P^\star_{\lambda_{j,n}}$ across the 
 positive Fourier frequencies, 
the terms $\Psi_2$ and $ \Upsilon_2$ in the law \eqref{eq.SpecMeansNormal} can not be captured by this part of the procedure.

This motivates  the  second part,  which builds upon   decomposition (\ref{LnBoot-dec}) and  
 corrects   the leading part $ T_{n,m}^\star$  of $L_n^\star$   for  the missing fourth order terms.
Note that the spectral density operator $ F_{\lambda_{j,n}}^{({\mathcal X}_m)}$ is by construction  
a  finite rank, non-negative and self-adjoint   operator which is bounded from below because its smallest eigenvalue is strictly positive in $[0,\pi]$. Therefore,  this operator  possesses  a bounded inverse operator. 
A second resampling procedure is then introduced, which is based on convolved periodograms of subsamples. 
For this,   let  $b<n$  be a subsample size and  consider the set of  periodogram operators  $ P_{b,\lambda_{j,b}}^{(t)}$   of 
 the subsample $ X_t, X_{t+1}, \ldots, X_{t+b-1}$ and calculated for the frequencies $ \lambda_{j,b}=2\pi j/b$, $ j \in {\mathcal G}(b)$. Note that  the upper index $t$  marks the starting point of the subsample
 and that we have in total a number   of $ N_b=n-b+1$ such periodogram operators of subsamples.  The important  point here is  that because each periodogram operator $ P^{(t)}_{b, \lambda_{j,b}}$  is  calculated over a   subsample,   the covariance  structure of this  operator   
 across the  Fourier  frequencies $\lambda_{j,b}$, $j=1,2 \ldots, B$, corresponding  to the subsample length $b$ is retained.
With then proceed with  correcting  the part  $ T_{n,m}^\star $ for the missing fourth order terms. For this a number   $k=\lfloor n/b \rfloor $ of  periodogram operators  $P^{(t)}_{b,\lambda_{j,b}}$, $ j\in{\mathcal G}(b)$,  is chosen  randomly  and with replacement from the set of all possible periodogram operators of subsamples.   Analogously  to (\ref{eq.Est-SpecMean}), the  subsample based  statistic 
 $ M_b(W,P^{(t)}_b)$  is  then calculated  
 and  averaged    over the $k$ selected periodograms to get a  convolved subsample version of $L_n$ which is denoted by $ L_n^+$.  
 Applying  the same  decomposition as  in (\ref{LnBoot-dec}),  we  obtain  the  leading part $ T_{n,m}^+ = \widehat{\PP}_m L_n^+\widehat{\PP}_m$ of $ L_n^+$. 
 The covariance and relation operator of $ T_{n,m}^+$ are  then used to obtain consistent estimators of 
 $\Psi_{2,m}$  and $ \Upsilon_{2,m}$, respectively, where  
  \begin{equation} \label{eq.Psi2m}
 \Psi_{2,m}  =  2\pi \int_{\pi}^\pi \int_{-\pi}^\pi W(\lambda_1)\overline{W}(\lambda_2) F^{({\mathcal X}_m)}_{\lambda_1,-\lambda_1,-\lambda_2}  d\lambda_1 d\lambda_2, 
 \end{equation}
and 
\begin{equation}
\label{eq.Upsilon2m}
\Upsilon_{2,m}  =  2\pi \int_{\pi}^\pi \int_{-\pi}^\pi W(\lambda_1)W(-\lambda_2) F^{({\mathcal X}_m)}_{\lambda_1,-\lambda_1,-\lambda_2}  d\lambda_1 d\lambda_2. 
\end{equation} 
Here,  $\F^{({\mathcal X}_m)}_{\lambda_1,\lambda_2,\lambda_3} =(\PP_m \otimes_{op} \PP_m) {\mathcal F}_{\lambda_1,\lambda_2,\lambda_3} (\PP_m\otimes_{op} \PP_m)$ denotes the fourth order cumulant density operator of $ \{X_{t,m}, t \in \Z\}$, which  can be expressed  as 
\[\F^{({\mathcal X}_m)}_{\lambda_1,\lambda_2,\lambda_3}:=\frac{1}{(2\pi)^3}\sum_{h_1,h_2,h_3\in \Z} \Gamma^{({\mathcal X}_m)}_{h_1,h_2,h_3}e^{-i(\lambda_1h_1+\lambda_2h_2 +\lambda_3h_3)}, \ \ \lambda_1,\lambda_2\lambda_3 \in [-\pi,\pi].\]
and
\begin{align}
\Gamma^{({\mathcal X}_m)}_{h_1,h_2,h_3}&=\E[(X_{h_1,m}\otimes X_{h_2,m})\otimes(X_{h_3,m}\otimes X_{0,m})]   \nonumber\\
& \ \ \ \ - \E[X_{h_1,m}\otimes X_{h_2,m}]\otimes \E[X_{h_3,m}\otimes X_{0,m}] \nonumber \\
& \ \ \ \ -\E[X_{h_1,m}\otimes X_{h_3,m}]\otimes_{op} \E[X_{h_2,m}\otimes X_{0,m}]  \nonumber \\
& \ \ \ \ - \E[X_{h_1,m}\otimes X_{0,m}]\otimes^\top_{op} \E[X_{h_2,m}\otimes X_{h_3,m}] . \nonumber
\end{align}

The third part of the procedure aims to merge in a proper way  the estimators from  the first and  from   the second part  in order 
to obtain a fully functional, bootstrap based estimator of the  distribution as well as of  the covariance and the  relation operator of the random operator $ L_n$. 
To elaborate, and having  consistent estimates of $\Psi_{2,m}  $ and $\Upsilon_{2,m} $ obtained in the second part of the procedure,  we  properly transform 
the bootstrap   element   $T_{n,m}^\star$  obtained in the first part,  so that its   covariance and relation operator is  enriched  with  these estimates. 
Note  that since the distribution of  the random quantities involved  are  complex valued, the aforementioned merging  requires  that the transformation  of $ T_{n,m}^\star$ applied,  
simultaneously corrects  both, the covariance and the relation operator of $ T_{n,m}^\star$.
This  simultaneous correction is  achieved   by  first defining  a two-dimensional   operator vector the  first and second  component of which consist of the    real and of 
 the imaginary part of  $ T^\star_{n,m}$, respectively.  The  covariance matrix operator 
of this  vector is  then  properly transformed  and the obtained   two dimensional vector is   transformed  back to a  complex-valued  bootstrap random element, which is denoted by $ T_{n,m}^o$.  Adding to this  pseudo element  the bootstrap replicate $ Q^\star_{n,m}= L_n^\star-T_{n,m}^\star$ obtained in the first part of the procedure,  leads  to the   final bootstrap proposal  
$L_n^o=T^o_{n,m} + Q^\star_{n,m}$, which is    used  to estimate  the  covariance, the relation operator and the distribution of $ L_n$.

\hide{
As a last point we stress the fact that the $m$-approximating  random element $ X_{t,m}$ as well as the remainder $ U_{t,m}$, see (\ref{eq.K-L}), are not observed. However, they can be estimated  based on the functional time series $X_1,X_2, \ldots, X_n$  at hand 
by   using $\widehat{\xi}_{j,t} = \langle X_t,\widehat{v}_j\rangle$ as an estimator of $\xi_{j,t}$ for $j=1,2, \ldots, m$. Here $ \widehat{v}_j$ is an estimator (up to a sign) of the eigenfunction $v_j$, $j=1,2, \ldots, m$,  of the estimated 
lag zero autocovariance operator     $ \widehat{\Gamma}_0=n^{-1}\sum_{t=1}^n X_t\otimes X_t$. The estimators of $ X_{t,m} $  and $ U_{t,m}$ are  then given by  $ \widehat{X}_{t,m} =\sum_{j=1}^m \widehat{\xi}_{j,t} \widehat{v}_j$  
and  $ \widehat{U}_{t,m}= X_t-\widehat{X}_{t,m}$, respectively,  for $t=1,2, \ldots, n$.
}


\subsection{The bootstrap algorithm} \label{sec.Boot-algo}

The following  algorithm precisely describes the different steps of the  bootstrap procedure outlined  in Section~\ref{sec.3.2}. The first
part of the procedure consists of Step 2 and 3, the second part of Step 4 and 5 while  Step 6 describes the correction of the main part of the statistic made for  capturing  the
missing fourth order terms. The final bootstrap estimator of $L_n$ is  given  in Step 7.

\vspace*{0.25cm}
\begin{enumerate}
\setlength{\itemsep}{3pt}
\item[]\hspace*{-0,75em}{\bf Step 1:}  \  Choose an estimator $ \widehat{{\F}}_{\lambda_{j,n}}$ for the spectral density operator $ \F_{\lambda_{j,n}}$,  $ j \in {\mathcal G}(n)$.
Also, choose a positive integer $m$, $m<n$, and calculate the estimated projection operator $ \widehat{\PP}_m$, see (\ref{eq.EmpiricalProjectionOperator}).
Here $ \widehat{v}_r$, $r=1,2 \ldots, m$,  denote a set of $m$  orthonormalized eigenfunctions corresponding to the $m$-largest eigenvalues $\widehat{\sigma}_1>\widehat{\sigma}_2 
\ldots  > \widehat{\sigma}_m$  of the lag zero sample autocovariance  operator $ \widehat{\Gamma}_0 $.
\item[]\hspace*{-0,75em}{\bf Step 2:} \ 
  For $ \lambda_{j,n}=2\pi j/n$ and $ j=1,2, \ldots, N$, generate   independent pseudo Fourier transforms 
  \begin{align*}
  J_n^\star(\lambda_{j,n}) \sim \mathcal{C}\mathcal{N}(0, \widehat{{\mathcal F}}_{\lambda_{j,n}})  
  \end{align*}
and obtain 
\begin{equation}\label{eq.BootPer}
P^\star_{n,\lambda_{j,n}} =  J_n^\star(\lambda_{j,n}) \otimes J_n^\star(\lambda_{j,n}) , \ \  {\rm for} \ \ j=1,2 \ldots, N,
\end{equation}
and $ P^\star_{n,\lambda_{j,n}}=\overline{P}^\star_{n,-\lambda_{j,n}}$ for $ j=-1,-2, \ldots, -N$. 
\item[]\hspace*{-0,75em}{\bf Step 3:} \ Calculate the bootstrap spectral mean operator
\begin{align*}
L^\star_n
& = \frac{2\pi}{\sqrt{n}} \sum_{j \in {\mathcal G}(n)} W(\lambda_{j,n}) \big(P_{n,\lambda_{j,n}}^\star - \widehat{\mathcal F}_{\lambda_{j,n}} \big) 
\end{align*}
and  the decomposition $ L_n^\star=T_{n,m}^\star + Q_{n,m}^\star$, where    
\begin{equation} \label{eq.Tnm-star}
T_{n,m}^\star
 =    \widehat{\PP}_m L_n^\star \widehat{\PP}_m 
 \end{equation}
 and
 \begin{equation}  \label{eq.Qnm-diamond}
Q_{n,m}^\star   =  L_n^\star -     \widehat{\PP}_m L_n^\star \widehat{\PP}_m.
\end{equation} 
\item[]\hspace*{-0,75em}{\bf Step 4:} \  Choose a positive integer  $ b =b_n <n$. Consider the set of $N_b$ subsamples $ \{ (X_{t} \ldots, X_{t+b-1,m})\,  |\,  t=1,2, \ldots, N_b\}$,  $ N_b=n-b+1$,  and  the corresponding set of periodogram operators,   
\begin{equation} \label{eq.Per-Set}
\{P^{(t)}_ {b, \lambda_{j,b}} = \frac{1}{2\pi b} \sum_{r,s=t}^{t+b-1} X_{r} \otimes X_{s} e^{-i(r-s)\lambda_{j,b}}\, | \, t=1, \ldots, N_b\}.
\end{equation}
For $j \in {\mathcal G}(b)$, let 
$$ \widetilde{\mathcal F}_{\lambda_{j,b}} = \frac{1}{N_b} \sum_{t=1}^{N_b} P^{(t)}_ {b,\lambda_{j,b}} $$
   be the averaged periodogram operators. 

Choose independently and with replacement  $k=\lfloor n/b \rfloor $ periodogram operators  from the set  (\ref{eq.Per-Set}). Denote these operators  by $ P^{ (I_\ell)}_ {b,\lambda_{j,b}} $, where 
$ I_\ell, \ \ell=1, \ldots, k$,  are i.i.d.  random variables having a  discrete uniform distribution  on the set $\{1,2, \ldots, N_b\}$. 
Calculate
\begin{align*}
L_n^+ & = \sqrt{kb}\Big( \frac{1}{k}\sum_{\ell=1}^k \frac{2\pi}{b}\sum_{j\in {\mathcal G}(b)} W(\lambda_{j,b})(P^{(I_\ell)}_ {b,\lambda_{j,b}}  - \widetilde{\mathcal F}_{\lambda_{j,b}} )\Big) \\
& =  \sqrt{kb}  \frac{2\pi}{b}\sum_{j\in {\mathcal G}(b)} W(\lambda_{j,b})(P^{+}_ {b,\lambda_{j,b}}  - \widetilde{\mathcal F}_{\lambda_{j,b}} ),
\end{align*}
where $P^{+}_ {b,\lambda_{j,b}}  : =k^{-1}\sum_{\ell=1}^k P^{(I_\ell)}_ {b,\lambda_{j,b}} $.
Calculate  the  random element  
\begin{equation} 
T_{n,m}^+ = \widehat{\PP}_m L_n^+ \widehat{\PP}_m.
\end{equation}
\item[]\hspace*{-0,75em}{\bf Step 5:} Let $ V^+_{n,m}= \big( {\rm Re}\{ T_{n,m}^+ \} , {\rm Im}\{ T_{n,m}^+ \}\big)^\top$ and 
 define   
\[ S_{2,m}^+ = \Cov^+(V_{n,m}^+) - \Sigma(\Psi^+_{1,m},\Upsilon^+_{1,m}),\]
where for $ \widehat{\widetilde{\mathcal F}}_{\lambda_{j,n}}^{(m)} =   \widehat{\PP}_m \widetilde{\mathcal F}_{\lambda_{j,n}}\widehat{\PP}_m$,
\begin{align*}
 \Psi_{1,m}^+ & =\frac{4\pi^2}{b} \sum_{j\in {\mathcal G}(b)} W(\lambda_{j,b})\overline{W(\lambda_{j,b})} \widehat{\widetilde{\mathcal F}}^{(m)}_{\lambda_{j,b}}\otimes_{op} \widehat{\widetilde{\mathcal F}}^{(m)}_{\lambda_{j,b}} \\
 & \ \ \ \  +\ \   \frac{4\pi^2}{b} \sum_{j\in {\mathcal G}(b)}W(\lambda_{j,b})\overline{W(-\lambda_{j,b})} \widehat{\widetilde{\mathcal F}}^{(m)}_{\lambda_{j,b}}\otimes^\top_{op} \widehat{\widetilde{\mathcal F}}^{(m)}_{-\lambda_{j,b}}
\end{align*}
and
\begin{align*}
\Upsilon_{1,m}^+ &  =\frac{4\pi^2}{b} \sum_{j\in {\mathcal G}(b)} W(\lambda_{j,b})W(-\lambda_{j,b}) \widehat{\widetilde{\mathcal F}}^{(m)}_{\lambda_{j,b}}\otimes_{op} \widehat{\widetilde{\mathcal F}}^{(m)}_{\lambda_{j,b}}  \\
& \ \ \ \ \ \  +  \frac{4\pi^2}{b} \sum_{j\in {\mathcal G}(b)}W(\lambda_{j,b})W(\lambda_{j,b}) \widehat{\widetilde{\mathcal F }}^{(m)}_{\lambda_{j,b}}\otimes^\top_{op} \widehat{\widetilde{\mathcal F}}^{(m)}_{-\lambda_{j,b}}.
\end{align*}
\item[]\hspace*{-0,75em}{\bf Step 6:} Let $ V^\star_{n,m}= \big( {\rm Re}\{ T_{n,m}^\star \} , {\rm Im}\{ T_{n,m}^\star \}\big)^\top$, where   $ T_{n,m}^\star$ is given in (\ref{eq.Tnm-star}).  
 Calculate  
\[ S_{1,m}^\star= \Cov^\star(V_{n,m}^\star) .\]
For
\[ S_m^o=S_{1,m}^\star + S_{2,m}^+,\]
calculate
\[ (Z_{0,m},Z_{1,m})^\top =(S_m^o)^{1/2} (S_{1,m}^\star)^{-1/2} V_{n,m}^{\star}\]
and 
\begin{equation} \label{eq.Tnm-circ}
  T^o_{n,m}= Z_{0,m}+ {\rm i} \, Z_{1,m}.
  \end{equation} 
\item[]\hspace*{-0,75em}{\bf Step 7:}  
The bootstrap approximation of  $ L_n$ is then   given by 
\[  L_n^o = T_{n,m}^\circ + Q_{n,m}^\star ,\]
with $T_{n,m}^\circ $  as  in (\ref{eq.Tnm-circ}) and $ Q_{n,m}^\star$  given in (\ref{eq.Qnm-diamond}). 
\end{enumerate}

\vspace*{0.2cm}

Some remarks concerning the above bootstrap algorithm are in order.

\begin{Rem} 
{~} 
\begin{enumerate}
\item[{\rm (i)}] \ The centering by $ \widetilde{\F}_{\lambda_{j,b}}$ in Step 4 avoids an unnecessary bias  and  is justified by the fact that 
$$ \E^\star( P^{(I_\ell)}_ {b,\lambda_{j,b}} ) = N_b^{-1} \sum_{t=1}^{N_b} P^{(t)}_ {b,\lambda_{j,b}}= \widetilde{\F}_{\lambda_{j,b}}.$$
\item[{\rm (ii)}] \  The operator  $S_{2,m}^+$ defined  in Step 5  aims to estimate  the parts of $ \Psi_m$ and $ \Upsilon_m$ which  only depend on  the fourth order structure of the  process  $\{X_{t,m},t\in\Z\}$, that is the terms   $ \Psi_{2,m}$ and $\Upsilon_{2,m}$.  At the same time, 
  in Step 6,   $S_{1,m}^\star$ is an estimator of the parts of $ \Psi_m$ and $ \Upsilon_m$ which depend on  the  second  order dynamics  of the  process,  that is  of   $ \Psi_{1,m}$ and $\Upsilon_{1,m}$ given by 
  \begin{align*}
	\Psi_{1,m}
	&= 2\pi \Big\{ \int_{\pi}^\pi W(\lambda)\overline{W}(\lambda) \F^{(m)}_\lambda \otimes_{op} \F^{(m)}_\lambda d\lambda +  \int_{\pi}^\pi W(\lambda)\overline{W}(-\lambda) \F^{(m)}_\lambda \otimes_{op}^\top \F^{(m)}_{-\lambda} d\lambda\Big\}
\end{align*}
and
\begin{align*}
	\Upsilon_{1,m} 
	&= 2\pi \Big\{ \int_{\pi}^\pi W(\lambda)W(-\lambda) \F^{(m)}_\lambda \otimes_{op} \F^{(m)}_\lambda d\lambda +  \int_{\pi}^\pi W(\lambda)W(\lambda) \F^{(m)}_\lambda \otimes_{op}^\top 
	\F^{(m)}_{-\lambda} d\lambda\Big\}.
\end{align*} 
Note that $ \Psi_{1,m} = (\PP_m\otimes_{op}\PP_m ) \Psi_{1} (\PP_m\otimes_{op}\PP_m )$ and   $ \Upsilon_{1,m} = (\PP_m\otimes_{op}\PP_m ) \Upsilon_{1} (\PP_m\otimes_{op}\PP_m )$.
\item[{\rm (iii)}]  The operator  $S_m^o$ defined  in Step 6 merges  together  the estimators of $ \Psi_{1,m}$ and $\Upsilon_{1,m}$  obtained in the first part  and  the 
estimators of  $ \Psi_{2,m}$ and $\Upsilon_{2,m}$  obtained in the second part of the  algorithm,  to form an estimator of $ \Psi_m=\Psi_{1,m}+\Psi_{2,m}$  and $ \Upsilon_m=\Upsilon_{1,m}+ \Upsilon_{2,m}$. This estimator 
is used to  properly correct the covariance and the relation operator of 
$ T^\star_{n,m}$ which is achieved  by properly correcting the covariance operator of the two dimensional operator $ V^\star_{n,m}= \big( {\rm Re}\{ T_{n,m}^\star \} , {\rm Im}\{ T_{n,m}^\star \}\big)^\top$.
 Finally,   $T_{n,m}^o$ denotes  the corrected complex-valued version of this operator.
\end{enumerate}
\end{Rem}

 \begin{Rem} 
The   operators $ \Psi_{1,m}^+$ and $ \Upsilon_{1,m}^+$  needed in Step 5 can also be estimated  using the bootstrap in the case  their explicit  calculation is   difficult. To elaborate, let $\widetilde{P}^+_{\lambda_{1,b}}, \widetilde{P}^+_{\lambda_{2,b}}, \ldots, \widetilde{P}^+_{\lambda_{B,b}},$ be a bootstrapped set of periodograms of subsamples, where   for each $ j\in\{1,2,\ldots, B\}$, $ \widetilde{P}^+_{\lambda_{j,b}}$ is chosen independently and with replacement from the set $\{ P_{\lambda_{j,b}}^{(1)}, 
P_{\lambda_{j,b}}^{(2)}, $ $ \ldots, $ $  P_{\lambda_{j,b}}^{(n-b+1)}\}$. Observe that in contrast to Step 4, such a resampling ensures that 
the generated  $\widetilde{P}^+_{\lambda_{j,b}} $ are independent  across the Fourier frequencies $ \lambda_{1,b}, \lambda_{2,b}, \ldots, \lambda_{B,b}$.
Note  that $ \E^\ast( \widetilde{P}^+_{\lambda_{j,b}}) =\widetilde{F}_{\lambda_{j,b}}$, for all $ \lambda_{j,b}$.  By repeating  this  resampling  $k$ times and denoting  by 
$ \widetilde{P}^{+(\ell)}_{\lambda_{1,b}}, \widetilde{P}^{+(\ell)}_{\lambda_{2,b}}, \ldots, \widetilde{P}^{+(\ell)}_{\lambda_{B,b}}$ the  set of periodogram ordinates generated in the $\ell$-th  repetition, $ \ell=1,2, \ldots, k$,  define  
\[ \widetilde{L}^+_n\ =  \sqrt{kb}  \frac{2\pi}{b}\sum_{j\in {\mathcal G}(b)} W(\lambda_{j,b})(\widetilde{P}^{+}_ {b,\lambda_{j,b}}  - \widetilde{\mathcal F}_{\lambda_{j,b}} ),\]
where $\widetilde{P}^{+}_ {b,\lambda_{j,b}}  :=k^{-1}\sum_{\ell=1}^k \widetilde{P}^{+(\ell)}_ {b,\lambda_{j,b}} $. Calculate  the projection $ \widetilde{T}_{n,m}^+ = \widehat{\PP}_m \widetilde{L}_n^+\widehat{\PP}_m$.  Then, $ \widetilde{\Psi}^+_{1,m}= {\rm Cov}^\star(\widetilde{T}_{n,m}^+)$ and 
$  \widetilde{\Upsilon}^+_{1,m}= {\rm Rel}^\star(\widetilde{T}^+_{n,m})$  are  estimators of     $ \Psi_{1,m}^+$ and $\Upsilon_{1,m}^+ $, respectively. 
\end{Rem}

\section{Validity of the functional frequency domain bootstrap}  \label{sec.boocons}
\subsection{Assumptions and basic lemmas}
To investigate the asymptotic properties of the frequency domain bootstrap procedure proposed, some assumptions have to be   imposed on  the moment and the dependence  structure  of the underlying functional process $ {\mathcal X}$ and  the behavior of the parameters  $b$ and $m$. 

\begin{Assu}  \label{as.1} 
$ {\mathcal X} $ is a centered strictly  stationary ${\mathcal H}_\R$-valued process with $ E\|X_0\|^8 <\infty$, which satisfies the  following  summability conditions.
\begin{itemize}
\item[{\rm (i)}] $\sum_{h\in\Z}\|\Gamma_h\|_N <\infty$
\item[{\rm (ii)}] $\sum_{h_1,h_2,h_3\in\Z} \big(|h_1|+|h_2|+|h_3|\big)\|\Gamma_{h_1,h_2,h_3}\|_N <\infty$
\item[{\rm (iii)}] $\sum_{h_1, \ldots, h_7\in\Z}\|{\rm cum}(X_{h_1}\otimes X_{h_2} - \Gamma_{h_1-h_2}, \ldots, X_{h_5}\otimes X_{h_6}-\Gamma_{h_5-h_6},  \\ X_{h_7}\otimes X_{0}-\Gamma_{h_7})\|_N<\infty$.
\end{itemize}
\end{Assu}

Assumption 1(i) ensures that the process $ {\mathcal X}$ possesses   a  continuous spectral density operator $\F_\lambda$.  The summability conditions  in  Assumption~\ref{as.1}(ii) and (iii)   are mainly needed in order to establish 
consistency of  the bootstrap estimators   based on the convolution of periodogram operators  of subsamples,  

 \begin{Assu}   \label{as.W}  The function 
 $W: [-\pi,\pi] \rightarrow \R$  is  of bounded variation.
\end{Assu}

 Assumption~\ref{as.W}  is    common  when dealing with   asymptotic investigations of spectral mean operators  and allows for  a wide range of statistics in functional time series analysis; see Section~\ref{sec2} for examples.

\begin{Assu}   \label{as.F} 
{~}
\begin{enumerate}
\item[(i)] The spectral density operator $ {\mathcal F}_\lambda $ satisfies $ kern({\mathcal F}_\lambda)=0$ for all $\lambda \in[0,\pi]$.
\item[(ii)]  The estimator  $ \widehat{{\mathcal F}}_{\lambda_{j,n}}$ satisfies for $ n \rightarrow \infty$, 
\[ \sup_{0\leq \lambda_{j,n} \leq \pi} \| \widehat{{\mathcal F}}_{\lambda_{j,n}} - {\mathcal F}_{\lambda_{j,n}} \|_N \stackrel{P}{\rightarrow} 0. \]
\end{enumerate}
\end{Assu}

Recall that  for all  $\lambda \in [0,\pi]$, the eigenvalues of $ \F_\lambda$  are  real, positive  and summable by the trace class property of ${\mathcal F}_\lambda$. 
The convergence with respect to the nuclear norm stated in   part (ii) of the above   assumption   is   due to the infinite dimensional structure 
of the space where the random elements considered live. 
For instance, Assumption~\ref{as.F} (ii) is  used  for  bounding  the $2p$-th  order moments of a Gaussian  random element (see Lemma~\ref{le.Mom})  or for  establishing  asymptotic distributional results (see Lemma~\ref{le.App-CLT-1} and the proof of the main Theorem~\ref{th.main}(iii)). 

Assumption~\ref{as.1}(i) and \ref{as.F} (ii)  are sufficient for establishing the following basic result  which refers to  the asymptotic properties  of the random element $ L_n^\star$
obtained  in Step 3 of the bootstrap algorithm.

 \begin{Lem} \label{le.App-CLT-1}
Suppose that  $ \{X_t,t\in\Z\}$ satisfies Assumption~\ref{as.1}(i), that Assumption~\ref{as.W} is fulfilled  and  that the spectral density estimator $\widehat{\F}_{\lambda_{j,n}}$  satisfies Assumption~\ref{as.F} (ii). 
Then, 
\begin{enumerate}
\item[{\rm (i)}]  $ \|{\rm Cov}^\star(L^\star_n)-  \Psi_1\|_{N}   \stackrel{P}{\rightarrow} 0$ and \ $ \|{\rm Rel}^\star(L_n^\star) - \Upsilon_1\|_{N} \stackrel{P}{\rightarrow} 0 $.\\
\item[{\rm (ii)}] In $ HS({\mathcal H})$ it holds true that,  $$L_n^\star  \stackrel{{\mathcal D}}{\rightarrow}  {\mathcal C}{\mathcal N}_{HS}(0,\Psi_1,\Upsilon_1),$$
\end{enumerate}
in probability. 
\end{Lem}

To deal with the estimators of the covariance and the relation operators $ \Psi_{2,m}^{+}={\rm Cov}^{\star}(T_{n,m}^+) - \Psi_{1,m}^+$ and $ \Upsilon_{2,m}^+={\rm Rel}^{\star}(T_{n,m}^+) - \Upsilon_{1,m}^+$, introduced in  Step 5, the following 
assumption regarding  the asymptotic  behavior of the  subsampling  parameter $b$ is imposed.

\begin{Assu} \label{as.3}
The subsampling width $ b$  and  the parameter $k=\lfloor n/b\rfloor$ satisfy the conditions, 
\begin{itemize}
\item[{\rm (i)}] \ $ b,k \rightarrow \infty$ \  and
\item[{\rm (ii)}] \  $b^3/n \rightarrow 0$ \ as \ $n\rightarrow\infty$.
\end{itemize}
\end{Assu}

The  next  result    shows  that  the random element   $L_n^+$ consistently estimates    the  entire covariance  and relation structure of  $L_n$.

\begin{Lem}\label{Lem:CBPConsistencyCovariance}
Suppose that Assumptions \ref{as.1}, \ref{as.W}  and \ref{as.3} are fulfilled.  Then, as $n\rightarrow \infty$, 
\begin{align*}
	\|\Cov^\star( L_n^+ ) -  \Psi\|_{HS}  \xrightarrow{\P} 0  \ \ \mbox{and} \ \  \|{\rm  Rel}^\star(L_n^+) -\Upsilon\|_{HS} \xrightarrow{\P} 0.
\end{align*}
\end{Lem}

\vspace*{0.2cm}

We conclude this section with the following considerations which  justify  the corrections made in Step  6 of the bootstrap algorithm. 
Define for  any (fixed)  $m\in \N$ the covariance and relation operators   of $ \{X_{t,m} , t \in \Z\}$,  
\begin{equation} \label{eq.Cov-Rel-m}
 \Psi_m=\Psi_{1,m} + \Psi_{2,m} \ \ \mbox{and} \ \   \Upsilon_{m}=\Upsilon_{1,m} + \Upsilon_{2,m}.
 \end{equation} 
Notice  that $\| \Sigma(\Psi_m,\Upsilon_m) - {\rm Cov}(V_{n,m})\|_{HS} \rightarrow 0$, where 
$ V_{n,m} = ({\rm Re}(T_{n,m}), {\rm Im}(T_{n,m}))^\top$. 
This  follows from    (\ref{eq.SpecMeansNormal})  and the continuity of the projection operator $ \PP_m$,  which by 
 the continuous mapping theorem yields 
\[ T_{n,m} = \PP_m L_n \PP_m \stackrel{\mathcal D}{\longrightarrow} {\mathcal C}{\mathcal N}(0,\Psi_m,\Upsilon_m).\]
  Therefore, $\Sigma(\Psi_m,\Upsilon_m) $ is  a bounded, nonnegative-definite  operator which possesses a unique nonnegative square root operator 
$  \Sigma(\Psi_m,\Upsilon_m)^{1/2}$; see Theorem 3.4.3 in \cite{HsingEubank20152017}.  

Consider next  the  existence and boundedness of the  inverse  square root operator  $ \Sigma(\Psi_{1,m}, \Upsilon_{1,m})^{-1/2}$.
Note first that    for any $m\in \N$, $ {\mathcal F}_\lambda^{(m)}$   possesses  $m$ strictly positive eigenvalues in the interval $[0,\pi]$.  Observe that 
$\PP_m(X_t) =\sum_{j=1}^m \xi_{j,t} v_j$, where $ \xi_{j,t} =\langle X_t, v_j\rangle$ and that $ {\mathcal F}_\lambda^{(m)} = \sum_{l_1=1}^m\sum_{l_2=1}^m f_{l_1, l_2} (\lambda)  v_{l_1} \otimes  v_{l_2}$.
Here,  
\[f_{l_1, l_2} (\lambda)=\frac{1}{2\pi} \sum_{h\in\Z} {\rm Cov}(\xi_{l_1,t},\xi_{l_2,t+h}) e^{-i h \lambda}, \]
is the cross spectral density of the two score processes $\{\xi_{l_1,t}, t\in\Z\}$ and  $ \{\xi_{l_2,t},t\in\Z\}$.
Since ${\mathcal F}^{(m)}_\lambda$ is  a self-adjoint,  positive definite  and finite rank operator,  $ \Psi_{1,m}$ and $ \Upsilon_{1,m} $ are finite rank operators and  the same holds true for the 
 2$\times$2  operator matrix $ \Sigma(\Psi_{1,m}, \Upsilon_{1,m})$. Moreover, $  \Sigma(\Psi_{1,m}, \Upsilon_{1,m})$  is  also self-adjoint and nonnegative definite 
  because it is a covariance operator. To see this consider the   $m\times m$ spectral density matrix  $ F_m(\lambda)=\big( f_{l_1,l_2}(\lambda) \big)_{l_1,l_2=1,2, \ldots, m}$. Define a 
      Gaussian functional process  $ \breve{\mathcal X}=\{\breve{X}_t, t\in\Z\}$, where  $ \breve{X}_t =\sum_{j=1}^m \breve{\xi}_{j,t}v_j$ and   $ \{(\breve{\xi}_{1,t}, \breve{\xi}_{2,t}, \ldots, \breve{\xi}_{m,t})^\top, t \in \Z\}$ is an $m$-dimensional, Gaussian process with mean zero and spectral density matrix 
  $F_m(\cdot)$.  Let $ \breve{X}_1, \breve{X}_2, \ldots, \breve{X}_n$  be a functional time series stemming from $\breve{\mathcal X}$ and denote by $ \breve{P}_{\lambda_{j,n}}$ the periodogram operator of this time series.
     Define 
 \[ \breve{L}_{n} := \frac{2\pi}{\sqrt{n}}\sum_{j\in {\mathcal G}(n)} W(\lambda_{j,n}) \big( \breve{P}_{n,\lambda_{j,n}}-{\mathcal F}_{\lambda_{j,n}}^{(m)} \big).\]
Then, as in the proof  of  (\ref{eq.SpecMeansNormal}), it can be shown that,   as $n\rightarrow \infty$,  $ \breve{T}_{n,m} :=\PP_m \breve{L}_n \PP_m \stackrel{{\mathcal D}}{\longrightarrow} {\mathcal C}{\mathcal N}(0,\Psi_{1,m}, \Upsilon_{1,m})$,  
   $ \breve{V}_{n,m} := ( {\rm Re}( \breve{T}_{n,m}), \ {\rm Im}( \breve{T}_{n,m}))^\top $ $   \stackrel{{\mathcal D}}{\longrightarrow} {\mathcal N}(0, \Sigma(\Psi_{1,m}, \Upsilon_{1,m})) $ in $ HS(\HH)$ 
 and $  \lim_{n\rightarrow\infty}\|{\rm Cov}(  \breve{V}_{n,m}) - \Sigma (\Psi_{1,m}, \Upsilon_{1,m})\|_{HS} =0$.\\

%
%

\subsection{Approximation errors  for  fixed $m$}

To better understand  the behavior of the frequency domain bootstrap procedure proposed, 
we first investigate its properties  when the decomposition  parameter   $m$  remains   fixed as $n\rightarrow \infty$.  This enables the  identification of  the approximation errors made in   correcting  the  leading term 
of the decomposed  
bootstrap spectral mean operator for the  missing fourth order components.    

We start with the following lemma 
which  shows that 
the part of the bootstrap procedure  based on convolved periodogram operators of subsamples (see Step 4 and 5),  consistently estimates the terms  $ \Psi_{2,m}$ and $ \Upsilon_{2,m}$  attributed to the fourth order characteristics  of the 
underlying  process.

\begin{Lem}  \label{le.m-Fixed-1}
Under Assumption~\ref{as.1}, \ref{as.W} and \ref{as.3} and for any  fixed $ m \in \N$,  it holds true that, as $n\rightarrow \infty$, 
$$\big\| S_{2,m}^+ -  \Sigma(\Psi_{2,m}, \Upsilon_{2,m})
\big\|_{HS} \stackrel{P}{\rightarrow} 0,$$
where $ \Psi_{2,m}$ and $\Upsilon_{2,m}$ are defined in (\ref{eq.Psi2m}) and (\ref{eq.Upsilon2m}), respectively.
\end{Lem}

Towards  quantifying  the error made  for any fixed $m$  in approximating the distribution  of $ L_n$  by  the bootstrap  analogue   $ L^o_n$,    some additional notation is needed. 
Recall the definitions of $ T_{n,m}^\star$ and $ Q_{n,m}^\star$ in Step 3 of the bootstrap algorithm and let 
\begin{align*}
\Psi_{1,m}^{(\F, G)}   & := \lim_{n\rightarrow \infty} {\rm Cov}^\star(T^\star_{n,m},Q^\star_{n,m}) \\
& = 2\pi \Big\{\int_{-\pi}^\pi W(\lambda)\overline{W}(\lambda) \F^{(m)}_\lambda \otimes_{op} G^{(m)}_\lambda d\lambda \\
& \ \ \ \ \  +\int_{-\pi}^\pi W(\lambda)\overline{W}(-\lambda) \F^{(m)}_\lambda \otimes^\top_{op} G^{(m)}_{-\lambda} d\lambda \Big\}
\end{align*}
and  
\begin{align*}
\Upsilon_{1,m}^{(\F, G)}  &  := \lim_{n\rightarrow\infty}{\rm Rel}^\star(T^\star_{n,m},Q^\star_{n,m}) \\
& = 2\pi \Big\{\int_{-\pi}^\pi W(\lambda)\overline{W}(-\lambda) \F^{(m)}_\lambda \otimes_{op} G^{(m)}_\lambda d\lambda \\
& \ \ \ \ \ \ +\int_{-\pi}^\pi W(\lambda)W(\lambda) \F^{(m)}_\lambda \otimes^\top_{op} G^{(m)}_{-\lambda} d\lambda \Big\}.
\end{align*}
Both convergences  above are with respect to the  Hilbert-Schmidt norm. 
From  Lemma~\ref{le.App-CLT-1} we easily  get  that  
$  \|{\rm Cov}(T^\star_{n,m},Q^\star_{n,m})   -  \Psi_{1,m}^{(\F,G)} \|_{HS} \rightarrow 0$ and    that $\|{\rm Rel}^\star (T^\star_{n,m},Q^\star_{n,m})  - \Upsilon_{1,m}^{(\F,G)}\|_{HS} \rightarrow 0$, in probability, as $n\rightarrow \infty$. 
Also define  the   "adjoined" limits, with respect to the  Hilbert-Schmidt norm,
  \[ \Psi_{1,m}^{(G, \F)} := \lim_{n\rightarrow\infty}  {\rm Cov}^\star(Q^\star_{n,m},T^\star_{n,m}),  \]
  and  
\[  \Upsilon_{1,m}^{(G, \F)} := \lim_{n\rightarrow\infty}{\rm Rel}^\star(Q^\star_{n,m},T^\star_{n,m}) .\]
Note that the above  limits  obey   the same  expressions as     $ \Psi_{1,m}^{(\F, G)} $ and $ \Upsilon_{1,m}^{(\F, G)}  $,  respectively,  with the only difference that the positions of 
$ \F^{(m)}_\lambda$ and $ G^{(m)}_\lambda$ in the corresponding operator  tensor products  and transposed  operator tensor products are interchanged. 

 To obtain the analogue quantities  for the bootstrap procedure, recall the definition of $ V_{n,m}^\star$ in Step 6 of the bootstrap algorithm and define the two dimensional vector $  W_{n,m}^\star = ({\rm  Re}\{Q_{n,m}^\star\}, {\rm Im}\{Q_{n,m}^\star\})^\top$. Consider then for $m$ fixed,  the  limits,  
 \begin{align*}
 C^{(m)} & =\left(  C^{(m)}_{i,j}\right)_{i,j=1,2}\\
& :=\lim_{n\rightarrow \infty} {\rm Cov}^\star\big( (S_m^o)^{1/2} (S_{1,m}^\star)^{-1/2} V_{n,m}^{\star},  W_{n,m}^{\star} \big)\\
& = \Sigma(\Psi_m,\Upsilon_m)^{1/2} \Sigma(\Psi_{1,m}, \Upsilon_{1,m})^{-1/2} \Sigma(\Psi_{1,m}^{(\F,G)}, \Upsilon_{1,m}^{(\F,G)})
 \end{align*}
 and
   \begin{align*}
 D^{(m)} & =\left(  D^{(m)}_{i,j}\right)_{i,j=1,2}\\
& :=\lim_{n\rightarrow \infty} {\rm Cov}^\star\big(  W_{n,m}^{\star} ,(S_m^o)^{1/2} (S_{1,m}^\star)^{-1/2} V_{n,m}^{\star}\big)\\
& =   \Sigma(\Psi_{1,m}^{(G,\F)}, \Upsilon_{1,m}^{(G,\F)})\Sigma(\Psi_{1,m}, \Upsilon_{1,m})^{-1/2}\Sigma(\Psi_m,\Upsilon_m)^{1/2}.
 \end{align*}
  Both limits above  are with respect to the Hilbert-Schmidt norm. As it can be seen from the proof of Lemma~\ref{le.m-Fixed-2} below, both limits exist.
 Further define  the following quantities:
 \[  \widetilde{\Psi}_m^{(\F, G)} :=  \big(C^{(m)}_{1,1}+C^{(m)}_{2,2}\big) +
  i \big(C^{(m)}_{2,1}-C^{(m)}_{1,2}\big), \]
  \[ \widetilde{\Psi}_m^{(G, \F)} :=\big( D^{(m)}_{1,1}+D^{(m)}_{2,2} \big) +
  i  \big(D^{(m)}_{2,1}-D_{1,2}\big), \]
   \[ \widetilde{\Upsilon}_M^{(\F,G)} :=\big(C^{(m)}_{1,1}-C^{(m)}_{2,2}\big)  +
  i \big(C^{(m)}_{2,1}+C^{(m)}_{1,2}\big)\]
and
  \[ \widetilde{\Upsilon}_m^{(G,\F)} :=
   \big(\widetilde{D}^{(m)}_{1,1}-\widetilde{D}^{(m)}_{2,2}\big) +  i  \big(\widetilde{D}^{(m)}_{2,1}+\widetilde{D}^{(m)}_{1,2}\big).\]

The following result clarifies what the frequency domain bootstrap procedure does   with respect to 
the  estimation of   the covariance and  the relation operator     as well 
as of the   distribution of $ L_n$, when $m$ remains  fixed as $n \rightarrow \infty$.  

\begin{Lem} \label{le.m-Fixed-2}
Under Assumptions~\ref{as.1}, \ref{as.W}, \ref{as.F} and \ref{as.3} and for $ m \in \N$ fixed, the following assertions hold true as $ n\rightarrow\infty$,
\begin{itemize}
\item[{\rm (i)}]  $ \|{\rm Cov}^\star(L^o_n)-(\Psi_1 + \Psi_{2,m} + \Delta_{\Psi,m})\|_{HS}   \stackrel{P}{\rightarrow}  0$ \\ and  $ \|{\rm Rel}^\star(L_{n}^o) -    (\Upsilon_1+\Upsilon_{2,m} +\Delta_{\Upsilon,m} )\|_{HS} \stackrel{P}{\rightarrow} 0$.\\
\item[{\rm (ii)}]  In $ HS({\mathcal H})$ it holds true that, $$ L_n^o  \stackrel{{\mathcal D}}{\rightarrow}  {\mathcal C}{\mathcal N}(0, \Psi_1 + \Psi_{2,m} + \Delta_{\Psi,m}, \Upsilon_1+\Upsilon_{2,m} +\Delta_{\Upsilon,m}),$$ 
in probability. 
\end{itemize}
In the above expressions,
\[ \Delta_{\Psi,m}:= \big(\widetilde{\Psi}_m^{(\F, G)} - \Psi_m^{(\F, G)}  \big) + \big( 
 \widetilde{\Psi}_m^{(G, \F)} - \Psi_m^{(G, \F)}\big) \]
and
\[ \Delta_{\Upsilon,m} :=  \big(\widetilde{\Upsilon}_m^{(\F, G)}-\Upsilon_m^{(\F, G)} \big)
 + \big( \widetilde{\Upsilon}_m^{(G,\F)}\big) - \Upsilon_m^{(G, \F)}\big).\]
 \end{Lem}
 
By the  above lemma and for any  $m$ fixed,    the limiting distribution of $ L_{n}^o$  is a Gaussian  random element in $HS({\mathcal H})$.   Due to the correction made in Step 6,  the    bootstrap estimators of  
the covariance and of  the relation operator   are enriched with the fourth order terms $\Psi_{2,m}$ and $ \Upsilon_{2,m}$; compare the  analogue  expressions in Lemma~\ref{le.App-CLT-1}.
However, two approximation errors  occur.

The first error consists of the two  terms $\Delta_{\Psi,m} $ and $ \Delta_{\Upsilon,m}$ appearing in the  limiting expressions for the covariance and the relation operator, respectively. 
As an  inspection of the proof of Lemma~\ref{le.m-Fixed-2} reveals, $\Delta_{\Psi,m} $ is  due  to the difference between the limits of the  two covariance operators 
$$ {\rm Cov}(V^\star_{n,m},W^\star_{n,m})   \ \ \mbox{and} \ \  {\rm Cov}^\star\big( (S_m^o)^{1/2} (S_{1,m}^\ast)^{-1/2} V_{n,m}^{\star},  W_{n,m}^{\star} \big),$$  
while $\Delta_{\Upsilon,m} $ 
is due to the difference  between the limits of the two  relation operators 
$$  {\rm Rel}(V^\star_{n,m}, W^\star_{n,m})  \ \ \mbox{and}    \ \ 
{\rm Rel}^\star\big( (S_m^o)^{1/2} (S_{1,m}^\ast)^{-1/2} V_{n,m}^{\star},  W_{n,m}^{\star} \big) .$$
Recall that  
 $ V^\star_{n,m} =  ({\rm  Re}\{T^\star_{n,m}\}, {\rm Im}\{T^\star_{n,m}\})^\top$ while  $ W^\star_{n,m} =   ({\rm  Re}\{Q^\star_{n,m}\}, {\rm Im}\{Q^\star_{n,m}\})^\top$. 
 Both error  terms,  $\Delta_{\Psi,m} $ and $\Delta_{\Upsilon,m} $,  can be bounded  in a similar way. For   instance  we get for the term  $\Delta_{\Psi,m}$, 
  \[  \|\Delta_{\Psi,m}\|_{HS} = {\mathcal O}_{P}(1) \| (S_m^o)^{1/2} (S_{1,m}^\ast)^{-1/2} -Id\|_{HS}\sqrt{ \int_{-\pi}^\pi ||\F^{(m)}_\lambda- \F_\lambda\|^2_{HS} d\lambda}  .\] 
Hence,   the closer  is  the spectral density operator $ \F^{(m)}_\lambda$  
 to the spectral density operator $ \F_\lambda$, the smaller  will $  \Delta_{\Psi,m}$ be. Also,  if $ \Psi_{2,m}=0$,
  then 
 $  (S_m^o)^{1/2} (S_{1,m}^\ast)^{-1/2} \stackrel{P}{\rightarrow}  Id $ in HS-norm, as $ n \rightarrow \infty$ ($m$ is fixed), that is the term $  \| (S_m^o)^{1/2} (S_{1,m}^\ast)^{-1/2} -Id\|_{HS}$ will vanish in probability, asymptotically  and for any fixed $m$. 
The same   conclusions  can   be 
 made   for the error term  $  \Delta_{\Upsilon,m}$.

 The second error  which appears in  the limiting expressions given in  Lemma~\ref{le.m-Fixed-2}, is related to  the fourth order terms  affecting the distribution of $L_n$. 
 Recall   that  the  complementary resampling procedure based on convolved periodogram operators  of subsamples only  is applied  in order to correct for  the missing fourth order terms related to 
 the part $T_{n,m}^\star =\PP_mL_n^\star \PP_m$  of the decomposition of $ L_n^\star $.
 The   error made  in this context   is  essentially due to the relevant fourth order characteristics of the remainder 
 $ Q^\star_{n,m}  = L_n^\star - T_{n,m}^\star$, which are not captured by this correction procedure.
Clearly,  this  error  depends on    the differences  
$ \Psi_2-\Psi_{2,m}$ and $ \Upsilon_2-\Upsilon_{2,m}$ for the covariance and for  the relation operator, respectively.  Both  differences  can  
be   
bounded by  $\| \F^{({\mathcal X}_m)}_{\lambda_1,-\lambda_1,-\lambda_2} -  \F_{\lambda_1,-\lambda_1,-\lambda_2}\|_{HS}$, that is,  by the difference between the fourth order spectral density operator 
of the $m$-approximating process $ {\mathcal X}_m$ and the fourth order spectral density operator of  the underlying process $ {\mathcal X}$. 
The better    the  fourth order structure of the underlying functional process is approximated  by the fourth order structure of   the process $\{X_{t,m}, t\in\Z \}$,  
the smaller will   be the difference $\|\Psi_2-\Psi_{2,m}\|_{HS}$, respectively,  $ \|\Upsilon_{2}-\Upsilon_{2,m}\|_{HS}$. 

It is worth mentioning that   if $\{X_t,t\in\Z\}$  is a Gaussian process in ${\mathcal H}_\R$, then the fourth order cumulant  spectral density operators  vanish, which implies that the second   error terms  are zero, i.e. $\Psi_2=\Upsilon_2=0 $ and, consequently, $ \Psi_{2,m}=\Upsilon_{2,m}=0$ for any $m\in\N$.
 In this  particular case it  also can be shown that $\|S_m^o -  S_{1,m}^\ast  \|_{HS} \stackrel{P}{\rightarrow} 0$, because $ \|S_{2,m}^+\|_{HS} \stackrel{P}{\rightarrow} 0$ as $n\rightarrow \infty$. This   implies that 
also the first approximation error as described  by the terms $  \Delta_{\Psi,m} $ and $   \Delta_{\Upsilon,m}$, will  vanish asymptotically for any fixed $m$. 
Apart from such  specific cases,  however, for  general stationary functional processes,   both errors become  asymptotically negligible if the    dimension   $m$  
used  in the   decomposition  of  $L_n^\star$,  increases to infinity at an appropriate rate as the sample size $n$ increases to infinity. This case,   which establishes  consistency of the frequency domain bootstrap procedure for general functional  processes,  
is addressed in     the next section.

\subsection{Bootstrap consistency}

In order to   allow  for  the decomposition parameter   $m$   to increase, 
the following  additional assumption is needed which concerns  the rate with which $m$  is allowed to  increase to 
infinity as the sample size $n$ increases to infinity,

\begin{Assu}   \label{as.4}
$m=m(n) \rightarrow  \infty$ as $ n\rightarrow \infty$ such that $ n^{-1/2} \sqrt{\sum_{j=1}^m 1/\alpha_j^2}  \rightarrow 0$,
where $\alpha_1=\sigma_1-\sigma_2$ and $ \alpha_j=\min\{\sigma_{j-1}-\sigma_j, \sigma_j-\sigma_{j+1}\}$ for $ 2\leq j \leq m$
\end{Assu}


\begin{Rem} Assumption~\ref{as.4}   is needed in order to control asymptotically the estimation error made by replacing the  unknown projection operator $ \PP_m$ by its 
 sample  counterpart  $\widehat{\PP}_m$ given  in (\ref{eq.EmpiricalProjectionOperator}).
According to this assumption,   the rate with which $m$ is allowed to increase to infinity  has to  take into account the 
 spectral gap $\alpha_j$, $j=1,2 \ldots, m$, i.e.,   the rate of decrease of the eigenvalues $\sigma_j$ of the lag zero covariance operator $ \Gamma_0$.
  Assume for instance   a  polynomial decay  of the eigenvalues $\sigma_j$, more specifically suppose  that $ \sigma_j-\sigma_{j+1} \geq Cj^{-\theta}$, for some constants  $ C>0$  
 and $\theta >1$. Integral approximation  of the sum   yields  the bound
 \[ \frac{1}{\sqrt{n}}\sqrt{\sum_{j=1}^m \frac{1}{\alpha_j^2}} \leq  
 \frac{\sqrt{C}}{\sqrt{2\theta +1}} \frac{1}{\sqrt{n}} (m+1)^{\theta +1/2},\]
which implies that in this  case, Assumption~\ref{as.4} is satisfied if   $m$ is allowed to  increase as $ m \sim n^{\beta}$ for any   $\beta \in (0, 1/(2\theta+1))$. 
\end{Rem}

Recall the final bootstrap proposal  $ L_n^o =T_{n,m}^o + Q^\star_{n,m}$ in Step 7 of the bootstrap algorithm, where $ T_{n,m}^o =(1, {\rm i})^\top (S_{1,m}^\star + S_{2,m}^+)^{1/2} (S_{1,m}^\star)^{-1/2} V_{n,m}^\star$,  is the  corrected leading part.   We conclude our  theoretical investigations with 
the following  theorem   which shows that, allowing   $m$ to increase to infinity as $n$ increases to infinity under the assumptions made,  $S_{2,n}^+$ consistently  estimates the  entire  operators $ \Psi_2$ and $ \Upsilon_2$  attributed to  the fourth order structure of $ \{X_t,t\in\Z\}$. Furthermore,  under the same assumptions, 
the  covariance and  the relation  operator of the leading  part $ T_{n,m}^o  $,   consistently estimate   
the entire covariance and relation operators, $\Psi= \Psi_1+\Psi_2$ and $ \Upsilon=\Upsilon_1+\Upsilon_2$, respectively. Finally,  if   the process ${\mathcal X}$   
fulfills  the conditions required such that  the sequence $\big(L_n=\sqrt{n}\big(M_n(W,P_n)-M_n(W,F) \big), n\in \N\big)$ satisfies (\ref{eq.SpecMeansNormal}), then 
consistency of the bootstrap  proposal  $L_n^o$  for  estimating the   distribution of $L_n$  also is established. 

\begin{Thrm}  \label{th.main}Under Assumption~\ref{as.1}, \ref{as.W}, \ref{as.F}, \ref{as.3} and \ref{as.4}, it holds true that, as $n\rightarrow \infty$, 
\begin{enumerate}
\item[(i)] \ $ \big\|S_{2,m}^+ -  \Sigma(\Psi_2,\Upsilon_2)\big\|_{HS} \stackrel{P}{\rightarrow} 0$
\item[{\rm (ii)}] \  $\|{\rm Cov}^\star(T^o_{n,m}) -  (\Psi_1+\Psi_2)\|_{HS}   \stackrel{P}{\rightarrow} 0$  and 
 $ \|{\rm Rel}^\star(T_{n,m}^o)   -   (\Upsilon_1+\Upsilon_2)\|_{HS}   \stackrel{P}{\rightarrow}0$
 \item[(iii)] \  $L_n^o  \stackrel{{\mathcal D}}{\rightarrow}  {\mathcal C}{\mathcal N}(0,\Psi_1+\Psi_2,\Upsilon_1+\Upsilon_2),$  in probability.
  \end{enumerate} 
\end{Thrm}

%
%

\section{Practical  Issues and Simulations} \label{sec.Impl-Sim}
\subsection{Choice of  Bootstrap Parameters}

To implement the frequency domain bootstrap procedure one has to select the estimator of the spectral operator $ \widehat{F}_{\lambda}$, the  decomposition  parameter $m$ and the subsampling   size $b$. 

Our numerical  experiments   suggest that the most crucial choice is that of  choosing   $ \widehat{F}_\lambda$. Recall  that for estimating ${\mathcal F}_\lambda$, some alternatives exist, like for instance,  lag window  or   spectral window estimators, where the latter   are   given by  
 	\begin{equation}
		\wh{\F}_{\lambda_{j,n}} = \frac{1}{n} \sum_{\lambda_{s,n}} K_{h}(\lambda_{j,n}-\lambda_{s,n}) P_{n,\lambda_{s,n}}.
	\end{equation} 
	Here $ K_h(\cdot) =h^{-1}K(\cdot/h)$ is  a kernel function with $K$ satisfying certain conditions and $h$  a smoothing bandwidth. Different  approaches  to choose $h$ have been proposed in the finite  dimensional time series case. In our simulation experiments, we use  different values  of $h$ in order to investigate the sensitivity of the results obtained with respect to the choice of this tuning parameter.

The choice of $m$ has been investigated under a variety of different settings in the literature and some different  proposals have been made; see among others,  
\cite{HorvathKokoszka2012},  \cite{LiWangCarolli2013}, \cite{HoermannKidzinski2015},  \cite{AueKlepsch2017}  and \cite{Paparoditis2018}.  A well known  simple rule,   initially proposed   for   i.i.d. functional data, is  the $VR_n(m)$  rule.  According to this rule $m$ is selected as the smallest integer for which the ratio \[
VR_n(m)=\sum_{j=1}^m \widehat{\sigma}_j/ \sum_{j=1}^n \widehat{\sigma}_j = \|\widehat{\PP}_m\widehat{\Gamma}_0 \widehat{\PP}_m\|_N/  \|\widehat{\Gamma}_0\|_N,\] exceeds some  pre-specified value $Q$, where $Q=0.80$ or $ Q=0.85$ are  two common choices. For our set up, a more appropriate  criterion  to select $m$ is one which  explicitly takes into account the temporal dependence of the  functional  data  at hand. In  this context,  the  use of  the following   dependent variance ratio ($DeVR_n$)  criterion is suggested, according to which,   $m$ is selected as the smallest integer for which the ratio
\begin{equation} \label{eq.DVR}
DeVR_n(m)=  \sum_{j\in{\mathcal G}(n)}  \|\widehat{\PP}_m P_{n,\lambda_{j,n}} \widehat{\PP}_m\|_N\Big/  \sum_{j\in{\mathcal G}(n)}  \| P_{n,\lambda_{j,n}} \|_N ,
\end{equation}
exceeds  a pre-specified value $Q$.   Note that   the sums of the eigenvalues of the lag zero sample autocovariance operator  $ \widehat{\Gamma}_0$ and of its projection
appearing in $VR_n$, have been replaced in the  $DeVR_n$ criterion by the sum of the (integrated) eigenvalues of the periodogram operator
$P_{n,\lambda_{j,n}}$ and of its projection, respectively.  It can be shown that for $m$ fixed,
\[  \sum_{j\in{\mathcal G}(n)}  \|\widehat{\PP}_m P_{n,\lambda_{j,n}} \widehat{\PP}_m\|_N\Big/  \sum_{j\in{\mathcal G}(n)}  \| P_{n,\lambda_{j,n}} \|_N  \stackrel{P}{\rightarrow} 
\int_{-\pi}^\pi \| \F_\lambda^{(m)}\|_N\Big/ \int_{-\pi}^\pi   \| \F_\lambda \|_N,\]
as $n\rightarrow\infty$. From this we see that   if  the  random elements $X_t$ are  uncorrelated then $ {\mathcal F}_\lambda=(2\pi)^{-1}\Gamma_0$, which implies that in this  case 
 $ DeVR_n$  coincides  (asymptotically) with  $VR_n$.

Finally,  and regarding the subsampling parameter  $b$, our simulation results suggest that  the quality of the bootstrap approximations obtained  are not sensitive with respect to the choice of this parameter, provided b is not too small, i.e.,  $ b$ is chosen to be larger than $\lfloor n/4\rfloor$.

\subsection{Numerical results} \label{sec.NumRes}

 To  investigate the finite sample performance of the frequency domain bootstrap proposed,   some numerical experiments   have been conducted 
using 
functional time series  $X_1, X_2, \ldots, X_n$, of length $ n=128$, where $ X_t\in \HH_\R=L^2([0,1], \R)$  and 
  the random elements $ X_t$ are observed over a grid of $T$ points in the interval $[0,1]$, that is, $ X_t(\tau_j)$, $ j=1,2, \ldots, T, $ and $t=1,2, \ldots, n,$ are the  observations available.  Note that 
  in this case,  the  kernel of the operator $ L_n$ can be described by a $ T\times T$ matrix with 
 elements  
 \[\ell_n(\tau_r,\tau_s)= \frac{2\pi}{\sqrt{n}} \sum_{j\in {\mathcal G}(n)} W(\lambda_{j,n}) \big( p_{\lambda_{j,n}}(\tau_r,\tau_s) -  f_{\lambda_{j,n}}(\tau_r,\tau_s)\big),\]
for $r,s \in \{1,2, \ldots, T\}$,  where $ p_{\lambda_{j,n}}(\tau_r,\tau_s)$ and $f_{\lambda_{j,n}}(\tau_r,\tau_s) $ denote the periodogram and the spectral density kernels   associated with the periodogram and spectral density operators $ P_{\lambda_{j,n}}$ and $ F_{\lambda_{j,n}}$, respectively.  The  kernel of   $ T_{n,m}=\PP_m L_n \PP_m$ is then given by  
 \begin{equation} \label{eq.TnmL2} t_{n,m}(\tau_r, \tau_s) =\sum_{k_1=1}^m\sum_{k_2=1}^m \Xi(k_1,k_2) v_{k_1}(\tau_r)v_{k_2}(\tau_s)
 \end{equation}
 where  $v_k=(v_{k}(\tau_1), v_k(\tau_2), \ldots, v_k(\tau_T)) $ is an  orthonormalized eigenvector associated with the $ k$th  largest eigenvalue of the  covariance operator $ \Gamma_0$, which has  kernel $ \gamma_0(\tau_r,\tau_s)=\E (X_t(\tau_r)X_t(\tau_s))$, $ r,s \in \{1, 2, \ldots, T\}$. Furthermore, 
 \[ \Xi(k_1, k_2) =\sum_{p=1}^T\sum_{q=1}^T \ell_n(\tau_p,\tau_q) v_{k_1}(\tau_q) v_{k_2}(\tau_p), \ \  k_1,k_2=1,2, \ldots, m,\]
 and the   $ m\times m$ random matrix $ \Xi_m=\big(\Xi(k_1,k_2)\big)_{k_1,k_2=1,2, \ldots,m}$, captures the entire stochastic information associated with the operator  $ T_{n,m}$.

In our simulations, the  random curves $X_t$ considered are approximated using $25$ equidistant points in the unit interval and they are transformed into functional objects using the Fourier basis with $D=25$ basis functions,  denoted by $f_1,f_2, \ldots, f_D$ in the sequel.    Three commonly used functional  autoregressive (FAR(1)) and  functional moving-average (FMA(1)) models are considered:
\begin{enumerate}
\item[] \ Model I: \ \  \ \ \  $ X_t =\Psi_1(X_{t-1}) + \varepsilon_{1,t}$
\item[] \  Model II: \ \  \ \ $ X_t =\Psi_1(\varepsilon_{2,t-1}) + \varepsilon_{2,t}$
\item[] \   Model III: \ \  $ X_t =\Psi_2(\varepsilon_{3,t-1}) + \varepsilon_{3,t}$.
\end{enumerate}
The $\varepsilon_{j,t}$, $j=1,2,3$, are  i.i.d. functional  innovations  which   follow  different laws: 
$ \varepsilon_{1,t},$ is a standard Brownian motion,  $ \varepsilon_{2,t}  =\sum_{j=1}^D Z_j f_j$ and $ \varepsilon_{3,t}=\sum_{j=1}^D U_j f_j$, where the $Z_j$ are independent Gaussian random variables with mean zero and  standard deviation $ 1/j$, while the $ U_j$ are independent uniform distributed random variables in the interval $ [-\sqrt{3}/j, \sqrt{3}/j]$.   $ \Psi_1$  and $\Psi_2$ are integral operators with kernels $ \psi_1(u,v) = C_1\cdot e^{-(u^2+v^2)/2}$ and
$\psi_2(u,v)=C_2\cdot \min\{u,v\}$,  $ u,v\in L^2[0,1]$, respectively. The  constants $ C_1$ and $ C_2$ are specified such that $ \|\Psi_1\|_{\mathcal L}=\|\Psi_2\|_{\mathcal L}=0.8$.
The FMA(1) Model III is of particular interest since, in contrast to the FAR(1) Model I and  the FMA(1)  Model II,  the fourth order spectral density operator does not vanish (recall that    
$ \E (U_j^4)/(\E (U_j^2))^2 -3 = -6/5$).  To see the effect  of the correction step using the convolved subsampling procedure, we 
report for this case also the estimation results obtained which are  based on the (non corrected)  bootstrap  quantity $L_n^\star$; see Step 3 of the bootstrap algorithm.

We first investigate the fine sample behavior of  $ VR_n$ and $DeVR_n$ to select the projection dimension $m$.  
Table~\ref{tab.Table-m}  shows   results for  the values of $m$ selected over 1,000 trials.   As it is seen from this table, for the models  and for  the sample size considered, 
the two methods behave  similar with the $ DeVR_n$ criterion, which takes into account the temporal dependence of the functional data, selecting   slightly more often higher values of $m$ compared to the $VR_n$ criterion. 

\begin{table}[t]
\begin{center}
\setlength{\tabcolsep}{2.6mm}
\renewcommand{\arraystretch}{1.2}
\begin{tabular}{|c|cc|cc|cc|}
\hline
 m & \multicolumn{2}{l|}{\textbf{Model I}} & \multicolumn{2}{l|}{\textbf{Model II}} & \multicolumn{2}{l|}{\textbf{Model III}} \\
   & $VR_n$ & $DeVR_n$ & $VR_n$ & $DeVR_n$ & $VR_n$ & $DeVR_n$ \\
\hline
 1 & - & - & - & - & - & - \\
 2 & 99.5 & 99.5 & 30.6 & 28.8 & 52.8 & 52.0  \\
 3 & 0.05 & 0.05 & 69.0 & 70.8 & 47.2 & 48.0  \\
 4 & - & - & 0.04 & 0.04 & - & - \\
\hline
\end{tabular}\hspace*{0.3cm}
\vspace*{2mm}
\caption{Frequency of selected values  of $m$ for the three models considered  by using the $VR_n$ and the $DeVR_n$ criterion with $Q=0.85$. \label{tab.Table-m}
} 
\end{center}
\end{table}

In our simulation study, the target parameter is the standard deviation  $Std(\tau_r)$ , $r=1,2, \ldots, T$,  of the sample lag zero autocovariance kernel, 
\[ \widehat{\gamma}_0(\tau_r,\tau_r)=\frac{1}{T}\sum_{t=1}^n X^2_t(\tau_r), \ \ r=1,2, \ldots, 25,\]
approximated  over   the grid of  $ 25$  equidistant points  in the interval $[0,1]$.   The  standard deviations  $Std(\tau_r)$, $r=1,2, \ldots, 25$,  are estimated using the frequency domain bootstrap 
procedure proposed.   To obtain the  exact  values  
   of the target function  $Std(\tau)$, 
   $10,000$ repetitions of  all three models have been used. 
   The spectral density operator  $\widehat{F}_\lambda$  involved in the bootstrap procedure is obtained by smoothing the periodogram operator  using  the Bartlett-Priestley spectral window with different values of the bandwidth $h$, i.e., the spectral window $ K_h(x)=3\big(1- x^2/(\pi^2h^2)\big)/(4\pi h)$ for $ |x|\leq \pi h$ and $ K_h(x)=0$ elsewhere. 
   To reduce computational costs, some simplifications in implementing the bootstrap procedure have also been made. In particular, 
   the subsampling parameter $b$ was set  in all cases equal to $32$.
 All bootstrap estimates  are based on $ B=300$ bootstrap repetitions while  $100$ trials have  been evaluated.  
  
   \begin{figure}[t]
\begin{center}  
\includegraphics[angle=0,height=5.5cm,width=5.4cm]{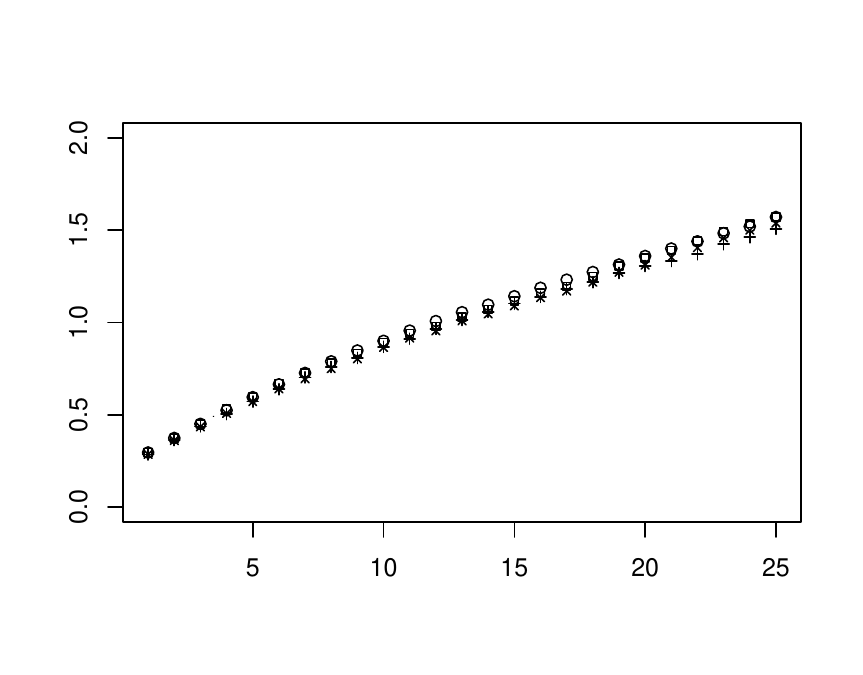},
\includegraphics[angle=0,height=5.5cm,width=5.4cm]{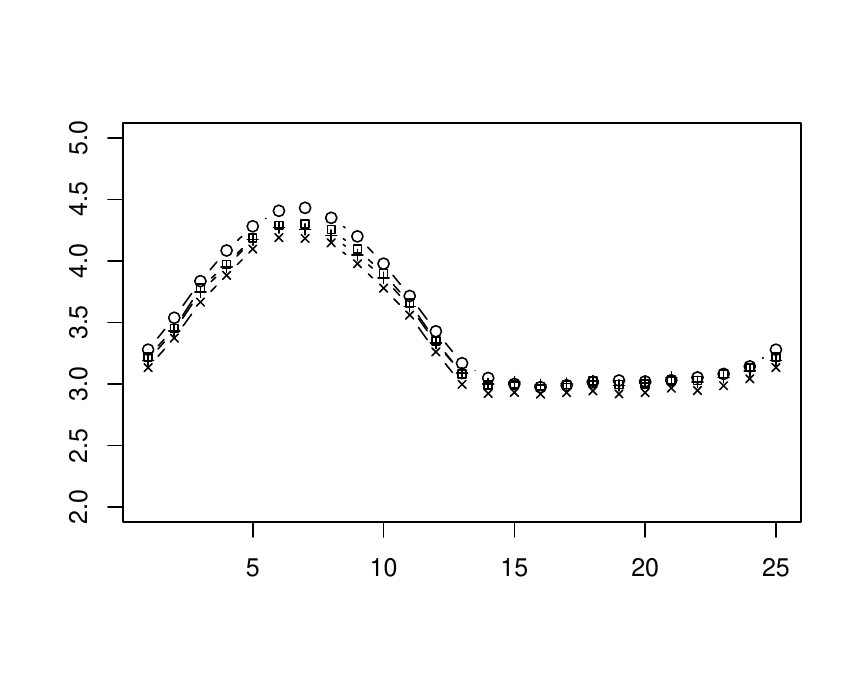}\\
\includegraphics[angle=0,height=5.5cm,width=5.4cm]{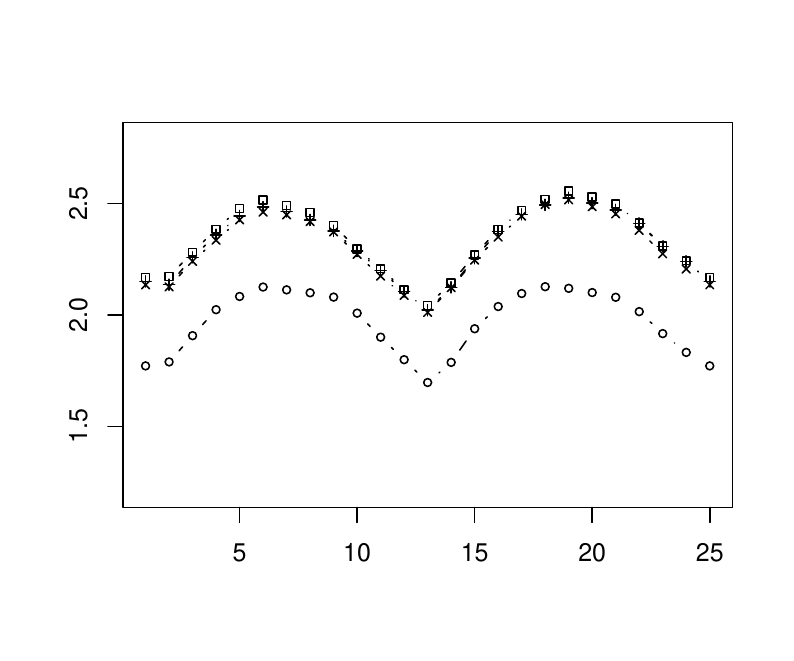},
\includegraphics[angle=0,height=5.5cm,width=5.4cm]{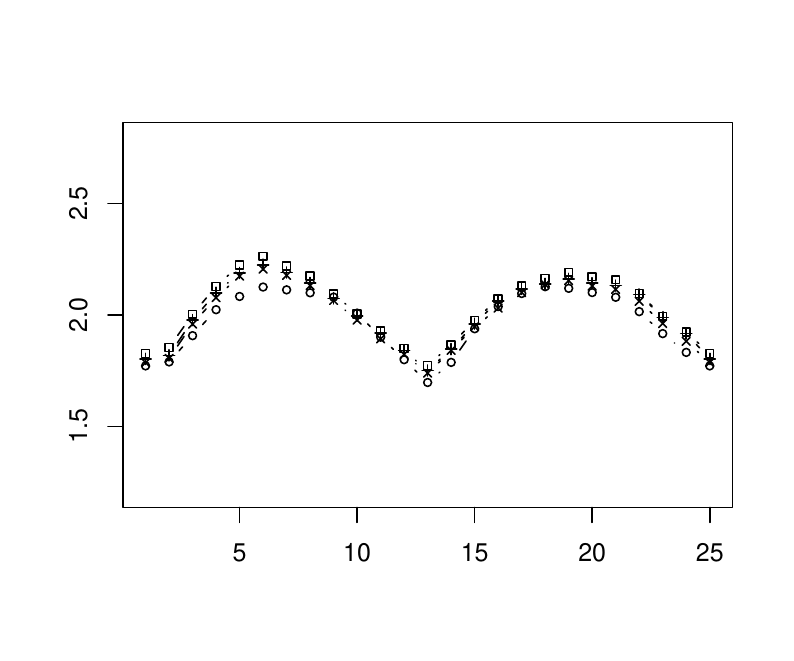}
\end{center}
\vspace*{-1.3cm}\caption{Estimated exact and (means)  of bootstrap   estimates  of $ Std(\tau_r)$ (vertical axes)  for  $ r=1,2, \ldots, 25$ (horizontal axes). Reading clockwise: First plot: Model I. The symbol "o"  refers to the estimated exact  values of $ Std(\tau_r)$, while  the  symbols  "$\square$", "$+$" and  "$ \times$",  refer to the means of three  bootstrap estimates based on  bootstrap proposal $ L_n^o$ and   bandwidths $h=0.08$, $h=0.10$ and $ h=0.14$, respectively. 
Second plot: Same as the first plot  but for Model II. Third plot:  Same as the first plot but for  Model III and for bandwidths $ h=0.14$, $h=0.20$ and $h=0.25$. Fourth plot:  As third plot  but for  bootstrap 
estimates based on the non corrected bootstrap  proposal  $ L^\star_n$.    \label{fig.ExaBoot1}
}
\end{figure}

\begin{scriptsize}
\begin{table}[t]
\begin{center}
\setlength{\tabcolsep}{2.6mm}
\renewcommand{\arraystretch}{1.2}
\begin{tabular}{|c|lllllllll|}
\hline
  &  &       &    &   &       &      &       &       &           \\
   &  \multicolumn{3}{l}{{\bf Model I}}      &   \multicolumn{3}{l}{{\bf Model II}}  &       \multicolumn{3}{l|}{{\bf Model III}}      \\
  &  &       &    &   &       &      &       &       &           \\
 & Exa       & BMean      & BStd      &  Exa       & BMean       & BStd    &  Exa     &  BMean      & BStd     \\
\hline
 &  &       &    &   &       &      &       &       &          \\
$\tau_1$   & 0.296 &  0.286     &   0.081  & 3.279 &3.134  &     0.747 &      1.771 &   1.826    &    0.310       \\
$\tau_2$     &  0.375 &    0.361   &  0.104  & 3.538 & 3.373 &    0.797   &   1.789   &    1.853   &     0.345             \\
$\tau_3$     &  0.451&      0.435   &  0.126 &3.836  &3.665 &   0.913  &   1.907  &   2.002    &           0.388      \\
$\tau_4$      &  0.524&       0.508 &  0.153 & 4.085  &3.883  &  0.954 &   2.024   &   2.128    &           0.409      \\
$\tau_5$          & 0.596 &     0.571  &  0.161  & 4.281 &4.096   &  1.022     &  2.083    &  2.224     &       0.432           \\
$\tau_6$    & 0.665 &      0.641 & 0.201 & 4.407 & 4.189 &    1.049   &    2.125 &    2.263   &         0.446         \\
$\tau_7$    & 0.727 &    0.695   & 0.215   & 4.431 &4.182  &   1.028    &   2.113   &   2.220    &        0.430          \\
$\tau_8$     &  0.789 &     0.750  &   0.228 &4.350  &4.147 &  1.010     &  2.099    &  2.175     &        0.428         \\   
$\tau_9$      &  0.848 &     0.803  &  0.237  & 4.199 & 3.978 &   0.924    &  2.080    &   2.095    &        0.403          \\  
$\tau_{10}$         & 0.900 & 0.865      & 0.252   & 3.978 & 3.779  & 0.883      & 2.008     & 2.005       & 0.387                  \\
$\tau_{11}$       & 0.956 &    0.916   &   0.270 & 3.715 & 3.560 &     0.867  &    1.900  &   1.929    &     0.371            \\  
$\tau_{12}$        &1.008  &    0.955   &   0.285 & 3.428 & 3.261 &   0.786    &    1.799  &   1.849    &          0.333       \\
$\tau_{13}$        & 1.056 &   1.009    &  0.309  & 3.169  & 2.998 &    0.715   &   1.697   &    1.773   &           0.316       \\
$\tau_{14}$       & 1.096 &   1.046    &  0.317  & 3.049  & 2.923 &    0.647   &   1.787   &    1.866   &             0.319     \\
$\tau_{15}$        & 1.142 &   1.089    &  0.319  &  3.001 & 2.932 &   0.638    &   1.938   &    1.975   &         0.387         \\
$\tau_{16}$          & 1.188 &   1.134    &   0.337 & 2.974 & 2.919  &   0.637    &    2.038  &   2.073    &     0.407            \\
$\tau_{17}$         & 1.232 &    1.170   &   0.353 & 2.990 & 2.930 &    0.639   &   2.096   &     2.129  &     0.402           \\
$\tau_{18}$     & 1.273 &   1.218    &  0.369  & 3.015 & 2.945 &    0.652   &   2.127   &     2.163  &           0.409       \\
$\tau_{19}$         & 1.313 &  1.271     &  0.406  & 3.029  & 2.929  &  0.626     &  2.119    &   2.191    &     0.408             \\
$\tau_{20}$       &  1.360&   1.311    &  0.422  &  3.021 & 2.930  &    0.592   &   2.100   &    2.172   &       0.419          \\
$\tau_{21}$        & 1.400 &    1.352   &0.427    &  3.030 &  2.966  &  0.591     &   2.079   &  2.158     &      0.407            \\
$\tau_{22}$     & 1.439 &    1.405   &  0.431  & 3.054 & 2.945 &    0.570   &    2.015  &      2.096 &      0.378            \\
$\tau_{23}$         & 1.483 &  1.458     & 0.440   & 3.082 & 2.986  &0.625       &   1.916   &   1.994    &    0.323              \\
$\tau_{24}$       & 1.520 &   1.501    &  0.440  & 3.143 & 3.043  &    0.654   &    1.832  &    1.924   &      0.274            \\
$\tau_{25}$        & 1.571 &    1.538   & 0.425   &  3.280 &3.134  &   0.747    &   1.771   &     1.826  &     0.310             \\
  &  &       &    &   &       &      &       &       &           \\
 \hline
\end{tabular}
\hspace*{0.3cm}
\vspace*{2mm}
\caption{Estimated exact (Exa) and bootstrap   estimates  of $ Std(\tau_r)$  for  $ r=1,2, \ldots, 25$. BMean and BStd,  refer to the means and to the standard deviations of the bootstrap estimates  based on bootstrap proposal $ L_n^o$, over  100 trials for the bandwidth $h=0.14$.    \label{tab.Table1}
}
\end{center}
\end{table}
\end{scriptsize}
  
   The results obtained are shown in Figure~\ref{fig.ExaBoot1}.  Note that   three of the four plots shown  in  Figure~\ref{fig.ExaBoot1}  refer to means of the bootstrap estimates over the 100 trials considered,  and each plot  corresponds to one of the   three different models used in the simulation experiment. The estimates shown in  these three plots  are    
    based on  the bootstrap proposal $L_n^o$ (see Step 7 of the algorithm) and  they are presented for 
    three    different values of the bandwidth $h$.
   As these exhibits show,  the bootstrap  estimates behave very well  for the    range of values of the bandwidth $ h$ considered. They are  quite close to the (estimated) exact values of the target  function $Std(\cdot)$ and  they closely follow  the behavior  of this function in the interval $[0,1]$.   A more detailed  presentation  of some  of the results obtained is given in Table~\ref{tab.Table1}. This table  shows  means and standard deviations of the bootstrap estimates over the 100 trials.  
   In order to quantify   the  importance  of the  contribution of the  procedure based on convolved periodogram operators of subsamples and the correction made in Step 6 of the bootstrap algorithm,   the fourth plot of Figure~\ref{fig.ExaBoot1} shows   (mean)  bootstrap estimates  which are  solely based on the  non-corrected bootstrap estimator  $L^\star_n$ (see Step 3 of the same algorithm). The improvements achieved by  using the convolved subsampling procedure to complement  the bootstrap estimates $L_n^\star$  for the missing fourth order terms,  are   clearly seen by  comparing  the two plots displayed in  the last row of Figure~\ref{fig.ExaBoot1}.\\

 \noindent {\bf Acknowledgements:} The authors thank Hanlin Shang for his help in the implementation of the R code used in the simulations of Section~\ref{sec.NumRes}\\
 
 \noindent {\bf Funding:}   This research was partly   funded  by  the Cyprus Academy of Sciences, Letters,   and  Arts, and   partly by the  Austrian Science Foundation (FWF), \\ $ \backslash {\rm ead}$\{https://doi.org/10.55776/P35520\}. For the purpose of open access, the authors have applied a CC BY public copyright licence to any Author Accepted Manuscript version arising from this submission.



\newpage

\section{Supplementary Material}

\subsection{Auxiliary lemmas}

\begin{Lem}\label{le.Mom}
 Let   $ X\in {\mathcal H}$ such that $ X\sim {\mathcal N}(0,\Gamma)$. Then for all $ p\in \N$ and $ \alpha>1$,
\[ \|\E \|X\|^{2p} \leq C_p\,  p! \, \|\Gamma\|_N^p, \]
with 
\[ C_p=\Big( -\frac{32\alpha (1+e)}{\log\big(e/(\alpha(1+e)-1)\big)}\Big)^p \Big(\sqrt{\frac{\alpha(1+e)-1}{e}}+ \frac{e^2}{e^2-1} \Big).\]
\end{Lem}
\begin{proof}
By Fernique's theorem,  see  \cite{Bogachev1998} Theorem 2.8.5, we have
\[ \E\big[ e^{\lambda \|X\|^2} \big] = \sum_{h=0}^\infty \frac{\lambda^h}{h!} \E \|X\|^{2h}  \leq e^{16\lambda r^2} + \frac{e^2}{e^2 -1},\]
if $\lambda$, $r>0$ satisfy 
\begin{equation} \label{eq.One-1}
\log \Big( \frac{1-P(X\in\overline{B}(0,r))}{P(X\in\overline{B}(0,r)}\Big) + 32\lambda r^2 \leq -1.
\end{equation}
In particular,
\begin{equation} \label{eq.One-1.1}
\E\|X\|^{2p} \leq \frac{p!}{\lambda^p} \Big(e^{16\lambda r^2} + \frac{e^2}{e^2-1} \Big).
\end{equation} 
By Markov's inequality we have
\[ 1 - P(X\in \overline{B}(0,r)) = P(\|X\|>r) \leq \frac{1}{r^2} \E\|X\|^2 = \frac{1}{r^2} \|\Gamma\|_N\]
and therefore, 
\[\frac{1-P(X\in\overline{B}(0,r))}{P(X\in \overline{B}(0,r))} \leq \frac{\|\Gamma\|_N /r^2}{1-\|\Gamma\|_N/r^2}.\]
This in turn implies that (\ref{eq.One-1}) is satisfied  if
\begin{align*}
&\log\Big(  \frac{\|\Gamma\|_N /r^2}{1-\|\Gamma\|_N/r^2}\Big) + 32 \lambda r^2 \leq -1\\
\iff &  \frac{\|\Gamma\|_N /r^2}{1-\|\Gamma\|_N/r^2} \leq {\rm exp}(-1) {\rm exp}(-32\lambda r^2) \\
\iff & \log\Big(  \frac{\|\Gamma\|_N /r^2}{(1-\|\Gamma\|_N/r^2){\rm exp}(-1)}\Big) \leq -32 \lambda r^2\\
\iff & \frac{1}{32r^2} \log\Big(  \frac{\|\Gamma\|_N /r^2}{(1-\|\Gamma\|_N/r^2) {\rm exp}(-1)}\Big) \leq -\lambda <0.
\end{align*}
However, 
\begin{align} \label{eq.One-2}
 & \log\Big( \frac{\|\Gamma\|_N /r^2}{(1-\|\Gamma\|_N/r^2) {\rm exp}(-1)} \Big) <0\\
 \iff & \|\Gamma\|_N/r^2 < (1-\|\Gamma\|_N/r^2) \frac{1}{e} = \frac{1}{e} - \frac{\|\Gamma\|_N}{r^2 e}\nonumber \\
 \iff & \frac{\|\Gamma\|_N}{r^2}\Big( 1 + \frac{1}{e}\Big) < \frac{1}{e} \nonumber \\
 \iff &  \sqrt{\|\Gamma\|_N (1+e)} <r. \nonumber
\end{align}
Thus, choosing $ r = \sqrt{\alpha \|\Gamma\|_N (1+e)} $ for some $ \alpha >1$ solves (\ref{eq.One-2}) and yields
\begin{align*}
\lambda \geq -\frac{1}{32 \alpha \|\Gamma\|_N (1+e)} \log\Big( \frac{1/(\alpha(1+e))}{(1-1/(\alpha(1+e))1/e}\Big),
\end{align*}
where the expression in the $\log()$ term simplifies to  $ e/\big(\alpha(1+e)-1\big)$.
Inserting these choices of $r$ and $\lambda$ into (\ref{eq.One-1.1}) yields the assertion 
\[ 16\lambda r^2 = -\frac{1}{2} \ln\Big(\frac{e}{\alpha(1+e)-1}\Big) \ \ \mbox{that is} \ \  e^{16\lambda r^2} = \sqrt{\frac{\alpha(1+e)-1}{e}}\]
and 
\[ \frac{p!}{\lambda^p} = \Big( -\frac{32\alpha(1+e)}{\ln\big(e/(\alpha(1+e)-1)\big)}\Big).\]
\end{proof}
\begin{Rem} Note that from Lemma~\ref{le.Mom} one gets using $\big(\E\|X\|^{2p-1}\big)^{1/(2p-1)} \leq  \big(\E\|X\|^{2p}\big)^{1/2p} $, that  
\begin{equation} \label{eq.Mom-1}
 \E\|X\|^{2p-1} \leq \big( \E\|X\|^{2p} \big)^{(2p-1)/2p} \leq  \big(C_p\, p!\big)^{(2p-1)/2p}\|\Gamma\|_N^{p-1/2}.
 \end{equation}
\end{Rem}

\vspace*{0.2cm}

\begin{Lem} \label{le.App-2}
Let $ (X_n,n\in\N)$ be a sequence  in ${\mathcal H}$  with $ X_n\sim {\mathcal C}{\mathcal N}(\mu_n,
 C_n,R_n) $ and let $ X\sim {\mathcal C}{\mathcal N}(\mu, C, R)$. If 
 \[ \|\mu_n-\mu\| \rightarrow 0, \ \|C_n-C\|_N \rightarrow 0, \ \mbox{and} \  \|R_n -R\|_N \rightarrow 0, \]
 then 
$X_n \stackrel{D}{\rightarrow} X$. 
 \end{Lem}
\begin{proof}
The result is proven as  Lemma D.10 in  the Supplementary Material of  
\cite{RademacherEtAl2024}.
\end{proof}

\vspace*{0.2cm}
\begin{Lem} \label{le.OperAlg}
Let $A,B,C,D,E,F \in {\mathcal L}(\HH)$. Then
\begin{enumerate}
\item[(i)] \  $\big(A\otimes_{op} E(C)\big)\otimes_{op}\big(B\otimes_{op}F(D)\big) = (A\otimes_{op}B)(C\otimes_{op}D)(E\otimes_{op} F)^\ast$.
\item[(ii)] \  \  $\big(A\otimes_{op} \overline{F}(C)\big)\otimes_{op}^\top \big(\overline{B}\otimes_{op}E(D)\big) = (A\otimes_{op}B)(C\otimes_{op}^\top D)(E\otimes_{op} F)^\ast$.
\end{enumerate}
Let $f,g,u,v\in{\mathcal H}$. Then
\begin{enumerate}
\item[(iii)] \  $(f\otimes g)\otimes_{op}(u\otimes v) =(f\otimes u)\otimes(g\otimes v)$.
\item[(iv)] \  $(f\otimes g) \otimes_{op}^\top (u\otimes v) = (f\otimes \overline{u})\otimes( v \otimes \overline{g})$.
\end{enumerate}
 \end{Lem}
\begin{proof}
For (i) and (ii), see  Remark D. 26 and D. 31  and for (iii) and (iv) see Lemma  D.30 and Lemma D.25,  in the Supplement File of  
\cite{RademacherEtAl2024}.
\end{proof}

\vspace*{0.2cm} 

The following lemma deals with the convergence properties of the inverse and the square root operators  for a particular class of Hilbert-Schmidt operators.

\begin{Lem}\label{le.inverse}
 Let  $ S \in HS({\mathcal H}) $  and $ (S_n,n\in\N ) $ be a sequence of  random, nonnegative definite   operators in $ HS({\mathcal H})$,   such that $ |\sigma(S_n) | =|\sigma(S)| =m$, $m\in \N$,  and $ 0\notin \sigma(S_n) \cup \sigma(S)$.
 If  $ \|S_n-S\|_{HS} \stackrel{P}{\rightarrow} 0$,
 as $ n \rightarrow \infty$, then,
\begin{enumerate}
\item[(i)]  \  $\|S_n^{1/2}-S^{1/2}\|_{HS} \stackrel{P}{\rightarrow} 0$.
\item[(ii)] \ $\|S_n^{-1}-S^{-1}\|_{HS} \stackrel{P}{\rightarrow} 0$ and
\item[(iii)] \ $\|S_n^{-1/2} -S^{-1/2}\|_{HS} \stackrel{P}{\rightarrow} 0$.
\end{enumerate}
\end{Lem} 
\begin{proof}
Consider (i). Since $ \sqrt \cdot \in C([0,+\infty))$, by Theorem 2.21 of \cite{Simons2005} it  suffices to show that $\|S^{1/2}_n\|_{HS} \stackrel{P}{\rightarrow} \|S^{1/2}\|_{HS}$ and  that $ \big|\langle S_n^{1/2}(u),u\rangle - \langle S^{1/2}(u),u\rangle\big| \stackrel{P}{\rightarrow} 0$ for every $u\in{\mathcal H}$. 
We have  using the notation    $ \sigma_j(S_n)$ and $\sigma_j (S) $ the $j$-th  eigenvalue of $S_n $ and $S $, respectively,  that
\begin{align*}
 \|S^{1/2}_n\|^2_{HS}  = \|S_n\|_N = \sum_{j=1}^m \sigma_j(S_n) \stackrel{P}{\rightarrow} \sum_{j=1}^m \sigma_j(S) = \|S\|_N = \|S^{1/2}\|_{HS}^2 .  
\end{align*}
Furthermore,  employing functional calculus, 
\begin{align*}
\big|\langle S_n^{1/2}(u),u\rangle - \langle S^{1/2}(u),u\rangle\big| & \leq \|S_n^{1/2} -S^{1/2}\|_{\mathcal L} \|u\|^2 \leq \|S_n-S\|_{\mathcal L}^{1/2} \|u\|^2 \stackrel{P}{\rightarrow} 0.
\end{align*}
For  (ii)  we have  by the assumption that $ S_n$ is bounded from bellow in probability,  that $S$ is bounded from bellow, that $ \|S_n^{-1}\|_{\mathcal L} = {\mathcal O}_P(1)$ and   that $ \|S^{-1}\|_{\mathcal L} = {\mathcal O}(1)$. Then,  
\begin{align*}
\|S_n^{-1}-S_n^{-1}\|_{HS}& = \|S_n^{-1}(S_n-S)S^{-1} \|_{HS}\\
& \leq \|S_n^{-1}\|_{\mathcal L} \|S_n - S\|_{HS} \|S^{-1}\|_{\mathcal L}\\
& = {\mathcal O}_P( \|S_n - S\|_{HS}) \rightarrow 0. 
\end{align*}
Finally (iii) is proved along the same lines as  (i)  by using (ii).
\end{proof}

The following lemma  quantifies   the error made when  applying  the estimated projection operator $\widehat{\PP}_m$ instead of $ \PP_m$ in the bootstrap procedure.

\begin{Lem} \label{le.App-3}
Let $ m \in \N$  and    $ \sigma_1>\sigma_2 > \cdots \sigma_m >0$.  Denote by $ \widetilde{T}^\star_{n,m}$ and $ \widetilde{Q}_{n,m}^\star$  the same bootstrap statistics as $  T^\star_{n,m}$ and $ Q_{n,m}^\star$
 obtained using the projection operators $ \PP_m=\sum_{j=1}^m v_j \otimes v_j$, respectively, ${\rm  Id}-\PP_m$,  instead of  $ \widehat{\PP}_m=\sum_{j=1}^m \widehat v_j \otimes  \widehat v_j$, respectively, $ {\rm Id}-\widehat{\PP}_m$. Then,
\begin{enumerate}
\item[(i)] \  $  \|  T^\star_{n,m} -  \widetilde{T}^\star_{n,m}\|_{HS} =  {\mathcal O}_{P}\Big(n^{-1/2}\sqrt{\sum_{j=1}^m \alpha_j^{-2}}\Big)$,
\item[(ii)] \  $  \| Q^\star_{n,m} - \widetilde{Q}^\star_{n,m} \|_{HS} =    {\mathcal O}_{P}\Big(n^{-1/2}\sqrt{\sum_{j=1}^m \alpha_j^{-2}}\Big)$,
\item[(iii)] \  $  \| {\rm Cov}^\star( T^\star_{n,m} ) - {\rm Cov}^\star( \widetilde{T}^\star_{n,m})\|_{HS} =    {\mathcal O}_{P}\Big(n^{-1/2}\sqrt{\sum_{j=1}^m \alpha_j^{-2}}\Big)$, \\  \hspace*{0cm} 
$  \|{\rm Rel}^\star( T^\star_{n,m} ) - {\rm Rel}^\star( \widetilde{T}^\star_{n,m})\|_{HS} =   {\mathcal O}_{P}\Big(n^{-1/2}\sqrt{\sum_{j=1}^m \alpha_j^{-2}}\Big)$,
\item[(iv)] \  $ \| {\rm Cov}^\star( Q^\star_{n,m} ) - {\rm Cov}^\star( \widetilde{Q}^\star_{n,m})\|_{HS}=  {\mathcal O}_{P}\Big(n^{-1/2}\sqrt{\sum_{j=1}^m \alpha_j^{-2}}\Big)$,  \\  \hspace*{0cm} 
$ \| {\rm Rel}^\star( Q^\star_{n,m} ) - {\rm Rel}^\star( \widetilde{Q}^\star_{n,m}) \|_{HS} =   {\mathcal O}_{P}\Big(n^{-1/2}\sqrt{\sum_{j=1}^m \alpha_j^{-2}}\Big)$,
\end{enumerate}
where $\alpha_1=\sigma_1-\sigma_2$ and $ \alpha_j=\min\{\sigma_{j-1}-\sigma_j, \sigma_j-\sigma_{j+1}\}$ {\rm for}  $ 2\leq j \leq m$,
 \end{Lem}
\begin{proof} 
(i). Recall  that $ \|\widehat{\PP}_m\|_{\mathcal L}=  \| \PP_m\|_{\mathcal L} = 1$. Furthermore, 
\begin{align} \label{P-diff}
\|\widehat{\PP}_m-\PP_m\|_{\mathcal L} & \leq \|\sum_{j=1}^m  (\widehat v_j-v_j)\otimes  v_j \|_{\mathcal L} + \|\sum_{j=1}^m \widehat v_j\otimes (\widehat{v}_j-v_j) \|_{\mathcal L}\nonumber \\
& \leq  2\sum_{j=1}^m \|\widehat{v}_j -v_j\| \nonumber \\
 & \leq C\frac{1}{\sqrt{n}} \sqrt{\sum_{j=1}^m \frac{1}{\alpha^2_j}},
\end{align}
for some  $C>0$, where  for the last inequality see   \cite{HoermannKokoszka2010}, p. 1857. Since  $  T^\star_{n,m} = \widehat{\PP}_mL_n^\star \widehat{\PP}_m$, where  
\[  L_n^\star=\frac{2\pi}{ \sqrt{n}} \sum_{j\in {\mathcal G}(n) } W(\lambda_{j,n}) \big( P^\star_{n,\lambda_{j,n}}-\widehat{\F}_{\lambda_{j,n}} \big), \ \ \mbox{and} \ \    P^\star_{n,\lambda_{j,n}}=J^\star_n(\lambda_{j,n}) \otimes J^\star_n(\lambda_{j,n}),\]
we immediately  get using (\ref{P-diff}),  $\|L_n^\star \|_{HS}={\mathcal O}_P(1) $, see Lemma~\ref{le.App-CLT-1},    and 
\[  \|(\widehat{\PP}_m-\PP_m) L_n^\star \widehat{\PP}_m \|_{HS} \leq  \|\widehat{\PP}_m-\PP_m \|_{\mathcal L} \|L_n^\star \|_{HS},  \]
that 
\begin{align} \label{eq.Tstar-tilde}
\|T_{n,m}^\star-\widetilde{T}_{n,m}^\star\|_{HS}  & \leq 2 \|\widehat{\PP}_m- \PP_m\|_{\mathcal L} \|L_n^\star\|_{HS} = {\mathcal O}_P\Big( n^{-1/2}\sqrt{\sum_{j=1}^m \alpha_j^{-2}}\Big). 
\end{align}
Assertion (ii) follows by the same arguments and because
 \[  Q^\star_{n,m} = L_n^\star -T^\star_{n,m} =   \widetilde{Q}_{n,m}^\star + (\widetilde{T}_{n,m}^\star - T_{n,m}^\star).\]
For (iii) and (iv) we only show $  {\rm Cov}( T^\star_{n,m} ) = {\rm Cov}( \widetilde{T}^\star_{n,m}) + {\mathcal O}_{P}\Big(n^{-1/2}\sqrt{\sum_{j=1}^m \alpha_j^{-2}}\Big)$ with respect to $\|\cdot\|_{HS}$, since the proofs of the other  statements follow  along the same lines. We have  using 
expression (\ref{eq.Tstar-tilde}) that  
\begin{align*}
{\rm Cov}^\star(T_{n,m}^\star) & = \E^\star T^\star_{n,m} \otimes T^\star_{n,m} \\
& = {\rm Cov}^\star(\widetilde{T}_{n,m}^\star)  + \E^\ast (T^\star_{n,m}- \widetilde{T}_{n,m}^\star) \otimes \widetilde{T}_{n,m}^\star \\
& \ \  \ \ + \E^\ast \widetilde{T}_{n,m}^\star \otimes (T^\star_{n,m}- \widetilde{T}_{n,m}^\star)  + \E^\star  (T^\star_{n,m}- \widetilde{T}_{n,m}^\star) \otimes (T^\star_{n,m}- \widetilde{T}_{n,m}^\star) \\
& =  {\rm Cov}^\star(\widetilde{T}_{n,m}^\star)  + {\mathcal O}_P\Big(n^{-1/2}\sqrt{\sum_{j=1}^m \alpha_j^{-2}}\Big) 
\end{align*}
since 
\begin{align*}
 \|\E^\ast (T^\star_{n,m}- \widetilde{T}_{n,m}^\star) \otimes \widetilde{T}_{n,m}^\star\|_{HS}  & \leq \| \widehat{\PP}_m-\PP_m\|_{HS}  \E^\star \|L_n^\star\|_{HS}^2\|\PP_m\|_{HS}\\
 & = {\mathcal O}_P\Big(n^{-1/2}\sqrt{\sum_{j=1}^m \alpha_j^{-2}}\Big),
 \end{align*}
 and similarly for $ \|\E^\ast (T^\star_{n,m}- \widetilde{T}_{n,m}^\star) \otimes T_{n,m}^\star\|_{HS}  $.
\end{proof}


The following two lemmas are used  in  the   proof of  Lemma~\ref{Lem:CBPConsistencyCovariance}.
 
\begin{Lem}\label{Lem:AsymptoticCovarianceRiemannApprox}
Granted Assumptions \ref{as.1}, \ref{as.W}  and \ref{as.3} are satisfied and let 
\begin{align*}
	W_{t,b} := \frac{2\pi}{\sqrt{b}} \sum_{j\in\G(b)} W(\lambda_{j,b}) (P_{\lambda_{j,b}}^{t,b} - \widetilde{F}_{\lambda_{j,b}}).
\end{align*}
Then
\begin{align*}
	\lim_{b\to\infty} \|\Cov(W_{1,b}) - \Psi\|_{N} = 0 \qquad\text{and}\qquad \lim_{b\to\infty} \| {\rm Rel}(W_{1,b}) - \Upsilon \|_{N} = 0. 
\end{align*}
\end{Lem}
\begin{proof}
	For the covariance operator of $W_{1,b}$ we have the following decomposition   
	\begin{align}\label{Eq:CovarianceRiemannApprox}
	\Cov(W_{1,b}) 
	=& \frac{(2\pi)^2}{b} \sum_{j,i \in \G(b)} W(\lambda_{j,b})  \overline{W(\lambda_{i,b})} 
	\Cov( P_{\lambda_{j,b}}^{1,b} - \widetilde{\F}_{\lambda_{j,b}}, P_{\lambda_{i,b}}^{1,b} - \widetilde{\F}_{\lambda_{i,b}} )\nonumber\\
	=& \frac{(2\pi)^2}{b} \sum_{j,i \in \G(b)} W(\lambda_{j,b})  \overline{W(\lambda_{i,b})} \Big\{ 
	\Cov( P_{\lambda_{j,b}}^{1,b} - \F_{\lambda_{j,b}} , P_{\lambda_{i,b}}^{1,b} - \F_{\lambda_{i,b}} )\nonumber\\
	&+ \Cov( F_{\lambda_{j,b}} - \widetilde{\F}_{\lambda_{j,b}} , P_{\lambda_{i,b}}^{1,b} - \F_{\lambda_{i,b}} )
	+ \Cov( P_{\lambda_{j,b}}^{1,b} - \widetilde{\F}_{\lambda_{j,b}} , \F_{\lambda_{i,b}} - \widetilde{\F}_{\lambda_{i,b}}  )\nonumber\\
	=&  \frac{(2\pi)^2}{b} \sum_{j,i \in \G(b)} W(\lambda_{j,b})  \overline{W(\lambda_{i,b})} 
	\left\{ \Cov (P_{\lambda_{j,b}}^{1,b}, P_{\lambda_{i,b}}^{1,b}) + A_{j,i} + B_{j,i} \right\}
	\end{align}
	The term $A_{j,i}$ is bounded by 
	\begin{align*}
	\|A_{j,i}\|_{N} \leq \sup_{\lambda, \omega\in[-\pi,\pi]} \sqrt{\| \Cov( \widetilde{\F}_{\lambda} ) \|_{N} \| \Cov(P_{\omega}^{1,b}) \|_{N}}. 
	\end{align*}
	and it follows from Theorem~4.4 in \cite{RademacherEtAl2024},  that $\sup_{\omega \in[-\pi,\pi]}\| \Cov(P_{\omega}^{1,b}) \|_{N} = \mathcal{O}(1)$.
	Furthermore the bounds for $G_1$,$G_2$ and $G_3$ in \eqref{Eq:PeriodogramMeanApproxOverFF} obtained in the proof of Lemma \ref{Lem:PeriodogramMeanAndCovarianceApproxOverFF} show that we have $\sup_{\lambda \in [-\pi,\pi]} \| \Cov(\widetilde{\F}_{\lambda}) \|_{N} = \mathcal{O}(b/N)$ and since
	by assumption $b^3/n \to 0$ we can conclude
	\begin{align*}
	\|A_{j,i}\|_{N} = \mathcal{O}(\sqrt{b/N}) = o(1/b).
	\end{align*}
	For the term $B_{j,i}$ we have
	\begin{align*}
	\|B_{j,i}\|_{N} \leq \sup_{\lambda, \omega\in[-\pi,\pi]} \sqrt{\| \Cov( \widetilde{\F}_{\lambda}  ) \|_{N} \| \Cov(P_{\omega}^{1,b}- \widetilde{\F}_{\omega}) \|_{N}}.
	\end{align*}
	Because 
	\begin{align*}
	\| \Cov(P_{\omega}^{1,b}- \widetilde{\F}_{\omega}) \|_{N} 
	\leq& \| \Cov( P_{\omega}^{1,b} ) \|_{N} + 2\| \Cov( P_{\omega}^{1,b},\widetilde{\F}_{\omega})  \|_{N} + \| \Cov(\widetilde{\F}_{\omega}) \|_{N}\\
	\leq& \| \Cov( P_{\omega}^{1,b} ) \|_{N} + 2 \sqrt{ \| \Cov( P_{\omega}^{1,b} ) \|_{N} \| \Cov(\widetilde{\F}_{\omega}) \|_{N}}\\ 
	&+ \| \Cov(\widetilde{\F}_{\omega}) \|_{N}\\
	=& \left( \sqrt{\| \Cov( P_{\omega}^{1,b} ) \|_{N(HS)}} + \sqrt{\| \Cov(\widetilde{\F}_{\omega}) \|_{N}} \right)^2
	= \mathcal{O}(1)
	\end{align*}
	we can also similarly conclude $\|B_{j,i}\|_{N} = o(1/b)$. Using $\sup_{\lambda\in[-\pi,\pi]} |W(\lambda)| <\infty$, $|\G(b)|=\mathcal{O}(b)$
	and the obtained results for $A_{j,i}$ and $B_{j,i}$, it is apparent from \eqref{Eq:CovarianceRiemannApprox} that
	\begin{align*}
	\Cov(W_{1,b})
	=& \frac{(2\pi)^2}{b} \sum_{j,i \in \G(b)} W(\lambda_{j,b})  \overline{W(\lambda_{i,b})} 
	\Cov (P_{\lambda_{j,b}}^{1,b}, P_{\lambda_{i,b}}^{1,b}) + o(1)\\
	=& \frac{(2\pi)^2}{b} \Bigg\{ \sum_{j \in \G(b)} W(\lambda_{j,b})  \overline{W(\lambda_{j,b})} \Cov (P_{\lambda_{j,b}}^{1,b}, P_{\lambda_{j,b}}^{1,b})  \\
	 & \ \ + W(\lambda_{j,b})  \overline{W(-\lambda_{j,b})} \Cov (P_{\lambda_{j,b}}^{1,b}, P_{-\lambda_{j,b}}^{1,b})\nonumber\\
	&\ \ + \sum_{j \in \G(b)} \sum_{i \in \G(b)\atop i\not= \pm j} W(\lambda_{j,b})  \overline{W(\lambda_{i,b})} \Cov (P_{\lambda_{j,b}}^{1,b}, P_{\lambda_{i,b}}^{1,b})
	\Bigg\} + o(1)\\
	=& R_1 + R_2 + R_3 + o(1) 
	\end{align*}
	with respect to $\| \cdot \|_{N}$. 
	Concerning $R_1$ and $R_2$ we can invoke Theorem~4.4 of 
	\cite{RademacherEtAl2024},  to see (with respect to $\| \cdot \|_{N}$)
	\begin{align*}
	\Cov (P_{\lambda_{j,b}}^{1,b}, P_{\lambda_{j,b}}^{1,b}) &= \F_{\lambda_{j,b}} \tsr_{op} \F_{\lambda_{j,b}} + \mathcal{O}(b^{-1}),\\
	\Cov (P_{\lambda_{j,b}}^{1,b}, P_{-\lambda_{j,b}}^{1,b}) &= \F_{\lambda_{j,b}} \tsr_{op}^T \F_{-\lambda_{j,b}} + \mathcal{O}(b^{-1}),
	\end{align*}
	where the $\mathcal{O}(b^{-1})$ terms are uniform in $j\in\G(b)$.  Since $|\G(b)| = b$ it follows 
	\begin{align}\label{Eq:RiemannSumR1}
	R_1 =  \frac{(2\pi)^2}{b}  \sum_{j \in \G(b)} W(\lambda_{j,b})  \overline{W(\lambda_{j,b})} \F_{\lambda_{j,b}} \tsr_{op} \F_{\lambda_{j,b}} 
	+ \mathcal{O}(b^{-1}).
	\end{align}
	Now continuity of $\lambda \mapsto \F_{\lambda} \in N(\HH)$ and 
	\begin{align*}
	\| \F_{\lambda}\tsr_{op} \F_{\lambda} - \F_{\omega}\tsr_{op} \F_{\omega} \|_{N}
	\leq ( \|\F_{\lambda}\|_N + \|\F_{\omega}\|_N )\|\F_{\lambda} - \F_{\omega}\|_N
	\end{align*}
	imply continuity of $\lambda \mapsto \F_{\lambda} \tsr_{op} \F_{\lambda} \in N(HS)$. Hence  and in view of Assumption~\ref{as.W}, 
	it follows 
	that the 
	Riemann sum in \eqref{Eq:RiemannSumR1} converges to the Bochner integral, 
	i.e. 
	\begin{align*}
	\lim_{b\to\infty} R_1 = 2\pi \int_{-\pi}^{\pi} W(\lambda)\overline{W(\lambda)} \F_{\lambda} \tsr_{op} \F_{\lambda} \,d\lambda.
	\end{align*} 
	By the same argument we have as well 
	\begin{align*}
	\lim_{b\to \infty}R_2 = 2\pi \int_{-\pi}^{\pi} W(\lambda)\overline{W(-\lambda)} \F_{\lambda} \tsr_{op} \F_{-\lambda} \,d\lambda.
	\end{align*}
	Concerning the last term $R_3$, again by Theorem~4.4 in \cite{RademacherEtAl2024},  we have 
	\begin{align*}
	\Cov (P_{\lambda_{j,b}}^{1,b}, P_{\lambda_{i,b}}^{1,b}) 
	= \frac{2\pi}{b} \F_{\lambda_{j,b},-\lambda_{j,b},-\lambda_{i,b}} + \mathcal{O}(b^{-2})
	\end{align*}
	where the $\mathcal{O}(b^{-2})$ term is uniform in $j,i\in\G(b)$. Therefore 
	\begin{align} \label{eq.Eq:RiemannSumR1-2}
	R_3 = \frac{(2\pi)^3}{b^2}\sum_{j \in \G(b)} \sum_{i \in \G(b)\atop i\not= \pm j} W(\lambda_{j,b})  \overline{W(\lambda_{i,b})}
	\F_{\lambda_{j,b},-\lambda_{j,b},-\lambda_{i,b}} + \mathcal{O}(b^{-2})
	\end{align}
	and the continuity of $(\lambda,\omega) \to \F_{\lambda,-\lambda,\omega} \in N(HS)$ (see Proposition~4.1 in 
	\cite{RademacherEtAl2024}) again yields
	\begin{align*}
	\lim_{b\to\infty} R_3=  2\pi \int_{-\pi}^{\pi} \int_{-\pi}^{\pi} W(\lambda)\overline{W(-\omega)} \F_{\lambda,-\lambda,-\omega} \,d\lambda \,d\omega.
	\end{align*}
	Combining these results shows $\lim_{b\to\infty} \|\Cov(W_{1,b}) - \Psi\|_{N} = 0$. For the relation  operator notice that 
	\begin{align*}
	{\rm Rel}( W_{1,b} ) & = \E[ W_{1,b} \tsr \overline{W_{1,b}} ] \\
	& =\frac{(2\pi)^2}{b} \sum_{j,i \in \G(b)} W(\lambda_{j,b})  W(\lambda_{i,b})
	\E[ (P_{\lambda_{j,b}}^{1,b} - \overline{\F}_{\lambda_{j,b}}) \tsr (P_{\lambda_{i,b}}^{1,b} - \overline{\F}_{\lambda_{i,b}})], 
	\end{align*}
	so repeating the above calculations also yields $\|{\rm Rel}( W_{1,b} ) - \Upsilon\|_N \rightarrow 0$.
\end{proof}

\begin{Lem}\label{Lem:PeriodogramMeanAndCovarianceApproxOverFF}
Granted Assumptions \ref{as.1} and \ref{as.3} are satisfied. Then
	\begin{itemize}
		\item[(i)] $ \displaystyle \sum_{j \in \G(b)} \| \widetilde{\F}_{\lambda_{j,b}} - \E P_{\lambda_{j,b}}^{1,b}  \|_{HS} = \mathcal{O}_{P}(\sqrt{b^3/N})$;
		\item[(ii)] $ \displaystyle \sum_{j,i \in \G(b)} \bigg\|  \frac{1}{N} \sum_{t=1}^N P_{\lambda_{j,b}}^{t,b}\tsr P_{\lambda_{i,b}}^{t,b} 
		- \E[ P_{\lambda_{j,b}}^{1,b} \tsr P_{\lambda_{i,b}}^{1,b} ] \bigg\|_{HS} = \mathcal{O}_{P}(\sqrt{b^5/N})$.
	\end{itemize}
\end{Lem}
\begin{proof}
	(i) Similarly as in the proof of Theorem~4.4 in \cite{RademacherEtAl2024},   
	we see that for arbitrary frequencies $\lambda,\omega \in [-\pi,\pi]$
	the covariance operator of periodograms may be bounded by 
\begin{align}\label{Eq:BoundForCovarianceOfPeriodogram}
	&\| \Cov( P_{\lambda_{j,b}}^{t,b}, P_{\lambda_{j,b}}^{u,b} ) \|_{N} \nonumber \\
	\leq& \frac{1}{(2\pi b)^2} \sum_{r_1,s_1,r_2,s_2=1}^b \| \Cov( X_{t+r_1-1} \tsr X_{t+s_1-1},  X_{u+r_2-1} \tsr X_{u+s_2-1} ) \|_{N} \nonumber\\
	\leq& \frac{1}{(2\pi b)^2} \sum_{r_1,s_1,r_2,s_2=1}^b \| \cum(X_{t+r_1-1}, X_{t+s_1-1}, X_{u+r_2-1}, X_{u+s_2-1}  ) \|_{N} \nonumber\\
	&+ \| \Gamma_{|t-u+r_1-r_2|} \|_N\| \Gamma_{|t-u+s_1-s_2|} \|_N + \| \Gamma_{|t-u+r_1-s_2|} \|_N\| \Gamma_{|t-u+s_1-r_2|} \|_N.
\end{align}
Using this and $\E[\widetilde{\F}_{\lambda_{j,b}} ] = \E P_{\lambda_{j,b}}^{1,b}$, we obtain 
\begin{align}\label{Eq:PeriodogramMeanApproxOverFF}
	& \E\left[ \sum_{j \in \G(b)} \| \widetilde{\F}_{\lambda_{j,b}} - \E P_{\lambda_{j,b}}^{1,b}  \|_{HS} \right] \nonumber\\
	\leq& \sum_{j \in \G(b)} \left( \E \| \widetilde{\F}_{\lambda_{j,b}} - \E P_{\lambda_{j,b}}^{1,b}  \|_{HS}^2 \right)^{1/2}\nonumber\\
	=& \sum_{j \in \G(b)} \left( \| \Cov(\widetilde{\F}_{\lambda_{j,b}})  \|_{N} \right)^{1/2}\nonumber\\
	\leq&  \sum_{j \in \G(b)} \left( \frac{1}{N^2} \sum_{t,u=1}^N  \| \Cov( P_{\lambda_{j,b}}^{t,b}, P_{\lambda_{j,b}}^{u,b} )\|_{N} \right)^{1/2}\nonumber\\
	\leq& \sum_{j \in \G(b)} \Bigg( \frac{1}{4\pi b^2N^2} \sum_{t,u=1}^N \sum_{r_1,s_1,r_2,s_2=1}^b 
	\cum(X_{t+r_1-1}, X_{t+s_1-1}, X_{u+r_2-1}, X_{u+s_2-1}  ) \|_{N}\nonumber\\
	&+ \| \Gamma_{|t-u+r_1-r_2|} \|_N\| \Gamma_{|t-u+s_1-s_2|} \|_N + \| \Gamma_{|t-u+r_1-s_2|} \|_N\| 
	\Gamma_{|t-u+s_1-r_2|} \|_N \Bigg)^{1/2}\nonumber \\
	=:&  \sum_{j \in \G(b)} ( G_1 + G_2 + G_3 )^{1/2}.
	\end{align}
For $G_2$ note that by first summing over $s_2$ and then over $u$ that
	\begin{align*}
	& \sum_{t,u=1}^N \sum_{r_1,s_1,r_2,s_2=1}^b \| \Gamma_{|t-u+r_1-r_2|} \|_N\| \Gamma_{|t-u+s_1-s_2|} \|_N\\
	\leq& \sum_{t,u=1}^N \sum_{r_1,s_1,r_2=1}^b \| \Gamma_{|t-u+r_1-r_2|} \|_N \sum_{ s_2 \in\Z } \| \Gamma_{s_2} \|_N\\
	\leq& \sum_{t=1}^N \sum_{r_1,s_1,r_2=1}^b \sum_{ u \in\Z } \| \Gamma_{u} \|_N \sum_{ s_2 \in\Z } \| \Gamma_{s_2} \|_N\\
	=& Nb^3 \left( \sum_{h\in\Z} \| \Gamma_h \|_N \right)^2 
\end{align*} 
	and hence $G_2 = \mathcal{O}(b/N)$. The same argument also shows $G_3 = \mathcal{O}(b/N)$. Finally summing the $4$th order cumulant term 
	first over $r_2$, then over $r_1$ and then over $u$ we get 
\begin{align*}
	&\sum_{t,u=1}^N \sum_{r_1,s_1,r_2,s_2=1}^b \| \cum(X_{t+r_1-1}, X_{t+s_1-1}, X_{u+r_2-1}, X_{u+s_2-1}  ) \|_{N}\\
	=& \sum_{t,u=1}^N \sum_{r_1,s_1,r_2,s_2=1}^b \| \cum(X_{t-u+r_1-s_2}, X_{t-u+s_1-s_2}, X_{r_2-s_2}, X_{0}  ) \|_{N}\\
	\leq& \sum_{t,u=1}^N \sum_{r_1,s_1,s_2=1}^b \sum_{r_2\in\Z} \| \cum(X_{t-u+r_1-s_2}, X_{t-u+s_1-s_2}, X_{r_2}, X_{0}  ) \|_{N}\\
	\leq& \sum_{t,u=1}^N \sum_{s_1,s_2=1}^b \sum_{r_1\in\Z} \sum_{r_2\in\Z} \| \cum(X_{r_1}, X_{t-u+s_1-s_2}, X_{r_2}, X_{0}  ) \|_{N}\\
	\leq& \sum_{t=1}^N \sum_{s_1,s_2=1}^b \sum_{u \in\Z}  \sum_{r_1\in\Z} \sum_{r_2\in\Z} \| \cum(X_{r_1}, X_{u}, X_{r_2}, X_{0}  ) \|_{N}\\
	=& Nb^2 \sum_{t_1,t_2,t_3 \in \Z} \| \cum(X_{t_1}, X_{t_2}, X_{t_3}, X_{0}  ) \|_{N}
	\end{align*}
	and hence $G_3 \in \mathcal{O}(1/N)$. Since $|\G(b)| = b$ we can conclude from \eqref{Eq:PeriodogramMeanApproxOverFF}, that 
	\begin{align*}
	\E\left[ \sum_{j \in \G(b)} \| \widetilde{\F}_{\lambda_{j,b}} - \E P_{\lambda_{j,b}}^{1,b}  \|_{HS} \right]
	\leq b \mathcal{O}(\sqrt{b/n}) = \mathcal{O}(\sqrt{b^3/N})
	\end{align*} 
	and assertion (i) follows.\\
	(ii) First notice that due to 
	\begin{align*}
	&( P_{\lambda_{j,b}}^{t,b} - \E[P_{\lambda_{j,b}}^{t,b}] ) \tsr ( P_{\lambda_{i,b}}^{t,b} - \E[P_{\lambda_{i,b}}^{t,b}] )\\
	=& P_{\lambda_{j,b}}^{t,b} \tsr P_{\lambda_{i,b}}^{t,b} - P_{\lambda_{j,b}}^{t,b} \tsr \E[P_{\lambda_{i,b}}^{t,b}]
	- \E[P_{\lambda_{j,b}}^{t,b}] \tsr P_{\lambda_{i,b}}^{t,b} + \E[P_{\lambda_{j,b}}^{t,b}] \tsr \E[P_{\lambda_{i,b}}^{t,b}]
	\end{align*} 
	and 
	\begin{align*}
	&\E\left[ (P_{\lambda_{j,b}}^{t,b} - \E[P_{\lambda_{j,b}}^{t,b}]) \tsr ( P_{\lambda_{i,b}}^{t,b} - \E[P_{\lambda_{i,b}}^{t,b}]) \right]
	= \E[ P_{\lambda_{j,b}}^{t,b} \tsr P_{\lambda_{i,b}}^{t,b}  ] - \E[P_{\lambda_{j,b}}^{t,b}] \tsr \E[P_{\lambda_{i,b}}^{t,b}]
	\end{align*}
	we have the following decomposition 
	\begin{align*}
	&\frac{1}{N} \sum_{t=1}^N P_{\lambda_{j,b}}^{t,b}\tsr P_{\lambda_{i,b}}^{t,b} 
	- \E[ P_{\lambda_{j,b}}^{1,b} \tsr P_{\lambda_{i,b}}^{1,b} ]\\
	=& \frac{1}{N} \sum_{t=1}^N \E[P_{\lambda_{j,b}}^{t,b}]\tsr ( P_{\lambda_{i,b}}^{t,b} - \E[P_{\lambda_{i,b}}^{t,b}] )
	+ ( P_{\lambda_{j,b}}^{t,b} - \E[P_{\lambda_{j,b}}^{t,b}] )\tsr \E[P_{\lambda_{i,b}}^{t,b}]\\
	&+ \Big\{ ( P_{\lambda_{j,b}}^{t,b} - \E[P_{\lambda_{j,b}}^{t,b}] ) \tsr ( P_{\lambda_{i,b}}^{t,b} - \E[P_{\lambda_{i,b}}^{t,b}] ) 
	- \E\left[ (P_{\lambda_{j,b}}^{t,b} - \E[P_{\lambda_{j,b}}^{t,b}]) \tsr ( P_{\lambda_{i,b}}^{t,b} - \E[P_{\lambda_{i,b}}^{t,b}]) \right] \Big\}\\
	=& K_{j,i}^{(1)} + K_{j,i}^{(2)} + K_{j,i}^{(3)}.
	\end{align*}
	Summing the first term over the Fourier frequencies it follows from assertion (i), that 
	\begin{align*}
	\sum_{j,i \in \G(b)} \| K_{j,i}^{(1)} \|_{HS}
	=& \sum_{j,i \in \G(b)} \left\| \frac{1}{N} \sum_{t=1}^N \E[P_{\lambda_{j,b}}^{1,b}] 
	\tsr ( P_{\lambda_{i,b}}^{t,b} - \E[P_{\lambda_{i,b}}^{1,b}] )  \right\|_{HS}\\
	\leq& \sum_{j\in\G(b)}  \| \E[P_{\lambda_{j,b}}^{1,b}] \|_{HS} 
	\sum_{i\in\G(b)} \left\| \frac{1}{N} \sum_{t=1}^N P_{\lambda_{i,b}}^{t,b} - \E[P_{\lambda_{i,b}}^{1,b}] \right\|_{HS}\\
	=& \mathcal{O}(b)\cdot \mathcal{O}_P(\sqrt{b^3/N}) = \mathcal{O}_P(\sqrt{b^5/N}).
	\end{align*}
	By the same argument $\sum_{j,i \in \G(b)} \| K_{j,i}^{(2)} \|_{HS} = \mathcal{O}_P(\sqrt{b^5/N})$. The last term is
	centered, i.e. $\E[K_{j,i}^{(3)}] = 0$,
	and therefore 
	\begin{align*}
	\E\left[ \sum_{j,i\in\G(b)} \| K_{j,i}^{(3)} \|_{HS} \right] 
	\leq \sum_{j,i \in\G(b)} \sqrt{\| \Cov(K_{j,i}^{(3)}) \|_{N}}.
	\end{align*}
	Thus to finish the proof it remains to show $\| \Cov(K_{j,i}^{(3)}) \|_{N} = \mathcal{O}(b/N)$. To that end let 
	$Q_{\lambda_{j,b}}^{t,b} := P_{\lambda_{j,b}}^{t,b} - \E[P_{\lambda_{j,b}}^{1,b}]$. Then we have for the covariance the decomposition
	\begin{align*}
	\Cov(K_{j,i}^{(3)}) 
	=& \frac{1}{N^2} \sum_{t,u=1}^N \Cov( Q_{\lambda_{j,b}}^{t,b}\tsr Q_{\lambda_{i,b}}^{t,b}, Q_{\lambda_{j,b}}^{u,b}\tsr Q_{\lambda_{i,b}}^{u,b} )\\
	=& \frac{1}{N^2} \sum_{t,u=1}^N \E[ Q_{\lambda_{j,b}}^{t,b}\tsr Q_{\lambda_{j,b}}^{u,b}  ]\tsr_{op} \E[ Q_{\lambda_{i,b}}^{t,b}\tsr Q_{\lambda_{i,b}}^{u,b} ]\\
	&+ \E[ Q_{\lambda_{j,b}}^{t,b}\tsr Q_{\lambda_{i,b}}^{u,b}  ]\tsr_{op}^T \E[ Q_{\lambda_{i,b}}^{t,b}\tsr Q_{\lambda_{j,b}}^{u,b} ]\\
	&+ \cum( Q_{\lambda_{j,b}}^{t,b},Q_{\lambda_{i,b}}^{t,b},Q_{\lambda_{j,b}}^{u,b},Q_{\lambda_{i,b}}^{u,b} )\\
	=& S_{j,i}^{(1)} + S_{j,i}^{(2)} + S_{j,i}^{(3)}.
	\end{align*} 
	Using \eqref{Eq:BoundForCovarianceOfPeriodogram} we see that the first term is bounded by 
	\begin{align*}
	\| S_{j,i}^{(1)} \|_{N} 
	\leq& \frac{1}{N^2} \sum_{t,u=1}^N \| \Cov( P_{\lambda_{j,b}}^{t,b},  P_{\lambda_{j,b}}^{u,b} ) \|_{N}
	\| \Cov( P_{\lambda_{i,b}}^{t,b},  P_{\lambda_{i,b}}^{u,b} ) \|_{N}\\
	\leq& \frac{1}{16 \pi^4b^4N^2} \sum_{t,u=1}^N \Bigg( \sum_{r_1,s_1,r_2,s_2=1}^b \| \cum(X_{t-u+r_1-s_2}, X_{t-u+s_1-s_2}, X_{r_2-s_2}, X_{0}  ) \|_{N} \\
	&+ \| \Gamma_{|t-u+r_1-r_2|} \|_N\| \Gamma_{|t-u+s_1-s_2|} \|_N + \| \Gamma_{|t-u+r_1-s_2|} \|_N\| \Gamma_{|t-u+s_1-r_2|} \|_N \Bigg)^2.
	\end{align*}
	Rewriting the squared term as a double sum we see that three types of terms appear, which are of the following form
	\begin{align}
	&\frac{1}{16 \pi^4b^4N^2} \sum_{t,u=1}^N \sum_{r_1,s_1,r_2,s_2=1}^b \sum_{c_1,d_1,c_2,d_2=1}^b
	\| \Gamma_{|t-u+r_1-r_2|} \|_N\| \Gamma_{|t-u+s_1-s_2|} \|_N\nonumber\\
	&\qquad\qquad\qquad\qquad\qquad\qquad\qquad\qquad\times \| \Gamma_{|t-u+c_1-c_2|} \|_N\| \Gamma_{|t-u+d_1-d_2|} \|_N;\label{Eq:TermType1}\\
	& \frac{1}{16 \pi^4b^4N^2} \sum_{t,u=1}^N \sum_{r_1,s_1,r_2,s_2=1}^b \sum_{c_1,d_1,c_2,d_2=1}^b
	\| \Gamma_{|t-u+r_1-r_2|} \|_N\| \Gamma_{|t-u+s_1-s_2|} \|_N \nonumber\\
	&\qquad\qquad\qquad\qquad\qquad\qquad\qquad\qquad\times \| \cum(X_{t-u+c_1-d_2}, X_{t-u+d_1-d_2}, X_{c_2-d_2}, X_{0}  ) \|_{N};\label{Eq:TermType2}\\
	& \frac{1}{16 \pi^4b^4N^2} \sum_{t,u=1}^N \sum_{r_1,s_1,r_2,s_2=1}^b \sum_{c_1,d_1,c_2,d_2=1}^b
	\| \cum(X_{t-u+r_1-s_2}, X_{t-u+s_1-s_2}, X_{r_2-s_2}, X_{0}  ) \|_{N} \nonumber\\
	&\qquad\qquad\qquad\qquad\qquad\qquad\qquad\qquad\times \| \cum(X_{t-u+c_1-d_2}, X_{t-u+d_1-d_2}, X_{c_2-d_2}, X_{0}  ) \|_{N}. \label{Eq:TermType3}
	\end{align}
	Summing in \eqref{Eq:TermType1} first over $d_2$, then over $c_2$, then over $s_2$ and finally over $u$ and estimating the sum each time with 
	$\sum_{h\in\Z} \| \Gamma_h \|_N$ yields the bound
	\begin{align*}
	\frac{1}{16 \pi^4b^4N^2} \sum_{t=1}^N \sum_{r_1,s_1,r_2=1}^b \sum_{c_1,d_1=1}^b \left( \sum_{h\in\Z} \| \Gamma_h \|_N\right)^4
	= \mathcal{O}(b/N). 
	\end{align*}
	Summing in \eqref{Eq:TermType2} first over $c_1$, then over $d_1$, then over $d_2$, then over $s_2$ and finally over $u$ yields the bound
	\begin{align*}
	\frac{1}{16 \pi^4b^4N^2} \sum_{t=1}^N \sum_{r_1,s_1,r_2=1}^b \sum_{c_2=1}^b \left( \sum_{h\in\Z} \| \Gamma_h \|_N \right)
	\left( \sum_{h_1,h_2,h_3 \in \Z} \| \cum(X_{h_1},X_{h_2},X_{h_3},X_0) \|_{N} \right) = \mathcal{O}(1/N).
	\end{align*}
	Lastly summing in \eqref{Eq:TermType3} first over $c_1$, then over $d_1$, then over $d_2$, then over $r_1$, then over $u$ and finally over
	$s_2$ yields the bound 
	\begin{align*}
	\frac{1}{16 \pi^4b^4N^2} \sum_{t=1}^N \sum_{s_1,r_2=1}^b \sum_{c_2=1}^b 
	\left( \sum_{h_1,h_2,h_3 \in \Z} \| \cum(X_{h_1},X_{h_2},X_{h_3},X_0) \|_{N} \right)^2 = \mathcal{O}(1/(Nb)).
	\end{align*}
	Combining these estimates we can conclude that $\| S_{j,i}^{(1)} \|_{N} = \mathcal{O}(b/N) $ uniformly in $j,i$. Since 
	\begin{align*}
	\| S_{j,i}^{(2)} \|_{N}
	\leq \frac{1}{N^2} \sum_{t,u=1}^N \| \Cov( P_{\lambda_{j,b}}^{t,b},  P_{\lambda_{i,b}}^{u,b} ) \|_{N}
	\| \Cov( P_{\lambda_{i,b}}^{t,b},  P_{\lambda_{j,b}}^{u,b} ) \|_{N}
	\end{align*}
	this term can be handled similarly, that is we also have $\| S_{j,i}^{(2)} \|_{N} = \mathcal{O}(b/N) $.
	For the last term $S_{j,i}^{(3)}$ we have 
	\begin{align*}
	&\| S_{j,i}^{(3)} \|_{N}\\ 
	\leq& \frac{1}{N^2} \sum_{t,u=1}^N
	\| \cum( Q_{\lambda_{j,b}}^{t,b}, Q_{\lambda_{i,b}}^{t,b}, Q_{\lambda_{j,b}}^{u,b}, Q_{\lambda_{i,b}}^{u,b} ) \|_{N}\\
	\leq& \frac{1}{16\pi^4N^2b^4} \sum_{t,u=1}^N \sum_{r_1,s_1,\ldots,r_4,s_4=1}^b
	\| \cum( X_{t+r_1-1}\tsr X_{t+s_1-1} - \Gamma_{r_1-s_1},  X_{t+r_2-1}\tsr X_{t+s_2-1}-\Gamma_{r_2-s_2}, \\
	&X_{u+r_3-1}\tsr X_{u+s_3-1}-\Gamma_{r_3-s_3},  X_{u+r_4-1}\tsr X_{u+s_4-1}-\Gamma_{r_4-s_4} ) \|_{N}.
	\end{align*}
	If we decompose the $4$th-order cumulant operator a tedious calculation shows that, due to Assumption~\ref{as.1}(iii), 
	the dominating terms in this sum are 
	given by products of covariance operators, with a typical one given by 
	\begin{align*}
	&\frac{1}{16\pi^4N^2b^4} \sum_{t,u=1}^N \sum_{r_1,s_1,\ldots,r_4,s_4=1}^b
	\|\Cov( X_{t+r_1-1}, X_{u+s_3-1} )\|_{N} \|\Cov( X_{t+r_2-1}, X_{u+s_4-1} )\|_{N} \\
	&\qquad\qquad\qquad\qquad\qquad\qquad\times \|\Cov( X_{u+s_3-1}, X_{t+s_1-1} )\|_{N}
	\|\Cov( X_{u+r_4-1}, X_{t+s_2-1} )\|_{N}\\
	=& \frac{1}{16\pi^4N^2b^4} \sum_{t,u=1}^N \sum_{r_1,s_1,\ldots,r_4,s_4=1}^b
	\|\Gamma_{t-u+r_1-s_3}\|_N  \|\Gamma_{t-u+r_2-s_4}\|_N \|\Gamma_{u-t+r_4-s_4}\|_N \|\Gamma_{u-t+r_4-s_2}\|_N
	\end{align*}
	Summing over $s_2$, $r_4$, $r_2$ and then  $u$, yields the upper bound 
	\begin{align*}
	\frac{1}{16\pi^4N^2b^4} \sum_{t=1}^N \sum_{r_1,s_1,r_3,s_3,s_4=1}^b \left( \sum_{h\i\Z} \|\Gamma_h\|_N\right)^4 = \mathcal{O}(b/N)
	\end{align*} 
	and hence $\|S_{j,i}^{(3)}\|_{N} = \mathcal{O}(b/N)$ as well. 
\end{proof}


\subsection{Proofs of main results}
\begin{proof}[Proof of Lemma~\ref{le.App-CLT-1}]
Consider  (i).  
Note that, by construction, $\E P_{n,\lambda_{j,n}}^\star=\widehat{F}_{\lambda_{j,n}}$,  the $P_{n,\lambda_{j,n}}^\star$ are  independent  
for $ j=1,2, \ldots, N$  and  they satisfy 
\begin{align*}
{\rm Cov}^\star (P_{n,\lambda_{j,n}}^\star) & = \E^\star  \big(P_{n,\lambda_{j,n}}^\star \otimes P_{n,\lambda_{j,n}}^\star \big)- \widehat{F}_{\lambda_{j,n}}\otimes \widehat{F}_{\lambda_{j,n}} \\
& = \E\big((J_n^\star (\lambda_{j,n}) \otimes J_n^\star (\lambda_{j,n}) )\otimes (J_n^\star (\lambda_{j,n}) \otimes J_n^\star (\lambda_{j,n}))\big) -  \widehat{F}_{\lambda_{j,n}}\otimes \widehat{F}_{\lambda_{j,n}}\\
& ={\rm cum}\big(J_n^\star (\lambda_{j,n}),\overline{J_n^\star} (\lambda_{j,n}), \overline{J_n^\star} (\lambda_{j,n}),J_n^\star (\lambda_{j,n})\big) \\
& \ \  \ \ +  \E\big((J_n^\star (\lambda_{j,n}) \otimes J_n^\star (\lambda_{j,n}) ) \otimes_{op}  \E\big((J_n^\star (\lambda_{j,n}) \otimes J_n^\star (\lambda_{j,n}) )\\
& \ \  \ \ +  \E\big((J_n^\star (\lambda_{j,n}) \otimes \overline{J_n^\star} (\lambda_{j,n}) ) \otimes_{op}^\top  \E\big((\overline{J_n^\star} (\lambda_{j,n}) \otimes J_n^\star (\lambda_{j,n}) )\\
& = \widehat{\F}_{\lambda_{j,n}}\otimes_{op} \widehat{\F}_{\lambda_{j,n}}.
\end{align*}
The last equality appears since   the term $ {\rm cum}^\star(J^\star_n(\lambda_{j,n}), \overline{J^\star}_n(\lambda_{j,n}), \overline{J^\star}_n(\lambda_{j,n}), J^\star_n(\lambda_{j,n}))=0$  and $  {\rm E}\big((J_n^\star (\lambda_{j,n}) \otimes \overline{J_n^\star} (\lambda_{j,n}) )={\rm Rel}( J_n^\star (\lambda_{j,n}) )=0$ 
because $ J_n^\star (\lambda_{j,n}) \sim {\mathcal C}{\mathcal N}(0,\widehat{\F}_{\lambda_{j,n}})$. To see why the former statement is true  note that for $f,g,u,v\in\HH$, we have 
\begin{align*}
\big\langle {\rm cum}^\star \big( J^\star_n(\lambda_{j,n}),&  \overline{J^\star}_n(\lambda_{j,n}), \overline{J^\star}_n(\lambda_{j,n}), J^\star_n(\lambda_{j,n})\big), (f \otimes g)\otimes(u\otimes v) \big \rangle _{HS(HS)}\\
& = {\rm cum}^\star \big( \langle J^\star_n(\lambda_{j,n}),  f \rangle,  \langle \overline{J^\star}_n(\lambda_{j,n}) ,\overline{g}\rangle,  \langle \overline{J^\star}_n(\lambda_{j,n}) ,\overline{u}\rangle,  \langle J^\star_n(\lambda_{j,n}),  v \rangle \big)\\
& =0,
\end{align*}
because $\big( \langle J^\star_n(\lambda_{j,n}),  f \rangle,  \langle \overline{J^\star}_n(\lambda_{j,n}) ,\overline{g}\rangle,  \langle \overline{J^\star}_n(\lambda_{j,n}) ,\overline{u}\rangle,  \langle J^\star_n(\lambda_{j,n}),  v \rangle \big) $ is  a jointly complex Gaussian vector.  Since $ \{(f\otimes g)\otimes (u\otimes v) \big| f,g,u,v \in \HH \}$  is dense in $HS(HS)$, this yields ${\rm cum}^\star(J^\star_n(\lambda_{j,n}), \overline{J^\star}_n(\lambda_{j,n}), \overline{J^\star}_n(\lambda_{j,n}), J^\star_n(\lambda_{j,n}))=0 $.

Along the same lines we get  
\[ {\rm Cov}^\star (\overline{P}_{n,\lambda_{j,n}}^\star)  = \E^\star  \big(\overline{P}_{n,\lambda_{j,n}}^\star \otimes \overline{P}_{n,\lambda_{j,n}}^\star \big)- \overline{\widehat{\F}}_{\lambda_{j,n}}\otimes \overline{\widehat{\F}}_{\lambda_{j,n}}=  \overline{\widehat{\F}}_{\lambda_{j,n}}\otimes_{op} \overline{\widehat{\F}}_{\lambda_{j,n}} ,\]
\[ {\rm Rel}^\star (P_{n,\lambda_{j,n}}^\star)  =  \E^\star  \big(P_{n,\lambda_{j,n}}^\star \otimes \overline{P}_{n,\lambda_{j,n}}^\star\big) - \widehat{\F}_{\lambda_{j,n}}\otimes \overline{\widehat{\F}}_{\lambda_{j,n}} 
   = \widehat{\F}_{\lambda_{j,n}}\otimes^\top_{op} \overline{\widehat{\F}}_{\lambda_{j,n}}\]
  and
  \[{\rm Rel}^\star (\overline{P}_{n,\lambda_{j,n}}^\star)  =  \E^\star  \big(\overline{P}_{n,\lambda_{j,n}}^\star \otimes P_{n,\lambda_{j,n}}^\star\big) - \overline{\widehat{\F}}_{\lambda_{j,n}}\otimes  \widehat{\F}_{\lambda_{j,n} }
   = \overline{\widehat{\F}}_{\lambda_{j,n}}\otimes^\top_{op} \widehat{\F}_{\lambda_{j,n}}.\]
 Next, we   write $ L_n^\star$ as 
\begin{align*}
L^\star_{n} & =\frac{2\pi}{\sqrt{n}} \sum_{j =1}^ N \Big\{W(\lambda_{j,n}) \big(P_{n,\lambda_{j,n}}^\star- \widehat{\F}_{\lambda_{j,n}}\big) + W(-\lambda_{j,n}) \big(\overline{P}_{n,\lambda_{j,n}}^\star - \overline{\widehat{\F}}_{\lambda_{j,n}}\big)\Big\}\\
& = \sum_{j =1}^N Z_{j,n}^\star,
\end{align*}
with an obvious notation for  $Z_{j,n}^\star$. According to the last expression,  $L^\ast_n$ is a  sum of $N$ independent random elements. 
Using the expression for the covariance and relation operator of the pseudo periodogram operator $P^\ast_{n,\lambda_{j,n}}$ we get by straightforward calculations, 
\begin{align} \label{eq.Cov1.1}
{\rm Cov}^\star(L^\star_n) & = \sum_{j=1}^N {\rm Cov}^\star(Z^\star_{j,n}, Z_{j,n}^\star) \nonumber \\
& = \frac{4 \pi^2}{n}\sum_{j\in {\mathcal G}(n)} W(\lambda_{j,n})\overline{W}(\lambda_{j,n})  \widehat{F}_{\lambda_{j,n}}\otimes_{op} \widehat{F}_{\lambda_{j,n}}  \nonumber \\
& \ \ \ \  +  \frac{4 \pi^2}{n}\sum_{j\in {\mathcal G}(n)} W(\lambda_{j,n})\overline{W}(-\lambda_{j,n})  \widehat{F}_{\lambda_{j,n}}\otimes_{op}^\top  \widehat{F}_{-\lambda_{j,n}}. 
\end{align}
Observe that 
\[ \widehat{\F}_{\lambda} \otimes_{op}\widehat{\F}_\lambda = (\widehat{\F}_\lambda -\F_\lambda)\otimes_{op} \widehat{\F}_\lambda +\widehat{\F}_\lambda \otimes_{op} (\widehat{\F}_\lambda-\F_\lambda) + \F_\lambda\otimes \F_\lambda,\]
with an analogue expression for $ \widehat{\F}_{\lambda} \otimes_{op}^\top \widehat{\F}_\lambda$, 
and by assumption, for example, 
\begin{align*}
\Big\| \frac{4\pi^2}{n}\sum_{j\in{\mathcal G}(n)} & W(\lambda_{j,n}) \overline{W}(\lambda_{j,n}) (\widehat{\F}_{\lambda_{j,n}}-\F_{\lambda_{j,n}}) \otimes_{op} \widehat{\F}_{\lambda_{j,n}} \Big\|_N\\
& \leq \frac{4\pi^2}{n}\sum_{j\in{\mathcal G}(n)} |W(\lambda_{j,n})|^2 \| \widehat{\F}_{\lambda_{j,n}}-\F_{\lambda_{j,n}}\|_N\|\widehat{\F}_{\lambda_{j,n}}\|_N\\
& \leq \sup_{\lambda_{j,n}} \| \widehat{\F}_{\lambda_{j,n}}-\F_{\lambda_{j,n}}\|_N \sup_{\lambda_{j,n}} \| \widehat{\F}_{\lambda_{j,n}} \|_N \frac{4\pi^2}{n}\sum_{j\in {\mathcal G}(n)} |W(\lambda_{j,n})|^2\\
& = {\mathcal O}_P\big(\sup_{\lambda_{j,n}} \| \widehat{\F}_{\lambda_{j,n}}-\F_{\lambda_{j,n}}\|_N \big) =o_{P}(1),
\end{align*}
by Assumption~\ref{as.F}. 
Therefore, 
\begin{align} \label{eq.Cov1}
\|{\rm Cov}^\star(L_n^\star) -\Psi_1\|_N  & \leq \big\| \frac{4 \pi^2}{n}\sum_{j\in {\mathcal G}(n)} W(\lambda_{j,n})\overline{W}(\lambda_{j,n})  \F_{\lambda_{j,n}}\otimes_{op} \F_{\lambda_{j,n}}  \\
& \ \ \ \ \ \ \ \ -  2\pi\int_{\pi}^\pi W(\lambda)\overline{W}(\lambda) \F_\lambda \otimes_{op} \F_\lambda d\lambda \big\|_N \nonumber \nonumber \\
 &  + \big\| \frac{4 \pi^2}{n}\sum_{j\in {\mathcal G}(n)} W(\lambda_{j,n})\overline{W}(-\lambda_{j,n}) \F_{\lambda_{j,n}}\otimes_{op}^\top \F_{\lambda_{j,n}} \nonumber \\
& \ \ \ \ \ \ \ \ \ - 2\pi\int_{\pi}^\pi W(\lambda)\overline{W}(-\lambda) \F_\lambda \otimes_{op}^\top  \F_\lambda d\lambda \big\|_N  + o_{P}(1) \nonumber \\
& \stackrel{P}{\rightarrow}  0 \nonumber 
\end{align}
since 
$ \lambda\mapsto W(\lambda)\overline{W}(\lambda)\F_\lambda\otimes_{op} \F_\lambda \in N(HS)$, respectively, 
  $ \lambda\mapsto W(\lambda)\overline{W}(\lambda)\F_\lambda\otimes_{op}^\top \F_\lambda \in N(HS)$ and due to  the 
 convergence of the Riemann sum to the Bochner  integral   in view of  Assumption~\ref{as.W} and the continuity of $\lambda\mapsto  \F_\lambda \otimes_{op} \F_{\lambda}$, respectively,  $\lambda\mapsto  \F_\lambda \otimes_{op}^\top  \F_\lambda$; see also the evaluation of (\ref{Eq:RiemannSumR1}) in the proof of Lemma~\ref{Lem:AsymptoticCovarianceRiemannApprox}.

Along the same lines  and as for ${\rm Cov}^\star(L_n^\star)$ we 
 get 
\begin{align} \label{eq.Rel1}
{\rm Rel}^\star(L^\star_n) & = \sum_{j=1}^N {\rm Cov}^\star(Z^\star_{j,n}, \overline{Z}_{j,n}^\star) \\
& = \frac{4 \pi^2}{n}\sum_{j\in {\mathcal G}(n)} W(\lambda_{j,n}) W(-\lambda_{j,n})  \widehat{\F}_{\lambda_{j,n}}\otimes_{op}  \widehat{\F}_{\lambda_{j,n}}  \nonumber \\
& \ \ \ \  +  \frac{4 \pi^2}{n}\sum_{j\in {\mathcal G}(n)} W(\lambda_{j,n})W(\lambda_{j,n})  \widehat{\F}_{\lambda_{j,n}}\otimes^\top_{op}  \widehat{\F}_{-\lambda_{j,n}}\nonumber  \\
& \stackrel{P}{\rightarrow} 
 2\pi\int_{\pi}^\pi W(\lambda)W(-\lambda) \F_\lambda \otimes_{op}  \F_\lambda d\lambda  \nonumber \\
 & \ \ \ \ \  + 2\pi  \int_{\pi}^\pi W(\lambda)W(\lambda) \F_\lambda \otimes_{op}^\top \F_{-\lambda} d\lambda=\Upsilon_1, \nonumber
\end{align}
in trace norm.

Consider (ii). Since  $L^\star_{n} =\sum_{j =1}^N Z_{j,n}^\star$ is a sum of independent random Hilbert-Schmidt operatros,  to establish assertion (ii)  it suffices by Theorem 3.2 of \cite{Billingsley1999} to show that 
for any basis $\{ e_r\otimes e_s: r,s=1,2, \ldots\}$  of  $HS(H)$, the following three conditions are satisfied (in $HS$-norm):
\begin{enumerate}
\item[1)] $\PP_k\otimes_{op}  \PP_{k}(L^\star_n) = \sum_{r,s=1}^k \sum_{j=1}^N \langle Z_{j,n}^\star, e_r \otimes e_s\rangle e_r \otimes e_s \stackrel{{\mathcal D} }{\rightarrow} {\mathcal G}_k$,
as $ n\rightarrow\infty$ for every $k\in\N$, where  
$ {\mathcal G}_k \sim {\mathcal C}{\mathcal N}(0,\Psi_1^{(k)},  \Upsilon_1^{(k)})$, with 
\begin{align*}
 \Psi^{(k)}_1 & = \sum_{r_1,s_1=1}^k \sum_{r_2,s_2=1}^k \Big\langle   \Psi_1(e_{r_2}\otimes e_{s_2}),  e_{r_1}\otimes e_{s_1}\Big\rangle ( e_{r_1}\otimes e_{s_1})\otimes ( e_{r_2}\otimes e_{s_2}) \\
 & =  ( \PP_k\otimes_{op}  \PP_{k}) \Psi_1 (\PP_k\otimes_{op}  \PP_{k})
 \end{align*}
and 
\begin{align*}
 \Upsilon^{(k)}_{1}& = \sum_{r_1,s_1=1}^k \sum_{r_2,s_2=1}^k \Big\langle   \Upsilon_1(e_{r_2}\otimes e_{s_2}),  e_{r_1}\otimes e_{s_1}\Big\rangle ( e_{r_1}\otimes e_{s_1})\otimes ( e_{r_2}\otimes e_{s_2}) \\
  & =  ( \PP_k\otimes_{op}  \PP_{k}) \Upsilon_1 (\PP_k\otimes_{op}  \PP_{k}).
 \end{align*}
\item[]
\item[2)] $  {\mathcal G}_k  \stackrel{{\mathcal D} }{\rightarrow} {\mathcal G} \sim  {\mathcal C}{\mathcal N}(0,\Psi_1, \Upsilon_1)$ as $ k \rightarrow \infty$,
 \ and
 \item[] 
 \item[3)]   For every $ \epsilon>0$, $ \lim_{k\rightarrow\infty}\limsup_{n\rightarrow\infty} P\big( \| \PP_k\otimes_{op} \PP_k  (L^\star_{n}) - L_{n}^\star\|_{HS} > \epsilon \big)=0$.
\end{enumerate} 

\noindent Consider 1). Note that
\begin{align*}
\PP_k\otimes \PP_{k}(L^\star_n) &= \sum_{r,s=1}^k\frac{2\pi}{\sqrt{n}}\sum_{j=1}^N \Big\{ W(\lambda_{j,n}) \langle P_{n,\lambda_{j,n}}^\star- \widehat{\F}_{\lambda_{j,n}},  e_r
\otimes e_s  \rangle  e_r\otimes e_s \\
& \ \ \ \ \  +W(-\lambda_{j,n}) \langle \overline{P}_{n,\lambda_{j,n}}^\star- \overline{\widehat{\F}}_{\lambda_{j,n}},  e_r
\otimes e_s  \rangle  e_r\otimes e_s \Big\} .
\end{align*} 
For the covariance and relation operator of $ \PP_k\otimes_{op}  \PP_k(L_n^\star)$ we have by(i),
\begin{align*}
{\rm Cov}^\star& (\PP_k\otimes_{op} \PP_k( L_n^\star) )=  (\PP_k\otimes \PP_k) {\rm Cov}^\star(L^\star_n) (\PP_k\otimes_{op} \PP_k) \\
& = \sum_{r_1,s_1=1}^k \sum_{r_2,s_2=1}^k \Big\langle   {\rm Cov}^\star(L^\star_n) (e_{r_2}\otimes e_{s_2}),  e_{r_1}\otimes e_{s_1}\Big\rangle ( e_{r_1}\otimes e_{s_1})\otimes ( e_{r_2}\otimes e_{s_2})\\
& \stackrel{P}{\rightarrow} \sum_{r_1,s_1=1}^k \sum_{r_2,s_2=1}^k \Big\langle   \Psi_1(e_{r_2}\otimes e_{s_2}),  e_{r_1}\otimes e_{s_1}\Big\rangle ( e_{r_1}\otimes e_{s_1})\otimes ( e_{r_2}\otimes e_{s_2}) .
\end{align*}
 Along the same lines  we  also get,
\begin{align*}
{\rm Rel}^\star(& \PP_k\otimes_{op} \PP_k(L_n^\star))   =  (\PP_k\otimes_{op} \PP_k) {\rm Rel}^\star(L^\star_n) (\PP_k\otimes_{op} \PP_k) \\
& \stackrel{P}{\rightarrow} \sum_{r_1,s_1=1}^k \sum_{r_2,s_2=1}^k \Big\langle   \Upsilon_1(e_{r_2}\otimes e_{s_2}),  e_{r_1}\otimes e_{s_1}\Big\rangle ( e_{r_1}\otimes e_{s_1})\otimes ( e_{r_2}\otimes e_{s_2}). \end{align*}
To establish the weak convergence of $ \PP_k\otimes_{op} \PP_k(L_n^\star)$,  let 
\begin{align*}  
Y^\star_{j,n} (r,s) & := 2\pi  \Big\{ W(\lambda_{j,n}) \big\langle P_{n,\lambda_{j,n}}^\star- \widehat{\F}_{\lambda_{j,n}},  e_r
\otimes e_s \big\rangle  \\
& \ \ \ \  \ \ \ \  +W(-\lambda_{j,n}) \big\langle \overline{P}_{n,\lambda_{j,n}}^\star- \overline{\widehat{\F}}_{\lambda_{j,n}},  e_r
\otimes e_s \big\rangle  \Big\} \\
& = \sqrt{n} \langle Z_{j,n}^\star, e_r\otimes e_s\rangle 
\end{align*}
such  that   
\begin{equation} \label{eq.CLTrs-1}
 \PP_k\otimes_{op} \PP_k(L_n^\star) = \sum_{r,s=1}^k  \frac{1}{\sqrt{n}} \sum_{j=1}^N Y^\star_{j,n} (r,s) (e_r\otimes e_s).
  \end{equation}
It then follows from independence and (i), that 
\begin{align*}
\frac{1}{n}\sum_{j=1}^N {\rm Cov}^\star\big(Y_{j,n}^\star(r,s)) &= \sum_{j=1}^N {\rm Cov}^\star \big(\big\langle Z_{j,n}^\star, e_r\otimes e_s\big\rangle\big) \\
& =\big\langle   {\rm Cov}^\star( L_n^\star)(e_r\otimes e_s), e_r\otimes e_s \big\rangle \\
& \stackrel{P}{\rightarrow} \Big\langle  \Psi_1(e_r\otimes e_s), e_r\otimes e_s\Big\rangle =: \Psi_1(r,s).
\end{align*}
 and  
\begin{align*}
\frac{1}{n}\sum_{j=1}^N {\rm Rel}^\star\big(Y_{j,n}^\star(r,s)) &= \big\langle  \sum_{j=1}^N {\rm Rel}^\star (Z_{j,n}^\star)(e_r\otimes e_s), e_r\otimes e_s \big\rangle \\
&=  \langle {\rm Rel}^\star(L_n^\star)(e_r\otimes e_s), e_r\otimes e_s\rangle \\
& \stackrel{P}{\rightarrow} \Big\langle  \Upsilon_1(e_r\otimes e_s), e_r\otimes e_s\Big\rangle =: \Upsilon_1(r,s).
\end{align*}
We next  use the complex version of Liapunov's CLT to show that 
\begin{equation} \label{eq.CLTrs}
  \frac{1}{\sqrt{n}} \sum_{j=1}^N Y^\star_{j,n} (r,s) =  \frac{1}{\sqrt{N}} \sum_{j=1}^N \sqrt{\frac{N}{n}}Y^\star_{j,n} (r,s)     \stackrel{{\mathcal D} }{\rightarrow}  {\mathcal C}{\mathcal N}\big(0,\Psi_1(r,s) , \Upsilon_1(r,s)\big),
\end{equation}
in probability, as $n\rightarrow\infty$. To that end, note that for  $\delta>0$,
\begin{align*}
\E^\star\Big| \sqrt{\frac{N}{n}}Y^\star_{j,n} (r,s) \Big|^{2+\delta} & \leq 2^{\delta/2}(2\pi)^{2+\delta}\Big\{|W(\lambda_{j,n})|^{2+\delta}\E^\star\big| \langle P^\star_{\lambda_{j,n}}-\widehat{\F}_{\lambda_{j,n}}, e_r\otimes e_s\rangle \big|^{2+\delta}\\
& \ \ \ \ \ \ \ \  +  |W(-\lambda_{j,n})|^{2+\delta}\E^\star\big| \langle \overline{P^\star}_{\lambda_{j,n}}-\overline{\widehat{\F}}_{\lambda_{j,n}}, e_r\otimes e_s\rangle \big|^{2+\delta}\\
& \leq 2^{\delta/2}(2\pi)^{2+\delta} \big(|W(\lambda_{j,n})|^{2+\delta} + |W(-\lambda_{j,n})|^{2+\delta}\big) \\
& \ \ \ \ \ \ \ \ \ \ \   \times \E^\star\|P^\star_{\lambda_{j,n}}-\widehat{\F}_{\lambda_{j,n}}\|_{HS}^{2+\delta}\\
& \leq 2^{\delta/2}(2\pi)^{2+\delta} \big(|W(\lambda_{j,n})|^{2+\delta} + |W(-\lambda_{j,n})|^{2+\delta}\big) \\
& \ \ \ \ \ \ \ \ \ \ \   \times 2^{1+\delta}\big( \E^\star\|P^\star_{\lambda_{j,n}}\|_{HS}^{2+\delta}  + \|\widehat{\F}_{\lambda_{j,n}}\|_{HS}^{2+\delta}\big)\\
& \leq 2^{1+3\delta/2}(2\pi)^{2+\delta} \big(|W(\lambda_{j,n})|^{2+\delta} + |W(-\lambda_{j,n})|^{2+\delta}\big) \\
& \ \ \ \ \ \ \ \ \ \ \   \times \big( \E^\star\|J^\star_{\lambda_{j,n}}\|^{4+2\delta}  + \|\widehat{\F}_{\lambda_{j,n}}\|_{HS}^{2+\delta}\big),
\end{align*}
where the second inequality follows using Cauchy-Schwarz's inequality.   Now, note that by construction,  ${\rm Re}(J_{\lambda_{j,n}}^\star) \sim {\mathcal N}(0,{\rm Re}(\widehat{\F}_{\lambda_{j,n}})/2)$ and 
  ${\rm Im}(J_{\lambda_{j,n}}^\star) \sim {\mathcal N}(0,{\rm Re}(\widehat{\F}_{\lambda_{j,n}})/2)$. Thus for  $\delta=2$ and  using   Lemma~\ref{le.Mom} we get
\begin{align*}
\E^\star \|J_{\lambda_{j,n}}^\star\|^8  & \leq 2^7\big(\E^\star \|{\rm Re}(J_{\lambda_{j,n}}^\star)\|^8 +  \E^\star \|{\rm Im}(J_{\lambda_{j,n}}^\star)\|^8\big)\\
 & \leq 2^7\Big( 2\, C_4 4! \|{\rm Re}(\widehat{\F}_{\lambda_{j,n}})/2\|_N^4\Big)\\
 & \leq 2^4 4!\, C_4  \|\widehat{\F}_{\lambda_{j,n}}\|_N^4.
\end{align*}
 Therefore  because  \[|W(\lambda_{j,n})|^4 + |W(-\lambda_{j,n})|^4 \leq 2\big( |W(0)| + \|W\|_{TV}\big)^4=:C_W<\infty,\]
 we have
\begin{align*}
\E^\star \Big| \sqrt{\frac{N}{n}}Y^\star_{j,n} (r,s) \Big|^4 &\leq 2^4(2\pi)^4 C_W \big\{2^4 4!\, C_4 \| \widehat{\F}_{\lambda_{j,n}}\|_N^4 + \| \widehat{\F}_{\lambda_{j,n}}\|_N^4\big\},
\end{align*}
which  is bounded  in probability,   uniformly in  $\lambda_{j,n}$ since 
$ \sup_{\lambda_{j,n} }\|\widehat{\F}_{\lambda_{j,n}}\|_N^4 $ is bounded in probability by Assumption~\ref{as.F}. Hence 
Lyapunov's condition  for $\delta=2$ is satisfied  by $\sqrt{N/n}Y^\star_{j,n}(r,s) $,  which  yields the 
weak convergence   (\ref{eq.CLTrs}).  

Now, to conclude the proof of 1.), observe that 
\[ {\mathcal C}{\mathcal N}\big(0, \Psi_{1}(r,s), \Upsilon_1(r,s)\big) \sim \langle  {\mathcal G}, e_r\otimes e_s \rangle,\]
and  therefore, 
\[ \sum_{r,s=1}^k \Big(   \frac{1}{\sqrt{n}} \sum_{j=1}^N Y^\star_{j,n} (r,s) \Big) (e_r\otimes e_s)   \stackrel{{\mathcal D}}{\rightarrow}  \sum_{r,s=1}^k \langle {\mathcal G}, e_r\otimes e_s\rangle (e_r\otimes e_s)= {\mathcal G}_k. \]

\noindent Consider 2).  \ By Lemma~\ref{le.App-2}
it suffices to show that , as $ k\rightarrow \infty$, 
\begin{equation} \label{eq.Co2.1}
\|{\rm Cov}({\mathcal G}_k) -\Psi_1\|_N \rightarrow 0,
\end{equation}
and
\begin{equation} \label{eq.Co2.2}
\|{\rm Rel}({\mathcal G}_k) -\Upsilon_1\|_N \rightarrow 0.
\end{equation}
We only show (\ref{eq.Co2.1}) since the same arguments  are used   for (\ref{eq.Co2.2}) too.  Following the calculation  to derive $ {\rm Cov}^\star(\PP_k\otimes_{op} \PP_k(L_n^\star))$ it
follows that  
\begin{align*}
\Psi_1^{(k)} & =  \sum_{r_1,s_1=1}^k \sum_{r_2,s_2=1}^k \Big\langle   \Psi_1(e_{r_2}\otimes e_{s_2}),  e_{r_1}\otimes e_{s_1}\Big\rangle ( e_{r_1}\otimes e_{s_1})\otimes ( e_{r_2}\otimes e_{s_2})\\
& = {\rm Cov}\big( \sum_{r,s=1}^k \langle {\mathcal G}, e_r\otimes e_s\rangle (e_r\otimes e_s)\big) \\
& = {\rm Cov}\big(\PP_k\otimes_{op} \PP_k({\mathcal G})\big)\\
& =(\PP_k\otimes_{op}  \PP_k) {\rm Cov}({\mathcal G})(\PP_k\otimes_{op}  \PP_k)= (\PP_k\otimes_{op}  \PP_k )\Psi_1(\PP_k\otimes_{op}  \PP_k).
\end{align*}
Thus $ \|\Psi_1^{(k)} -\Psi_1\|_N =\|( \PP_k\otimes_{op}  \PP_k )\Psi_1(\PP_k\otimes_{op}  \PP_k)- \Psi_1\|_N \rightarrow 0$ as $ k\rightarrow \infty$ by Theorem VI. 4.3  of \cite{GohbergEtAl1990}.

\noindent Consider 3). \ By Markov's inequality we get for any $\epsilon >0$ and because $ \E^\star(L_n^\star)=0$,
\begin{align*}
P\big( \|\PP_k\otimes_{op} \PP_k(L_n^\star)&  -L_n^\star\|_{HS} >\epsilon \big)  \leq \frac{1}{\epsilon^2}\E^\star\|\PP_k\otimes_{op} \PP_k(L_n^\star) -L_n^\star\|_{HS}^2\\
& = \frac{1}{\epsilon^2}\Big\| \E^\star\big(\PP_k\otimes_{op} \PP_k(L_n^\star)-L_n^\star\big) \otimes \big(\PP_k\otimes_{op} \PP_k(L_n^\star)-L_n^\star \big)\Big\|_N\\
& =\frac{1}{\epsilon^2}\Big\|(\PP_k\otimes_{op} \PP_k){\rm Cov}^\star (L_n^\star)(\PP_k\otimes_{op} \PP_k) -(\PP_k\otimes_{op} \PP_k ) {\rm Cov}^\star (L_n^\star)\\
& \ \ \ \ \ \ \ \ -{\rm Cov}^\star (L_n^\star)(\PP_k\otimes_{op} \PP_k) -{\rm Cov}^\star (L_n^\star)\Big\|_N\\
& \leq \frac{1}{\epsilon^2} \big\|(\PP_k\otimes_{op} \PP_k){\rm Cov}^\star (L_n^\star)(\PP_k\otimes_{op} \PP_k) -(\PP_k\otimes_{op} \PP_k ) {\rm Cov}^\star (L_n^\star)\big\|_N \\
& \ \ \ \ \ \ \ \  +  \frac{1}{\epsilon^2} \big\|{\rm Cov}^\star (L_n^\star)- {\rm Cov}^\star (L_n^\star)(\PP_k\otimes_{op} \PP_k)\big\|_N \\
& = \frac{1}{\epsilon^2} \big\|(\PP_k\otimes_{op} \PP_k)\big({\rm Cov}^\star (L_n^\star)(\PP_k\otimes_{op} \PP_k) -  {\rm Cov}^\star (L_n^\star)\big) \big\|_N \\
& \ \ \ \ \ \ \ \  +  \frac{1}{\epsilon^2} \big\|{\rm Cov}^\star (L_n^\star)- {\rm Cov}^\star (L_n^\star)(\PP_k\otimes_{op} \PP_k)\big\|_N\\
& \leq \frac{2}{\epsilon^2}  \big\|{\rm Cov}^\star (L_n^\star)- {\rm Cov}^\star (L_n^\star)(\PP_k\otimes_{op} \PP_k)\big\|_N \\
& \stackrel{P}{\rightarrow}  \frac{2}{\epsilon^2}  \big\|\Psi_1-\Psi_1(\PP_k\otimes_{op} \PP_k)\big\|_N,
\end{align*}
as $ n\rightarrow \infty$, by assertion (i) of the lemma. Condition 3) is then  established because 
\begin{align*}
\|\Psi_1-\Psi_1(\PP_k\otimes_{op}&  \PP_k)\|_N   \leq  \|\Psi_1-(\PP_k\otimes_{op} \PP_k)\Psi_1 (\PP_k\otimes_{op} \PP_k)\|_{N} \\
& \ \  + \|(\PP_k\otimes_{op} \PP_k)\Psi_1 (\PP_k\otimes_{op} \PP_k)(\PP_k\otimes_{op} \PP_k) -  \Psi_1(\PP_k\otimes_{op} \PP_k)\|_{N}  \rightarrow 0,
\end{align*}
 as $ k \rightarrow \infty$ by Theorem VI. 4.3  of \cite{GohbergEtAl1990} again.
\end{proof}


\begin{proof}[Proof of Lemma \ref{Lem:CBPConsistencyCovariance}]
The random variables $P_{\lambda_{j,b}}^{(l),b}$, $l=1,\ldots, k,$ are conditionally independent, hence  	
\begin{align*}
	\Cov^*(L_n^+) 
	&= \frac{(2\pi)^2 k}{b}  \sum_{j,i \in \G(b)} W(\lambda_{j,b}) \overline{W(\lambda_{i,b})} \Cov^*( P_{j,b}^+,  P_{i,b}^+)\\
	&= \frac{(2\pi)^2 k}{b}  \sum_{j,i \in \G(b)} W(\lambda_{j,b}) \overline{W(\lambda_{i,b})} 
	\frac{1}{k^2} \sum_{l,r=1}^k \Cov^*( P_{\lambda_{j,b}}^{(l),b}, P_{\lambda_{i,b}}^{(r),b} )\\
	&= \frac{(2\pi)^2 k}{b}  \sum_{j,i \in \G(b)} W(\lambda_{j,b}) \overline{W(\lambda_{i,b})} 
	\frac{1}{k^2} \sum_{l=1}^k \Cov^*( P_{\lambda_{j,b}}^{(l),b}, P_{\lambda_{i,b}}^{(l),b} )\\
	&= \frac{(2\pi)^2}{b}  \sum_{j,i \in \G(b)} W(\lambda_{j,b}) \overline{W(\lambda_{i,b})} \Cov^*( P_{\lambda_{j,b}}^{(1),b}, P_{\lambda_{i,b}}^{(1),b} )\\
	&= \frac{(2\pi)^2}{b} \Bigg\{ 
	\sum_{j \in \G(b)} W(\lambda_{j,b}) \overline{W(\lambda_{j,b})} \Cov^*( P_{\lambda_{j,b}}^{(1),b}, P_{\lambda_{j,b}}^{(1),b} ) \\
	& \ \ \ \ +W(\lambda_{j,b}) \overline{W(-\lambda_{j,b})} \Cov^*( P_{\lambda_{j,b}}^{(1),b}, P_{-\lambda_{j,b}}^{(1),b} )\\
	&\ \ \ \ + \sum_{j \in \G(b)} \sum_{i \in \G(b)\atop i\not= \pm j} W(\lambda_{j,b})  \overline{W(\lambda_{i,b})}  
	\Cov^*( P_{\lambda_{j,b}}^{(1),b}, P_{\lambda_{i,b}}^{(1),b} )
	\Bigg\}\\
	&= R_1^* + R_2^* + R_3^*.
\end{align*}
Notice that for 
\begin{align*}
	\Cov^*( P_{\lambda_{j,b}}^{(1),b}, P_{\lambda_{i,b}}^{(1),b} )
	&= \E^*[ P_{\lambda_{j,b}}^{(1),b} \tsr P_{\lambda_{i,b}}^{(1),b} ] - \E^*[ P_{\lambda_{j,b}}^{(1),b} ] \tsr \E^*[ P_{\lambda_{i,b}}^{(1),b} ]\\
	&= \frac{1}{N} \sum_{t=1}^N P_{\lambda_{j,b}}^{t,b}\tsr P_{\lambda_{i,b}}^{t,b} - \wt{\F}_{\lambda_{j,b}} \tsr \wt{\F}_{\lambda_{i,b}}
	=: H_{j,i,n}
\end{align*}
we have 
\begin{align*}
	&\sum_{j,i \in\G(b)} \left\| H_{j,i,n} - \Cov( P_{\lambda_{j,b}}^{1,b}, P_{\lambda_{i,b}}^{1,b} ) \right\|_{HS}\\
	\leq& \sum_{j,i \in\G(b)}  
	\left\| \frac{1}{N} \sum_{t=1}^N P_{\lambda_{j,b}}^{t,b}\tsr P_{\lambda_{i,b}}^{t,b} 
	- \E[ P_{\lambda_{j,b}}^{1,b} \tsr P_{\lambda_{i,b}}^{1,b} ] \right\|_{HS}\\
	&+ \left\| \widetilde{\F}_{\lambda_{j,b}} \tsr \widetilde{\F}_{\lambda_{i,b}} 
	\mp \widetilde{\F}_{\lambda_{j,b}} \tsr \E[ P_{\lambda_{i,b}}^{1,b} ] 
	- \E[ P_{\lambda_{j,b}}^{1,b} ] \tsr \E[ P_{\lambda_{i,b}}^{1,b} ]\right\|_{HS}\\
	\leq&  \sum_{j,i \in\G(b)} \left\| \frac{1}{N} \sum_{t=1}^N P_{\lambda_{j,b}}^{t,b}\tsr P_{\lambda_{i,b}}^{t,b} 
	- \E[ P_{\lambda_{j,b}}^{1,b} \tsr P_{\lambda_{i,b}}^{1,b} ] \right\|_{HS}\\
	&+ \sum_{j \in\G(b)} \left\| \widetilde{\F}_{\lambda_{j,b}} \right\|_{HS} \sum_{i \in\G(b)} \left\|\widetilde{\F}_{\lambda_{i,b}} - \E[ P_{\lambda_{i,b}}^{1,b} ] \right\|_{HS}\\
	&+ \sum_{i \in\G(b)} \left\| \E[ P_{\lambda_{i,b}}^{1,b} ] \right\|_{HS} \sum_{j \in\G(b)} \left\|\widetilde{\F}_{\lambda_{j,b}} - \E[ P_{\lambda_{j,b}}^{1,b} ]  \right\|_{HS}.
\end{align*}
Since by assumption $b^3/n \to 0$, it follows from Lemma \ref{Lem:PeriodogramMeanAndCovarianceApproxOverFF} that 
\begin{align*}
	\sum_{j,i \in\G(b)} \left\| \frac{1}{N} \sum_{t=1}^N P_{\lambda_{j,b}}^{t,b}\tsr P_{\lambda_{i,b}}^{t,b} 
	- \E[ P_{\lambda_{j,b}}^{1,b} \tsr P_{\lambda_{i,b}}^{1,b} ] \right\|_{HS} = o_P(b)
\end{align*}
as well as
\begin{align*}
	\sum_{j \in\G(b)} \left\|\wt{\F}_{\lambda_{j,b}}^b - \E[ P_{\lambda_{j,b}}^{1,b} ] \right\|_{HS} = o_P(1)
\end{align*}
and therefore 
\begin{align*}
	\sum_{j,i \in\G(b)} \left\| H_{i,j,n} - \Cov( P_{\lambda_{j,b}}^{1,b}, P_{\lambda_{i,b}}^{1,b} ) \right\|_{HS}
	&= o_P(b) + \mathcal{O}_P(b)o_P(1) + \mathcal{O}(b) o_P(1)\\
	&= o_P(b).
\end{align*}
The terms $R_1^*$, $R_2^*$, $R_3^*$ are bootstrap analogues of $R_1$, $R_2$, $R_3$ in \eqref{Eq:CovarianceRiemannApprox} and the above considerations 
yield 
\begin{align*}
	\| R_1^* - R_1 \|_{HS} \leq \frac{(2\pi)^2}{b} \sup_{\lambda \in [-\pi,\pi]} |W(\lambda)|^2 
	\sum_{j\in \G(b)} \| H_{j,i,n} - \Cov( P_{\lambda_{j,b}}^{1,b}, P_{\lambda_{i,b}}^{1,b} ) \|_{HS}
	= o_P(1)
\end{align*}
and similarly $\| R_2^* - R_2 \|_{HS} = o_P(1) = \|  R_3^* - R_3 \|_{HS}$. Invoking the limiting results for $R_1$, $R_2$ and $R_3$ from the proof
of Lemma \ref{Lem:AsymptoticCovarianceRiemannApprox} we obtain (with respect to $\|\cdot\|_{HS}$,
\begin{align*}
	R_1^* + R_2^* + R_3^* = R_1+R_2+R_3 + o_P(1) = \Psi + o_P(1).
\end{align*}
For the conditional relation  operator we have
\begin{align*}
	{\rm Rel}^*(L_n^+) &= \frac{(2\pi)^2 k}{b}  \sum_{j,i \in \G(b)} W(\lambda_{j,b}) W(\lambda_{i,b}) {\rm Rel}^*( P_{j,b}^*,  P_{i,b}^*)\\
	&= \frac{(2\pi)^2}{b} \sum_{j,i \in \G(b)}  W(\lambda_{j,b}) W(\lambda_{i,b}) {\rm Rel}^*(P_{\lambda_{j,b}}^{(1),b}, P_{\lambda_{i,b}}^{(1),b}  )
\end{align*}
and 
\begin{align*}
	{\rm Rel}^*(P_{\lambda_{j,b}}^{(1),b}, P_{\lambda_{i,b}}^{(1),b})
	= \frac{1}{N} \sum_{t=1}^N P_{\lambda_{j,b}}^{t,b} \tsr  P_{-\lambda_{i,b}}^{t,b} - \widetilde{\F}_{\lambda_{j,b}} \tsr \widetilde{\F}_{-\lambda_{i,b}}.
\end{align*}
Since $-\lambda_{j,b} = \lambda_{-j,b}$ all calculations can be repeated analogously to also obtain ${\rm Rel}^*(L_n^+) = \Upsilon + o_P(1)$. 
\end{proof}


\begin{proof}[Proof of Lemma~\ref{le.m-Fixed-1}]
We use the notation $ \widehat{\widetilde{\mathcal F}}_\lambda^{(m)}=\widehat{\PP}_m \widetilde{\mathcal F}_\lambda \widehat{\PP}_m$ and $ \widetilde{\mathcal F}_\lambda^{(m)}=\PP_m \widetilde{\mathcal F}_\lambda \PP_m$. Recall that $ S_{2,m}^+ = {\rm Cov}^\star(V_{n,m}^+) - \Sigma(\Psi^+_{1,m}, Y_{1,m}^+)$. As in the proof of Lemma~\ref{le.App-3}(iii), we get that, in $HS$-norm, 
\begin{align*}
{\rm Cov}^\star(T_{n,m}^+) & = {\rm Cov}^\star(\widehat{\PP}_m L_n^+ \widehat{\PP}_m) \\
 &= {\rm Cov}^\star(\PP_m L_n^+ \PP_m) +o_{P}(1)\\
 & \stackrel{P}{\rightarrow} \Psi_{1,m} + \Psi_{2,m} =:\Psi_m.
\end{align*} 
Similarly, $ \|{\rm Rel}^\star(T_{n,m}^+) - \Upsilon_m\|_{HS} \stackrel{P}{\rightarrow} 0$, where $ \Upsilon_m:=\Upsilon_{1,m} + \Upsilon_{2,m}$. Therefore,
\begin{equation}
\label{eq.Le-m-Fixed1}
 \|{\rm Cov}^\star(V_{n,m}^+)  -\Sigma(\Psi_m,\Upsilon_m)\|_{HS} \stackrel{P}{\rightarrow} 0.\end{equation}   Since 
\begin{align*}
\Psi_{1,m}^+  = \frac{4\pi^2}{b}\sum_{j\in{\mathcal G}(b)}&  \Big\{ W(\lambda_{j,b})\overline{W}(\lambda_{j,b}) \widehat{\widetilde{\F}}^{(m)}_{\lambda_{j,b}}\otimes_{op}\widehat{\widetilde{\F}}^{(m)}_{\lambda_{j,b}} \\
& \ \ \ \ 
+  W(\lambda_{j,b})\overline{W}(-\lambda_{j,b}) \widehat{\widetilde{\F}}^{(m)}_{\lambda_{j,b}}\otimes^\top_{op}\widehat{\widetilde{\F}}^{(m)}_{-\lambda_{j,b}} \Big\} + o(1/b),
\end{align*}
we can split $ \| \Psi_{1,m}^+ -\Psi_{1,m}\|_{HS}$ using
\begin{align*}
\widehat{\widetilde{\F}}^{(m)}_{\lambda_{j,b}}\otimes_{op}\widehat{\widetilde{\F}}^{(m)}_{\lambda_{j,b}} & - \F^{(m)}_{\lambda_{j,b}}\otimes_{op} \F^{(m)}_{\lambda_{j,b}} =\widehat{\widetilde{\F}}^{(m)}_{\lambda_{j,b}}\otimes_{op} \big(\widehat{\widetilde{\F}}^{(m)}_{\lambda_{j,b}} - \F^{(m)}_{\lambda_{j,b}} \big) \\
& + 
\ \  \big(\widehat{\widetilde{\F}}^{(m)}_{\lambda_{j,b}}- \F^{(m)}_{\lambda_{j,b}}\big) \otimes_{op} \F^{(m)}_{\lambda_{j,b}},
\end{align*} 
 in four  very similar terms,  with a typical one   given by 
\[D_{1,n}^+ = \Big\|\frac{4\pi^2}{b}\sum_{j\in{\mathcal G}(b)} W(\lambda_{j,b})\overline{W}(\lambda_{j,b})  \widehat{\widetilde{\F}}^{(m)}_{\lambda_{j,b}}\otimes_{op}\big(\widehat{\widetilde{\F}}^{(m)}_{\lambda_{j,b}}  -\F^{(m)}_{\lambda_{j,b}} \big)\Big\|_{HS}.\]
Now, 
\begin{align*}
\sum_{j\in {\mathcal G}(b)}\|\widehat{\widetilde{\F}}_{\lambda_{j,b}}^{(m)} - \F_{\lambda_{j,b}}^{(m)}\|_{HS} &
 \leq \sum_{j\in {\mathcal G}(b)}\|\widehat{\widetilde{\F}}_{\lambda_{j,b}}^{(m)} -\widetilde{\F}_{\lambda_{j,b}}^{(m)}\|_{HS}+  \sum_{j\in {\mathcal G}(b)}\|\widetilde{\F}_{\lambda_{j,b}}^{(m)} - \F_{\lambda_{j,b}}^{(m)}\|_{HS} 
\end{align*}
and
\begin{align*}
 \sum_{j\in {\mathcal G}(b)}\|\widehat{\widetilde{\F}}_{\lambda_{j,b}}^{(m)}  -\widetilde{\F}_{\lambda_{j,b}}^{(m)}\|_{HS} & \leq   
 \sum_{j\in {\mathcal G}(b)}\|(\widehat{\PP}_m-\PP_m)\widetilde{\F}_{\lambda_{j,b}}^{(m)} \PP_m\|_{HS} \\
 & \ \ \ \ + \sum_{j\in {\mathcal G}(b)}\|\widehat{\PP}_m\widetilde{\F}_{\lambda_{j,b}}^{(m)} (\widehat{\PP}_m-\PP_m)\|_{HS} \\
 &\leq  2 \|\widehat{\PP}_m-\PP_m\|_{\mathcal L} \sum_{j\in {\mathcal G}(b)}\|\widetilde{\F}_{\lambda_{j,b}}^{(m)} \|_{HS} \\
 &= O_{P}\big(b\|\widehat{\PP}_m-\PP_m\|_{\mathcal L} \big).
\end{align*}
Furthermore, because $\| {\rm E}(\widetilde{\F}_{\lambda_{j,b}} ) - \F_{\lambda_{j,b}} \|_{HS}= O(1/b)$, where the $O(1/b)$ term holds true uniformly in $\lambda_{j,b}$,  we get  by Lemma~\ref{Lem:PeriodogramMeanAndCovarianceApproxOverFF} (i), that 
\begin{align*}
  \sum_{j\in {\mathcal G}(b)}\|\widetilde{\F}_{\lambda_{j,b}}^{(m)} - \F_{\lambda_{j,b}}^{(m)}\|_{HS}  & = \sum_{j\in {\mathcal G}(b)}\| \PP_m(\widetilde{\F}_{\lambda_{j,b}} - \F_{\lambda_{j,b}}) \PP_m\|_{HS} \\
  & \leq \|\PP_m\|^2_{\mathcal L}  \sum_{j\in {\mathcal G}(b)} \|\widetilde{\F}_{\lambda_{j,b}} - \F_{\lambda_{j,b}}\|_{HS}\\
  & \leq  \sum_{j\in {\mathcal G}(b)} \|\widetilde{\F}_{\lambda_{j,b}} - {\rm E}(\widetilde{\F}_{\lambda_{j,b}})\|_{HS}+  \sum_{j\in {\mathcal G}(b)} \|{\rm E}(\widetilde{\F}_{\lambda_{j,b}}) - \F_{\lambda_{j,b}}\|_{HS}\\
  & = O_{P}(\sqrt{b^3/n}) + O(1).
\end{align*}
 Hence,
\begin{align*}
D_{1,n}^+ & \leq \big(\sup_{\lambda\in[-\pi,\pi]} |W(\lambda)|\big)^2 \frac{4\pi^2}{b}\sum_{j\in{\mathcal G}(b)} \|\widetilde{\F}_{\lambda_{j,b}}^{(m)} \|_{HS} 
\| \widetilde{\F}^{(m)}_{\lambda_{j,b}} - \F_{\lambda_{j,b}}^{(m)}\|_{HS}\\
&  \leq \big(\sup_{\lambda\in[-\pi,\pi]} |W(\lambda)|\big)^2 \max_{\lambda_{j,b} \in {\mathcal G}(b) }   \|\widetilde{\F}_{\lambda_{j,b}}^{(m)} \|_{HS}   \frac{4\pi^2}{b}\sum_{j\in{\mathcal G}(b)} 
\| \widetilde{\F}^{(m)}_{\lambda_{j,b}} - \F_{\lambda_{j,b}}^{(m)}\|_{HS}\\
& = O_{P}(\|\widehat{\PP}_m-\PP_m\|_{\mathcal L}) + O_{P}(\sqrt{b/n} + b^{-1})\\
 & = O_{P}\Big(\frac{1}{\sqrt{n}} \sum_{j=1}^m \frac{1}{\alpha^2_j} \Big) + O_{P}(\sqrt{b/n} + b^{-1}) \rightarrow 0,
\end{align*}
as $ n\rightarrow \infty$.  Therefore,
$ \| \Psi_{1,m}^+ -\Psi_{1,m}\|_{HS}  \stackrel{P}{\rightarrow} 0 $ and we can  show along the same lines that also, $ \| \Upsilon_{1,m}^+ -\Upsilon_{1,m}\|_{HS}  \stackrel{P}{\rightarrow} 0 $.
From this we get  $\| \Sigma(\Psi_{1,m}^+ ,  \Upsilon_{1,m}^+ ) - \Sigma(\Psi_{1,m}, \Upsilon_{1,m})\|_{HS} \stackrel{P}{\rightarrow} 0$, which  together with (\ref{eq.Le-m-Fixed1}) 
and  
\begin{align*}
S_{2,m}^+ = & \Sigma(\Psi_m,\Upsilon_m) - \Sigma(\Psi_{1,m},\Upsilon_{1,m}) + \big( \Sigma(\Psi_{1,m},\Upsilon_{1,m}) - \Sigma(\Psi_{1,m}^+,\Upsilon_{1,m}) \big) \\
& \ \ \ \  + \big({\rm Cov}^\star (V_{n,m}^+) -\Sigma(\Psi_m,\Upsilon_m\big),
\end{align*}
implies  that $ \|S_{2,m}^+ -\Sigma(\Psi_{2,m},\Upsilon_{2,m})\|_{HS} \stackrel{P}{\rightarrow} 0$. 
\hfil $\Box$
\end{proof}


\begin{proof}[Proof of Lemma~\ref{le.m-Fixed-2}]
Consider (i). We only show the convergence of the covariance operator $ {\rm Cov}^\star (L_n^o)$ since the proof for $ {\rm Rel}^\star (L_n^o)$  follows the same lines. 
Recall that $m$ is fixed and note that  $ \|(S_{1,m}^\star + S_{2,m}^+) -\Sigma(\Psi_m,\Upsilon_m)\|_{HS} \stackrel{P}{\rightarrow} 0$,
where $\Sigma(\Psi_m,\Upsilon_m)=\lim_{n\rightarrow \infty}  {\rm Cov}(V_{n,m}) $ in HS-norm and $\Psi_m=\Psi_{1,m}+\Psi_{2,m}$, $\Upsilon_m=\Upsilon_{1,m}+\Upsilon_{2,m}$; see the proof of (\ref{eq.Sigma-1n}) 
 below.  Since $\Sigma(\Psi_m,\Upsilon_m)$ is  a finite rank, nonnegative-definite operator, $S_{1,m}^\star + S_{2,m}^+$ is for $n$ large enough, with high 
probability, also nonnegative-definite,  i.e., for all $ \epsilon>0$, there exists $ n_0\in \N$ such that  for all $ n> n_0$, 
\[  P\Big(\big \langle (S_{1,m}^\star + S_{2,m}^+)(x),x \big\rangle_{HS} \geq 0 \ \mbox{for all} \  x \in {\mathcal H} \oplus\mathcal{H}, \Big) >1-\epsilon.\]
This implies that  $ (S_{1,m}^\star + S_{2,m}^+)$   possess  for $n$ large enough,  with high probability (in the sense specified above),   a  unique square root operator  denoted by $(S_{1,m}^\star + S_{2,m}^+)^{1/2}$.
Furthermore, $ S_{1,m}^\star$ is the covariance operator of $ V_{n,m}^\star$ and 
\begin{equation} \label{eq.S1m}
 \|S_{1,m}^\star -\Sigma(\Psi_{1,m},\Upsilon_{1,m})\|_{HS} \stackrel{P}{\rightarrow} 0. 
 \end{equation}
 Also, $\Sigma(\Psi_{1,m},\Upsilon_{1,m})  $ possesses  the inverse square root operator $  \Sigma(\Psi_{1,m},\Upsilon_{1,m}) ^{-1/2}$; see Lemma~\ref{le.inverse}. By the same lemma and because of (\ref{eq.S1m}),
 for $n$ large enough,  with high probability,  $( S_{1,m}^\star)^{-1/2}  $ exists and is bounded.  Let
\begin{equation} \label{eq.LoTilde}
 \widetilde{L}_n^o =   \big( S_{1,m}^\star + S_{2,m}^+\big)^{1/2}\big(S_{1,m}^\star\big)^{-1/2}V^\star_{n,m} + W^\star_{n,m},
 \end{equation}
and recall that $  V_{n,m}^\star = ({\rm  Re}\{T_{n,m}^\star\}, {\rm Im}\{T_{n,m}^\star\})^\top$ and  $  W_{n,m}^\star = ({\rm  Re}\{Q_{n,m}^\star\}, {\rm Im}\{Q_{n,m}^\star\})^\top$.   
We then get 
\begin{align*}
{\rm Cov}^\star(\widetilde{L}_n^o) &= (S_{1,m}^\star+S_{2,m}^+)^{1/2} (S_{1,m}^\star)^{-1/2} {\rm Cov}(V_{n,m}^\star,V_{n,m}^\star)  (S_{1,m}^\star)^{-1/2}  (S_{1,m}^\star+S_{2,m}^+)^{1/2}\\
 &  \ \ \ \ +(S_{1,m}^\star+S_{2,m}^+)^{1/2} (S_{1,m}^\star)^{-1/2} {\rm Cov}(V_{n,m}^\star,W_{n,m}^\star)  \\
 &  \ \ \ \  +  {\rm Cov}(W_{n,m}^\star,V_{n,m}^\star)  (S_{1,m}^\star)^{-1/2}  (S_{1,m}^\star+S_{2,m}^+)^{1/2}\\
 & \ \ \ \ +   {\rm Cov}(W_{n,m}^\star)\\
 & = \sum_{j=1}^4\Sigma_{j,n}^\star, 
\end{align*}
with an obvious notation for $ \Sigma_{j,n}^\star$.

Consider $ \Sigma^\star_{1,n}$. Note that  $ \Sigma_{1,n}^\star = (S_{1,m}^\star + S_{2,m}^+)$ because  
$ {\rm Cov}^\star (V_{n,m}^\star) = S_{1,m}^\star$. Furthermore, since  $ T_{n,m}^\star = \widehat{\PP}_m L_n^\ast \widehat{\PP}_m$, we get  in view of Lemma~\ref{le.App-CLT-1}(i)  and Lemma~\ref{le.App-3}, that
$\| {\rm Cov}(T_{n,m}^\star)  -\Psi_{1,m}\|_{HS} \stackrel{P}{\rightarrow} 0$ and $ \|{\rm Rel}(T_{n,m}^\star) - \Upsilon_{1,m}\|_{HS} \stackrel{P}{\rightarrow} 0$, or equivalently,  that
$ \|{\rm Cov}(V_{n,m}^\star) - \Sigma(\Psi_{1,m},\Upsilon_{1,m})\|_{HS} \stackrel{P}{\rightarrow} 0$.  From Lemma~\ref{le.m-Fixed-1} we have 
$ \|S_{2,m}^+ -\Sigma(\Psi_{2,m},\Upsilon_{2,m})\|_{HS} \stackrel{P}{\rightarrow} 0$. Therefore, 
\begin{align} \label{eq.Sigma-1n}
\|\Sigma_{1,n}^\star  - \Sigma(\Psi_{1,m} + \Psi_{2,m}, \Upsilon_{1,m} + \Upsilon_{2,m}  \|_{HS} & \stackrel{P}{\rightarrow} 0.
\end{align}
 
Consider $\Sigma_{2,n}^\star$.  Since  $ {\rm Cov}^\star(T_{n,m}^\star,Q_{n,m}^\star) = {\rm Cov}^{\star}(\widehat{\PP}_m L_n^\star \widehat{\PP}_m, L_n^\star -\widehat{\PP}_m L_n^\star \widehat{\PP}_m)$,
we get in view of   Lemma~\ref{le.App-CLT-1}(i)  and Lemma~\ref{le.App-3} again, that
\[ \| {\rm Cov}^\star (T^\star_{n,m}, Q^\star_{n,m})  - \Psi_{1,m}^{(F^{(m)},G^{(m)})} \|_{HS} \stackrel{P}{\rightarrow} 0,\]
and
 \[ \|{\rm Rel}^\star (T^\star_{n,m}, Q^\star_{n,m})  - \Upsilon_{1,m}^{(F^{(m)},G^{(m)})} \|_{HS} \stackrel{P}{\rightarrow} 0,\]
that is,  
\[  \| {\rm Cov}^\star(V_{n,m}^\star, W_{n,m}^\star) - \Sigma( \Psi_{1,m}^{(F^{(m)},G^{(m)})},  \Upsilon_{1,m}^{(F^{(m)},G^{(m)})} )\|_{HS} \stackrel{P}{\rightarrow} 0.\]
By Lemma~\ref{le.inverse}  and  because $ \|(S_{1,m}^\star+S_{2,m}^+)  - \Sigma(\Psi_{1,m}+\Psi_{2,m}, \Upsilon_{1,m} + \Upsilon_{2,m}) \|_{HS} \stackrel{P}{\rightarrow} 0$ and
$ \|S_{1,m}^\star - \Sigma(\Psi_{1,m}, \Upsilon_{1,m})\|_{HS} \stackrel{P}{\rightarrow} 0$, we get  $ \big\| \Sigma_{2,n}^\star  - \Sigma_{2,m}\|_{HS} \stackrel{P}{\rightarrow} 0$,
where 
\[ \Sigma_{2,m} = \Sigma(\Psi_m,\Upsilon_m)^{1/2} \Sigma(\Psi_{1,m}, \Upsilon_{1,m})^{-1/2}  \Sigma( \Psi_{1,m}^{(F^{(m)},G^{(m)})},  \Upsilon_{1,m}^{(F^{(m)},G^{(m)})} )\]
and $ \Psi_m=\Psi_{1,m} + \Psi_{2,m}$, $\Upsilon_{m} = \Upsilon_{1,m} + \Upsilon_{2,m}$.

Along the same lines it follows that $ \| \Sigma_{3,n}^\star  - \Sigma_{3,m}\|_{HS} \stackrel{P}{\rightarrow} 0$,
where 
\[ \Sigma_{3,m} =   \Sigma( \Psi_{1,m}^{(G^{(m)},F^{(m)})},  \Upsilon_{1,m}^{(G^{(m)},F^{(m)})} ) \Sigma(\Psi_{1,m}, \Upsilon_{1,m})^{-1/2}\Sigma(\Psi_m,\Upsilon_m)^{1/2}.\]

Finally, for $\Sigma_{4,n}^\star$ we have  by an analogue  reasoning, that 
\[ \|\Sigma_{4,n}^\star- \Sigma(  \Psi_{1,m}^{(G^{(m)},G^{(m)})}, \Upsilon_{1,m}^{(G^{(m)},G^{(m)})})\|_{HS} \stackrel{P}{\rightarrow} 0.\]
Putting the  derived limiting expressions for $ \Sigma_{j,n}^\star$, $j=1,\ldots,4$, together, we get after some simple algebra that 
\[ \|{\rm Cov}^\star(\widetilde{L}^o_n)-\Sigma(\Psi_1 + \Psi_{2,m} + \Delta_{\Psi,m}, \Upsilon_1+\Upsilon_{2,m} +\Delta_{\Upsilon,m} ) \|_{HS}   \stackrel{P}{\rightarrow}  0.\]
 
Consider (ii). Recall that $ L_n^o=(1, {\rm i} ) \widetilde{L}_n^o$, where $ \widetilde{L}_n^o$ is given in (\ref{eq.LoTilde}).  Furthermore, 
$ T^\star_{n,m} =\widehat{\PP}_m L_n^\star \widehat{\PP}_m=\PP_m L_n^\star \PP_m + o_{P}(1)$ and $ Q^\star_{n,m} =L_n^\star - \widehat{\PP}_m L_n^\star \widehat{\PP}_m=L_n^\star - \PP_m L_n^\star \PP_m + o_{P}(1)$,
 by Lemma~\ref{le.App-3} and Assumption~\ref{as.4}. 
Furthermore, by Lemma~\ref{le.App-CLT-1} and the mapping theorem (see \cite{Billingsley1999},
 Theorem 2.7), we get by the  continuity of the projection operator,   that   
 \[ \left( \begin{array}{c}  \PP_m L_n^\star \PP_m \\ L_n^\star - \PP_m L_n^\star \PP_m \end{array}\right)  \stackrel{\mathcal D}{\longrightarrow} {\mathcal C}{\mathcal N} (0,\Psi_{D},\Upsilon_{D}),\]
 where
 \begin{align*}
\Psi_{D} & = \left( \begin{array}{cc}  \Psi_{1,m} & \Psi_{1,m}^{(F^{(m)}, G^{(m)})} \\ \Psi_{1,m}^{(G^{(m)}, F^{(m)})} & \Psi_{1,m}^{(G^{(m)}, G^{(m)})} \end{array}\right)
 \end{align*}
 and
  \begin{align*}
\Upsilon_{D} & =   \left( \begin{array}{cc}  \Upsilon_{1,m} & \Upsilon_{1,m}^{(F^{(m)}, G^{(m)})} \\ \Upsilon_{1,m}^{(G^{(m)}, F^{(m)})} & \Upsilon_{1,m}^{(G^{(m)}, G^{(m)})} \end{array}\right).
 \end{align*}
That is, as $n\rightarrow \infty$, we have in $HS$-norm, 
\begin{equation} \label{eq.Dist-VW}
\left( \begin{array}{c}  V_{n,m}^\star \\ W_{n,m}^\star \end{array}\right)  \stackrel{\mathcal D}{\longrightarrow} {\mathcal N} (0, \Sigma(\Psi_{D},\Upsilon_{D})).
\end{equation}
 Now,
 \begin{align*}
\widetilde{L}_n^o & =    \Sigma(\Psi_m,\Upsilon_m)^{1/2} \Sigma(\Psi_{1,m},\Upsilon_{1,m})^{-1/2}V^\star_{n,m} + W^\star_{n,m} \\
 &   \ \ \ \ \ + \big(\big( S_{1,m}^\star + S_{2,m}^+\big)^{1/2}\big(S_{1,m}^\star\big)^{-1/2} -    \Sigma(\Psi_m,\Upsilon_m)^{1/2} \Sigma(\Psi_{1,m},\Upsilon_{1,m})^{-1/2}\big) V^\star_{n,m}\\
 & = \widetilde{L}_{1,n}^o +  \widetilde{L}_{2,n}^o, 
 \end{align*}
 with an obvious notation for $ \widetilde{L}_{1,n}^o$ and $ \widetilde{L}_{2,n}^o$. 
 By  (\ref{eq.Dist-VW}) and the mapping theorem again, we get 
 \begin{align*}
   \widetilde{L}_{1,n}^o=  \Sigma(\Psi_m,\Upsilon_m)^{1/2} & \Sigma(\Psi_{1,m},\Upsilon_{1,m})^{-1/2}V^\star_{n,m} + W^\star_{n,m} \\
   & \ \  \stackrel{\mathcal D}{\longrightarrow} {\mathcal N} ( 0, 
 \Psi_1 + \Psi_{2,m} + \Delta_{\Psi_1,m}, \Upsilon_1+\Upsilon_{2,m} +\Delta_{\Upsilon_1,m}).
 \end{align*}
 By (\ref{eq.Sigma-1n}),  (\ref{eq.S1m}) and  Lemma~\ref{le.inverse}, it  follows that  
 \begin{align*} 
  \|\big( S_{1,m}^\star + S_{2,m}^+\big)^{1/2}\big(S_{1,m}^\star\big)^{-1/2} -    \Sigma(\Psi_m,\Upsilon_m)^{1/2} \Sigma(\Psi_{1,m},\Upsilon_{1,m})^{-1/2}\|_{HS}  \stackrel{P}{\rightarrow} 0,
  \end{align*}
  which together with the fact that $ \|V^\star_{n,m}\|_{HS} = {\mathcal O}_P(1)$,  implies $ \widetilde{L}_{1,n}^o=o_P(1) $ and concludes the proof of  assertion (ii). 
\end{proof}


\begin{proof}[Proof of Theorem~\ref{th.main}] 
Consider (i). It suffices to show that under the assumptions of the theorem, the following assertions hold true:
\begin{enumerate}
\item[(a1)] \ $ \|{\rm Cov}^\star(T_{n,m}^+) -\Psi\|_{HS} \stackrel{P}{\rightarrow} 0$ \ and \  $ \|{\rm Rel}^\star(T_{n,m}^+) -\Upsilon\|_{HS} \stackrel{P}{\rightarrow} 0$ 
\item[(a2)] \ $\|\Psi_{1,m}^+ -\Psi_1\|_{HS} \stackrel{P}{\rightarrow} 0 $  \ and \  $\|\Upsilon_{1,m}^+ -\Upsilon_1\|_{HS} \stackrel{P}{\rightarrow} 0 $.
\end{enumerate}
We prove the assertions regarding the covariance operators only since the arguments used for the relation operators are  similar.

Consider  (a1) and recall that $T_{n,m}^+ = \widehat{\PP}_m L_n^+ \widehat{\PP}_m$.  We have
\begin{align} \label{eq.Bound-1}
\|{\rm Cov}^\star(T_{n,m}^+) -\Psi\|_{HS} & \leq \| \widehat{\PP}_m{\rm Cov}^\star( L_{n}^+) \widehat{\PP}_m - \PP_m {\rm Cov}^\star(L_{n}^+) \PP_m \|_{HS} \nonumber \\
& \ \ \ \ +  \|  \PP_m{\rm Cov}^\star( L_{n}^+) \PP_m -\PP_m\Psi\PP_m\|_{HS} \nonumber  \\
& \ \ \ \ +  \| \PP_m\Psi \PP_m-\Psi\|_{HS}.
\end{align}
The first term of the above inequality is bounded by 
\[ \|(\widehat{\PP}_m-\PP_m) {\rm Cov}^\star (L_n^+)\widehat{\PP}_m\|_{HS} +  \|\PP_m {\rm Cov}^\star(L_n^+)(\PP_m- \widehat{\PP}_m)\|_{HS}.\]
We deal with one of the two  terms only.  In view of   Lemma~\ref{Lem:CBPConsistencyCovariance}, we obtain  for  the first  term  the bound   
\begin{align*}
\|\widehat{\PP}_m -\PP_m\|_{\mathcal L} \|{\rm Cov}^\star(L_n^+)\|_{HS} \|\widehat{\PP}_m\|_{\mathcal L} & = {\mathcal O}_P( \|\widehat{\PP}_m-\PP_m\|_{\mathcal L}) \\
& = {\mathcal O}_{P}\Big(n^{-1/2}\sqrt{\sum_{j=1}^m\alpha_j^{-2}}\Big) \rightarrow 0.
\end{align*}
Furthermore, by Lemma~\ref{Lem:CBPConsistencyCovariance},
\begin{align*}
  \|  \PP_m{\rm Cov}^\star( L_{n}^+) \PP_m -\PP_m\Psi\PP_m\|_{HS} & =   \|  \PP_m( {\rm Cov}^\star( L_{n}^+) -\Psi)\PP_m\|_{HS}\\
  & \leq \|{\rm Cov}^\star( L_{n}^+) -\Psi\|_{HS} \stackrel{P}{\rightarrow} 0,
\end{align*}
 Finally  $\|\PP_m\Psi\PP_m - \Psi\|_{HS} \rightarrow 0$ as $ m\rightarrow \infty$,  \cite{GohbergEtAl1990}, Theorem VI. 4.3, which proves  (\ref{eq.Bound-1}).

Consider (a2). We only show $ \|\Psi_{1,m}^+ -\Psi_1\|_{HS} \stackrel{P}{\rightarrow} 0$.  We have
\[  \|\Psi_{1,m}^+ -\Psi_1\|_{HS}  \leq \|\Psi_{1,m}^+ -\Psi_{1,m}\|_{HS} +  \|\Psi_{1,m} -\Psi_1\|_{HS}.\]
The second term  converges to zero  as $ m\rightarrow \infty$, while for the first term we have as in the proof of Lemma~\ref{le.m-Fixed-1},  that 
\begin{align*}
 \|\Psi_{1,m}^+ -\Psi_{1,m}\|_{HS}  & = {\mathcal O}_P\Big(n^{-1/2} \sqrt{\sum_{j=1}^m 1/\alpha_j^2} \Big) + {\mathcal O}_P\big(\sqrt{b/n}  + 1/b\big)  \rightarrow 0,
 \end{align*}
 by Assumption~\ref{as.3} and  Assumption~\ref{as.4}. 

Consider (ii). We only prove the assertion  for the covariance operator. It suffices to show  that $ \|{\rm Cov}^\ast(V_{n,m}^o) - \Sigma(\Psi,\Upsilon)\|_{HS} \stackrel{P}{\rightarrow} 0$. Note that 
$ {\rm Cov}^\star(V_{n,m}^o) = (S_{1,m}^\star + S_{2,m}^+)$.  Furthermore,  
\begin{align*}
\|(S_{1,m}^\star + S_{2,m}^+) - \Sigma(\Psi,\Upsilon)\|_{HS} & \leq \|S_{1,m}^\star -\Sigma(\Psi_1,\Upsilon_1) \|_{HS} +   \|S_{2,m}^+ -\Sigma(\Psi_2,\Upsilon_2) \|_{HS}\\
& =  \|S_{1,m}^\star -\Sigma(\Psi_1,\Upsilon_1) \|_{HS} +o_P(1),
\end{align*}
by part (i) of the theorem. Also, $\|S_{1,m}^\star -\Sigma(\Psi_1,\Upsilon_1) \|_{HS} \stackrel{P}{\rightarrow} 0$ since 
\[ \| {\rm Cov}^\star(\widehat{\PP}_m L_n^\star \widehat{\PP}_m)-\Psi_1|\|_{HS} \stackrel{P}{\rightarrow }0 \ \ \mbox{and} \ \  \|{\rm Rel}^\star(\widehat{\PP}_m L_n^\star \widehat{\PP}_m)-\Upsilon_1 \|_{HS} \stackrel{P}{\rightarrow} 0.\]
Both assertions above follow by Lemma~\ref{le.App-CLT-1}  and the same arguments  as in the proof that  (\ref{eq.Bound-1}) convergence to zero in probability. 


Consider (iii). 
We use  Theorem 3.2 of \cite{Billingsley1999}. For this it suffices to show that 
\begin{enumerate}
\item[(a)] $T^o_{n,m} \stackrel{\mathcal D}{\rightarrow}  Z_m :={\mathcal C}{\mathcal N}(0, \Psi_{1,m}+\Psi_{2,m}, \Upsilon_{1,m}+\Upsilon_{2,m})$ as $ n\rightarrow \infty$ for every $ m \in \N$. \\
\item[(b)] $Z_m\stackrel{\mathcal D}{\rightarrow} {\mathcal C}{\mathcal N}(0, \Psi_1+\Psi_2, \Upsilon_1 + \Upsilon_2)$ as $ m \rightarrow \infty$.\\
\item[(c)]  $ \lim_{m\rightarrow \infty} \limsup_{n\rightarrow \infty} P(\|L_n^o-T_{n,m}^o\|_{HS} > \epsilon ) =0$ for every $ \epsilon >0$. 
\end{enumerate} 

Consider (a).  Fix any $m\in \N$. Let $ V_{n,m}^o=(S_{1,m}^\star + S_{2,m}^+)^{1/2} (S_{1,m}^\star)^{-1/2} V^\star_{n,m}$. Note that $ {\rm Cov}(V_{n,m}^o) =S_{1,m}^\star + S_{2,m}^+$. By Lemma~\ref{le.m-Fixed-1},  $\| S_{2,m}^+- \Sigma(\Psi_{2,m},\Upsilon_{2,m})\|_{HS} \stackrel{P}{\rightarrow} 0$. Furthermore,  
$\|S_{1,m}^\star- \Sigma(\Psi_{1,m}, \Upsilon_{1,m})\|_{HS} \stackrel{P}{\rightarrow} 0$. Hence, 
\[  \|(S_{1,m}^\star+S_{2,m}^+)  - \Sigma(\Psi_{1,m} + \Psi_{2,m}, \Upsilon_{1,m} + \Upsilon_{2,m})\|_{HS} \stackrel{P}{\rightarrow} 0.\]
By Lemma~\ref{le.App-3}, $ V_{n,m}^\star =\PP_m L_n^\star \PP_m +o_P(1)$. From this and  Lemma~\ref{le.App-CLT-1} we get by the projection theorem,  that   
\[T^o_{n,m} = (1,  {\rm i}) V_{n,m}^o\stackrel{\mathcal D}{\rightarrow}  {\mathcal C}{\mathcal N}(0, \Psi_{1,m}+\Psi_{2,m}, \Upsilon_{1,m}+\Upsilon_{2,m}),\]
 as $ n\rightarrow \infty$.
 
 Consider (b). By Lemma~\ref{le.App-2} it suffices to show that, as $m\rightarrow \infty$, 
 \begin{equation} \label{eq.b-1} 
 \|\Psi_m-\Psi\|_N \rightarrow 0
 \ \ \mbox{and} \ \  \|\Upsilon_m-\Upsilon\|_N \rightarrow 0.
 \end{equation}
 We show the first assertion only.   
  Note that  $ {\mathcal F}^{(m)}_\lambda = \PP_m {\mathcal F}_\lambda \PP_m = \PP_m\otimes_{op}\PP_m ({\mathcal F}_\lambda)$ and  similarly,   
  $ {\mathcal F}^{({\mathcal X}_m)}_{\lambda_1,\lambda_2, \lambda_3} = (\PP_m\otimes_{op}\PP_m) {\mathcal F}_{\lambda_1,\lambda_2,\lambda_3}  (\PP_m\otimes_{op}\PP_m)$. Hence 
 \begin{align*}
 \Psi_{m}  =&  2\pi\Big\{ \int_{-\pi}^{\pi} W(\lambda)\overline{W}(\lambda)\big[ \PP_m\otimes_{op}\PP_m({\mathcal F}_\lambda) \big]\otimes_{op}\big[ \PP_m\otimes_{op}\PP_m({\mathcal F}_\lambda) \big] d\lambda\\
 & \ \ \ \ + \int_{-\pi}^{\pi} W(\lambda)\overline{W}(-\lambda)\big[ \PP_m\otimes_{op}\PP_m({\mathcal F}_\lambda) \big]\otimes^{\top}_{op}\big[ \PP_m\otimes_{op}\PP_m({\mathcal F}_{-\lambda}) \big] d\lambda\\
 & \ \ \ \  + \int_{-\pi}^{\pi} \int_{-\pi}^\pi W(\lambda)_1\overline{W}(\lambda_2)\big( \PP_m\otimes_{op}\PP_m\big) {\mathcal F}_{\lambda_1,-\lambda_1, -\lambda_2}\big(\PP_m\otimes_{op}\PP_m\big) d\lambda_1d\lambda_2 \Big\}\\
 =& \big( \PP_m\otimes_{op}\PP_m\big)\Psi \big( \PP_m\otimes_{op}\PP_m\big), 
 \end{align*}
 where the last equality follows using Lemma~\ref{le.OperAlg}. Because for an operator $ A\in HS(\HH)$ it holds  $ \| \PP_m\otimes_{op}\PP_m(A) -A\|_{HS}^2=\sum_{r,s>m}|\langle A, v_r\otimes v_s\rangle_{HS}\big|^2 \rightarrow 0$ as $ m\rightarrow \infty$, it follows from Theorem VI. 4.3 of \cite{GohbergEtAl1990} that 
 \[ \big\|\big( \PP_m\otimes_{op}\PP_m\big)\Psi \big( \PP_m\otimes_{op}\PP_m\big) -\Psi\big\|_{N} \rightarrow 0,\]
 as $ m \rightarrow \infty$.

 Consider (c).  By Markov's  inequality, 
 \begin{align*}
 \lim_{m\rightarrow \infty}\limsup_{n\rightarrow \infty}P\big( \| L_n^o-T_{n,m}^o\|_{HS} >\epsilon \big) \leq \frac{1}{\epsilon^2} 
  \lim_{m\rightarrow \infty}\limsup_{n\rightarrow \infty} \E^\star \|Q^\star_{n,m}\|_{HS}^2
 \end{align*}
 and
 \begin{align*}
 \E^\star \|Q^\star_{n,m}\|_{HS}^2 & = \E^\ast\|  \widehat{\PP}_m L_n^\star ({\rm Id}-\widehat{\PP}_m)  +   ({\rm Id}-\widehat{\PP}_m) L_n^\star \widehat{\PP}_m +  ({\rm Id}- \widehat{\PP}_m) L_n^\star ({\rm Id}-\widehat{\PP}_m)\|_{HS}^2\\
 & \leq {\mathcal O}(1) \|{\rm Id}-\widehat{\PP}_m\|^2_{\mathcal L} \E^\star \|L_n^\star\|_{HS}^2 \\
 &\leq {\mathcal O}(1)\big(\| {\rm Id}-\PP_m\|^2_{\mathcal L} + \|\widehat{\PP}_m -\PP_m\|^2_{\mathcal L}\big)\E^\star \|L_n^\star\|_{HS}^2\\
 & \leq  {\mathcal O}(1)\big(\| {\rm Id}-\PP_m\|^2_{\mathcal L} + n^{-1/2}\sum_{j=1}^m \alpha_j^{-2} \big)\E^\star \|L_n^\star\|_{HS}^2
 \end{align*}
 Note that 
 \[ \E^\star\|L_n^\star\|_{HS}^2   = \|{\rm Cov}^\star(L_n^\star) \|_N  \leq  \|\Psi_1\|_N + \|{\rm Cov}^\star(L_n^\star)  - \Psi_1\|_N \rightarrow \|\Psi_1\|_N \]
in probability,   by Lemma~\ref{le.App-CLT-1} and  therefore,  by Assumption~\ref{as.4}, 
\[
\limsup_{n\rightarrow \infty}  \Big(\| {\rm Id}-\PP_m\|^2_{\mathcal L} + n^{-1/2}\sqrt{\sum_{j=1}^m1/ \alpha_j^{2} }\Big)\E^\star \|L_n^\star\|_{HS}^2 = {\mathcal O}_P(1) \|{\rm Id}-\PP_m\|_{\mathcal L}^2,
\]
which converges to zero as $ m \rightarrow\infty$.
\end{proof}

\begin{thebibliography}{50}
\bibitem{AueKlepsch2017}  
\textsc{Aue, A.} and \textsc{Klepsch, J.} (2017). 
{Estimating functional time series by moving average model fitting.} 
\textit{arXiv:1701.00770}.

\bibitem{Bogachev1998}
\textsc{Bogachev, V.} (1998).
\textit{Gaussian Measures.}
American Mathematical Society, Providence.

\bibitem{Billingsley1999}  
\textsc{Billingsley, P.} (1999).
\textit{Convergence of Probability Measures}, 2nd ed. 
Wiley, New York.

\bibitem{CerovekiHoermann2015}  
\textsc{Cerovecki, C.} and \textsc{H\"ormann, S.} (2015). 
On the CLT for discrete Fourier transforms of functional time series.  
\textit{Journal of Multivariate Analysis} \textbf{154} 282--295. 

\bibitem{DahlhausJanas1996} 
\textsc{Dahlhaus, R.} and \textsc{Janas, D.}  (1996).
A frequency domain bootstrap for ratio statistics in time series analysis.
\textit{The Annals of  Statistics} \textbf{24} 1934--1963.

\bibitem{Dehling2015} 
\textsc{Dehling, H.}, \textsc{Shapirov, O. S.} and \textsc{Wendler, M.} (2015). 
Bootstrap for dependent Hilbert space valued random variables with application to von-Mises statistics. 
\textit{Journal of Multivariate Analysis}  \textbf{133}  200--215. 

\bibitem{FrankeHardle1992} 
\textsc{Franke, J.} and \textsc{H\"ardle, W.} (1992).
On bootstrapping kernel spectral estimates.
\textit{The Annals of  Statistics} \textbf{20} 121--145.

\bibitem{FrankeNuar2016} 
\textsc{Franke, J.} and  \textsc{Nyarigue, E.} (2016). 
Residual-based bootstrap for functional autoregressions. 
\textit{Preprint}.

\bibitem{GohbergEtAl1990}
\textsc{Gohberg, I.},  \textsc{Goldberg, S.} and  \textsc{Kaashoek, M. A.} 
\textit{Classes of Linear Operators Vol. 1},
Springer, Basel,
1990.



\bibitem{HorvathKokoszka2012} 
\textsc{Horv\`ath, L.} and \textsc{Kokoszka, P.}  (2012). 
\textit{Inference for Functional Data with Applications.}  
Springer, New York.

\bibitem{HoermannKokoszka2010} 
\textsc{H\"ormann, S.} and \textsc{Kokoszka, P.}  (2010). 
Weakly dependent functional data. 
\textit{The Annals of Statistics} \textbf{38} 1845--1884.

\bibitem{HoermannKidzinski2015}  
\textsc{H\"ormann, S.}  and \textsc{Kidzi\'nski, L.} (2015). 
A note on estimation in Hilbertian linear models. 
\textit{Scandinavian Journal of Statistics} \textbf{42} 43--62. 

\bibitem{HsingEubank20152017} 
\textsc{Hsing, T.} and \textsc{Eubank, R.} (2015). 
\textit{Theoretical Foundations of Functional Data Analysis, with an Introduction to Linear Operators.}
Wiley, Chichester.   

\bibitem{HurvichZeger1987}  
\textsc{Hurvich, C. M.} and \textsc{Zeger, S. L.} (1987),
Frequency domain bootstrap methods for time series.
\textit{Technical Report \#87-14,  New York University.}

\bibitem{LeuchtEtAl2023} 
\textsc{Leucht, A.}, \textsc{Paparoditis, E.}, \textsc{Rademacher, D.} and  \textsc{Sapatinas, T.}  (2022). 
Testing equality of  spectral density operators for  functional processes. 
\textit{Journal of Multivariate Analysis} \textbf{189} 104889.

\bibitem{LiWangCarolli2013}  
\textsc{Li, Y.}, \textsc{Wang, N.}  and \textsc{Caroll, R. J.}  (2013). 
Selecting the number of principal components in functional data. 
\textit{Journal of the American Statistical Association} \textbf{108} 1284--1294.

\bibitem{MeyerEtAl2020} 
\textsc{Meyer,  M.},  \textsc{Kreiss, J.-P.} and \textsc{Paparoditis, E.} (2020). 
Extending the validity of frequency domain bootstrap methods to general stationary processes. 
\textit{The Annals of  Statistics} \textbf{48} 2404--2427.

\bibitem{MeyerPaparoditis2023}
\textsc{Meyer, M.} and \textsc{Paparoditis, E.} (2023). 
A frequency domain bootstrap for general multivariate stationary processes. 
\textit{Bernoulli} \textbf{29} 2367--2391.

\bibitem{PanaretosTavakoli2013} 
\textsc{Panaretos, V. and S. Tavakoli} (2013). 
Fourier analysis of stationary time series in function space. 
\textit{The Annals of Statistics} \textbf{41}  568-603.


\bibitem{Paparoditis2018} 
\textsc{Paparoditis, E.} (2018). 
Sieve bootstrap for functional time series. 
\textit{The Annals of Statistics} \textbf{46}  3510--3538.

\bibitem{PilavakisEtAl2019} 
\textsc{Pilavakis, D.},  \textsc{Paparoditis, E.} and \textsc{Sapatinas, T.} (2019). 
Moving block and  tapered block bootstrap for functional time series with an application to the K-sample mean problem. 
\textit{Bernoulli} \textbf{25} 3496--3526.

\bibitem{PilavakisEtAl2020} 
\textsc{Pilavakis, D.}, \textsc{Paparoditis, E.} and \textsc{Sapatinas, T.} (2020). 
Testing equality of covariance operators for functional time series. 
\textit{Journal of Time Series Analysis} \textbf{41} 571--589.

\bibitem{Politis1994}  
\textsc{Politis, D. N.} and \textsc{Romano, J.} (1994). 
Limit theorems for weakly dependent Hilbert space valued random variables with applications to the stationary bootstrap. 
\textit{Statistica Sinica} \textbf{4} 461--476.

\bibitem{RademacherEtAl2024}  
\textsc{Rademacher, D.}, \textsc{Kreiss, J.-P.} and \textsc{E. Paparoditis} (2024). 
Asymptotic normality  of spectral means of Hilbert space valued random processes. 
\textit{Stochastic Processes and their Applications} \textbf{173} 104357.





\bibitem{Shapirov2016} 
\textsc{Shapirov, O. S.}, \textsc{Tewes, J.} and \textsc{Wendler, M.}  (2016). 
Sequential block bootstrap in a Hilbert space with application to change point analysis. 
\textit{The Canadian Journal of Statistics} \textbf{44} 300--322. 

\bibitem{Simons2005} 
\textsc{Simons, B.}  (2005).  
\textit{Trace Ideals and their Applications.}
American Mathematical Society, Mathematical Surveys and Monographs, Vol. 120. 

\bibitem{YuEtAl2023}  
\textsc{Yu, H.}, \textsc{Kaiser, M. S.} and \textsc{Nordman, D. J.} (2023). 
A subsampling perspective for extending the validity of state-of-the-art bootstraps in the frequency domain. 
\textit{Biomatrika} \textbf{110}  1099-1115.


\bibitem{ZhuPolitis2017}  
\textsc{Zhu, T.} and \textsc{Politis, D. N.} (2017). 
Kernel estimates of first-order nonparametric functional autoregression model and its bootstrap approximation
\textit{Electronic Journal of Statistics} \textbf{11} 2876--2906.

\end{thebibliography}
\end{document}